\documentclass[a4paper]{amsart}
\usepackage[T1]{fontenc}
\usepackage{amstext}
\usepackage{amsthm}
\usepackage{amssymb}
\usepackage{stackrel}
\usepackage[pdfusetitle,
 bookmarks=true,bookmarksnumbered=false,bookmarksopen=false,
 breaklinks=false,pdfborder={0 0 0},pdfborderstyle={},backref=false,colorlinks=false]
 {hyperref}

\makeatletter

\numberwithin{equation}{section}
\numberwithin{figure}{section}

\usepackage{tikz-cd}
\usepackage{mathtools}
\usepackage{adjustbox}
\usepackage{enumitem}
\tikzcdset{scale cd/.style={every label/.append style={scale=#1},
    cells={nodes={scale=#1}}}}
\usepackage{quiver}
\usepackage{xcolor}
\usepackage{eucal}
\usepackage{tikz}
\usetikzlibrary{arrows.meta, positioning, shapes, fit}
\usepackage{mathrsfs}
\newtheorem{innercustomthm}{Theorem}
\newenvironment{customthm}[1]
  {\renewcommand\theinnercustomthm{#1}\innercustomthm}
  {\endinnercustomthm}

\makeatother

\theoremstyle{plain}
\newtheorem{thm}{\protect\theoremname}[section]
\newtheorem{prop}[thm]{\protect\propositionname}
\theoremstyle{remark}
\newtheorem{rem}[thm]{\protect\remarkname}
\theoremstyle{definition}
\newtheorem{defn}[thm]{\protect\definitionname}
\newtheorem{notation}[thm]{\protect\notationname}
\newtheorem{construction}[thm]{\protect\constructionname}
\theoremstyle{plain}
\newtheorem{cor}[thm]{\protect\corollaryname}
\newtheorem{lem}[thm]{\protect\lemmaname}
\theoremstyle{definition}
\newtheorem{recollection}[thm]{\protect\recollectionname}
\newtheorem{example}[thm]{\protect\examplename}
\newtheorem{variant}[thm]{\protect\variantname}
\newtheorem{convention}[thm]{\protect\conventionname}

\providecommand{\constructionname}{Construction}
\providecommand{\conventionname}{Convention}
\providecommand{\corollaryname}{Corollary}
\providecommand{\definitionname}{Definition}
\providecommand{\examplename}{Example}
\providecommand{\lemmaname}{Lemma}
\providecommand{\notationname}{Notation}
\providecommand{\propositionname}{Proposition}
\providecommand{\recollectionname}{Recollection}
\providecommand{\remarkname}{Remark}
\providecommand{\theoremname}{Theorem}
\providecommand{\variantname}{Variant}

\begin{document}
\global\long\def\u#1{\underline{#1}}%

\global\long\def\cat#1{\mathcal{#1}}%

\global\long\def\tild#1{\widetilde{#1}}%

\global\long\def\mrm#1{\mathrm{#1}}%

\global\long\def\pr#1{\left(#1\right)}%

\global\long\def\abs#1{\left|#1\right|}%

\global\long\def\inp#1{\left\langle #1\right\rangle }%

\global\long\def\br#1{\left\{  #1\right\}  }%

\global\long\def\norm#1{\left\Vert #1\right\Vert }%

\global\long\def\hat#1{\widehat{#1}}%

\global\long\def\opn#1{\operatorname{#1}}%

\global\long\def\bigmid{\,\middle|\,}%

\global\long\def\Top{\mathsf{Top}}%

\global\long\def\Set{\mathsf{Set}}%

\global\long\def\SS{\mathsf{sSet}}%

\global\long\def\Kan{\mathsf{Kan}}%

\global\long\def\FB{\mathsf{FB}}%

\global\long\def\Fin{\mathsf{Fin}}%

\global\long\def\urCard{\mathsf{urCard}}%

\global\long\def\LPC{\mathsf{LocPresCat}}%

\global\long\def\RelCat{\mathsf{RelCat}}%

\global\long\def\Rel{\mathcal{R}\mathsf{el}}%

\global\long\def\PermRelCat{\mathsf{PermRelCat}}%

\global\long\def\RelCatlarge{\mathsf{Rel}\hat{\mathsf{Cat}}}%

\global\long\def\SMRelCat{\mathsf{SMRelCat}}%

\global\long\def\SMRelCatlarge{\mathsf{SMRel}\widehat{\mathsf{Cat}}}%

\global\long\def\CMMC{\mathsf{CombMMC}}%

\global\long\def\CSMMC{\mathsf{CombSMMC}}%

\global\long\def\TSMMC{\mathsf{TractSMMC}}%

\global\long\def\CombSymAlg#1{\mathsf{CombSym}#1\text{-}\mathsf{Alg}}%

\global\long\def\CombCentAlg#1{\mathsf{CombCent}#1\text{-}\mathsf{Alg}}%

\global\long\def\BiCat{\mathcal{B}\mathsf{i}\mathcal{C}\mathsf{at}}%

\global\long\def\Cat{\mathcal{C}\mathsf{at}}%

\global\long\def\SM{\mathcal{SM}}%

\global\long\def\Mon{\mathcal{M}\mathsf{on}}%

\global\long\def\LFib{\mathcal{L}\mathsf{Fib}}%

\global\long\def\Cart{\mathcal{C}\mathsf{art}}%

\global\long\def\Pr{\mathcal{P}\mathsf{r}}%

\global\long\def\Del{\mathbf{\Delta}}%

\global\long\def\Sig{\mathbf{\Sigma}}%

\global\long\def\id{\operatorname{id}}%

\global\long\def\Aut{\operatorname{Aut}}%

\global\long\def\End{\operatorname{End}}%

\global\long\def\Hom{\operatorname{Hom}}%

\global\long\def\Env{\operatorname{Env}}%

\global\long\def\Map{\operatorname{Map}}%

\global\long\def\Und{\operatorname{Und}}%

\global\long\def\Ar{\operatorname{Ar}}%

\global\long\def\cEnd{\mathcal{E}\mathrm{nd}}%

\global\long\def\ho{\operatorname{ho}}%

\global\long\def\ob{\operatorname{ob}}%

\global\long\def\-{\text{-}}%

\global\long\def\rep{\mathrm{rep}}%

\global\long\def\op{\mathrm{op}}%

\global\long\def\bi{\mathrm{bi}}%

\global\long\def\sc{\mathrm{sc}}%

\global\long\def\cc{\mathrm{cc}}%

\global\long\def\gr{\mathrm{gr}}%

\global\long\def\act{\mathrm{act}}%

\global\long\def\cof{\mathrm{cof}}%

\global\long\def\fib{\mathrm{fib}}%

\global\long\def\BK{\mathrm{BK}}%

\global\long\def\Seg{\mathrm{Seg}}%

\global\long\def\semi{\mathrm{semi}}%

\global\long\def\weq{\mathrm{weq}}%

\global\long\def\opd{\mathrm{opd}}%

\global\long\def\idem{\mathrm{idem}}%

\global\long\def\rel{\mathrm{rel}}%

\global\long\def\Quilleq{\mathrm{Quill.eq}}%

\global\long\def\loceq{\mathrm{loc.eq}}%

\global\long\def\To{\Rightarrow}%

\global\long\def\rr{\rightrightarrows}%

\global\long\def\rl{\rightleftarrows}%

\global\long\def\mono{\rightarrowtail}%

\global\long\def\epi{\twoheadrightarrow}%

\global\long\def\comma{\downarrow}%

\global\long\def\ot{\leftarrow}%

\global\long\def\from{\colon}%

\global\long\def\corr{\leftrightsquigarrow}%

\global\long\def\lim{\operatorname{lim}}%

\global\long\def\colim{\operatorname{colim}}%

\global\long\def\holim{\operatorname{holim}}%

\global\long\def\hocolim{\operatorname{hocolim}}%

\global\long\def\Ran{\operatorname{Ran}}%

\global\long\def\Lan{\operatorname{Lan}}%

\global\long\def\Ind{\operatorname{Ind}}%

\global\long\def\Fun{\operatorname{Fun}}%

\global\long\def\Fact{\operatorname{Fact}}%

\global\long\def\Perm{\operatorname{Perm}}%

\global\long\def\Alg{\operatorname{Alg}}%

\global\long\def\CAlg{\operatorname{CAlg}}%

\global\long\def\Coll{\operatorname{Coll}}%

\global\long\def\MInd{\mathbf{Ind}}%

\global\long\def\SSeq{\Sigma\mathrm{Seq}}%

\global\long\def\MSSeq{\mathbf{\Sigma Seq}}%

\global\long\def\B{\mathrm{B}}%

\global\long\def\H{\mathrm{H}}%

\global\long\def\Day{\mathrm{Day}}%

\global\long\def\xmono#1#2{\stackrel[#2]{#1}{\rightarrowtail}}%

\global\long\def\xepi#1#2{\stackrel[#2]{#1}{\twoheadrightarrow}}%

\global\long\def\adj{\stackrel[\longleftarrow]{\longrightarrow}{\bot}}%

\global\long\def\w{\wedge}%

\global\long\def\t{\otimes}%

\global\long\def\bt{\boxtimes}%

\global\long\def\ev{\operatorname{ev}}%

\global\long\def\bp{\boxplus}%

\global\long\def\rcone{\triangleright}%

\global\long\def\lcone{\triangleleft}%

\global\long\def\teq{\trianglelefteq}%

\global\long\def\ll{\vartriangleleft}%

\global\long\def\S{\mathsection}%

\global\long\def\p{\prime}%

\global\long\def\pp{\prime\prime}%

\global\long\def\sto{\rightsquigarrow}%

\vspace{-1.5cm}

\title[Equivalence of two approaches to enriched $\infty$-operads]{On the equivalence of Brantner's and Chu--Haugseng's approaches to enriched $\infty$-operads}
\begin{abstract}
We prove that two models of (monochromatic) enriched $\infty$-operads,
due to Brantner and Chu--Haugseng, are equivalent. We show this as
a consequence of the equivalence of two models of monoidal $\infty$-categories
of symmetric sequences and the composition product, due to Brantner
and Haugseng. As a consequence, constructions and results formulated
in either framework, such as notions of algebra and Koszul duality,
are also shown to be equivalent.
\end{abstract}

\author{Kensuke Arakawa}
\email{arakawa.kensuke.22c@st.kyoto-u.ac.jp}
\address{Department of Mathematics, Kyoto University, Kyoto, 606-8502, Japan}

\maketitle
\vspace{-1cm}

\section*{Introduction}

Enriched operads provide a unifying language for algebraic structures
across topology, geometry, algebra, and physics. But their classical
theory is sometimes too rigid to support homotopical applications.
For example, a meaningful homotopy theory of algebras is available
only under strong cofibrancy assumptions on the operad (such as $\Sigma$-cofibrancy).
This motivates the study of \textit{enriched $\infty$-operads}, which
describe algebraic structures in symmetric monoidal $\infty$-categories
and provide a more flexible framework.

But this framework also comes with a problem. A central issue is that
there are many seemingly non-equivalent definitions of enriched $\infty$-operads,
each having its own strengths. Despite substantial progress in the
work of Chu--Haugseng \cite{CH20}, a full comparison between existing
models remains incomplete. 

The goal of this paper is to establish an equivalence of two models
of (monochromatic) enriched $\infty$-operads. This appears to be
the last missing equivalence in the literature.

The models we will compare are those of Chu--Haugseng \cite{CH20}
and Brantner \cite{Brantner_PhD}. Our main result in fact proves
more than a mere comparison of the two models of enriched $\infty$-operads.
To explain this, we recall a standard perspective on operads. Namely,
monochromatic operads in a symmetric monoidal category can be described
as an algebra object in the monoidal category of \textit{symmetric
sequences }and the\textit{ composition product}.\textit{ }Brantner
and Haugseng constructed two seemingly different $\infty$-categorical
refinements of the composition product monoidal structure \cite{Brantner_PhD,Haug22}.
More precisely, to each presentably symmetric monoidal $\infty$-category
$\cat C$, they associated monoidal $\infty$-categories $\SSeq_{\B}\pr{\cat C}$
and $\SSeq_{\H}\pr{\cat C}$. The algebra objects in these monoidal
$\infty$-categories recover Brantner's and Chu--Haugseng's enriched
$\infty$-operads, respectively. Our main theorem then asserts that
these monoidal $\infty$-categories are equivalent:

\begin{customthm}{A}[Theorem \ref{thm:main}]\label{thm:main_intro}

There is a natural equivalence of monoidal $\infty$-categories
\[
\SSeq_{\B}\pr{\cat C}\simeq\SSeq_{\H}\pr{\cat C}.
\]

\end{customthm}

As an immediate consequence, we find that Brantner and Chu--Haugseng's
definitions of monochromatic enriched $\infty$-operads are equivalent.
Moreover, the monoidal equivalence ensures that the two definitions
of Koszul duality of enriched $\infty$-operads and algebras over
enriched $\infty$-operads, based on the two models, are also equivalent
to each other. To our knowledge, this comparison also has not appeared
before.

We note that the two definitions are based on fundamentally different
constructions. Because of this, there is no evident common framework
in which to relate them, explaining why the comparison has not appeared
before.

Perhaps because of their difference, the two approaches also have
complementary strengths. On the one hand, Chu--Haugseng's approach
is well-suited to developing abstract theories of enriched $\infty$-operads,
and a substantial body of such theory has already been developed with
their approach. For example, it is compatible with the classical theory
of $\infty$-operads enriched in spaces, or more generally, nice model
categories \cite[Theorem 5.2.10]{CH20}; algebras over enriched $\infty$-operads
in this sense are also known to be well-behaved and are compatible
with their strict counterpart \cite[Theorem 4.10]{Haug19}; and they
admit a characterization in terms of maps into endomorphism $\infty$-operads
\cite{Haug19}. However, their approach is technically involved and
has seen limited use in applications. On the other hand, Brantner's
approach is comparatively lightweight and has been used extensively
in the literature \cite{FG12,Knu18,araminta22,BCN23,Heuts24,ABH25, heuts_land_formality, LPY26}.
However, it is less suited to develop abstract theories of enriched
$\infty$-operads with his approach. For example, it is not apparent
from the definition that his construction is functorial in the enriching
symmetric monoidal $\infty$-categories! (See Table \ref{tab:A-quick-comparison}
for a summary.) 

This contrast highlights the need for a comparison, which Theorem
\ref{thm:main_intro} provides.

\begin{table}
\begin{tabular}{|c|c|c|}
\hline 
 & Developing abstract theories & Applications\\
\hline 
\hline
Chu--Haugseng & Well-suited & Less commonly used\\
\hline 
Brantner & Less developed & Widely used\\
\hline 
\end{tabular}
\caption{\label{tab:A-quick-comparison}A quick comparison} \end{table}

Our strategy for Theorem \ref{thm:main_intro} is to ``de-localize''
each monoidal $\infty$-category in the theorem to a model category.
That this will produce a natural equivalence is ensured by the main
theorem of \cite{arakawa_pres_smcat_comb_smcat}, which asserts that
the homotopy theory of combinatorial symmetric monoidal model categories
is equivalent to that of presentably symmetric monoidal $\infty$-categories.
At present, we do not know a direct $\infty$-categorical proof of
Theorem \ref{thm:main_intro}. This suggests the genuine subtlety
of the problem and that it is not a consequence of formal arguments.
We also remark that, while the underlying idea is simple, its implementation
requires substantial work, as reflected in the length of the paper.

\subsection*{Organization of the paper}

This paper consists of 6 sections and an appendix.

In Section \ref{sec:statement}, we review the definitions of Brantner
and Haugseng's composition product monoidal structure, and then state
our main theorem precisely. In Section \ref{sec:review_on_loc}, we
recall the definition of monoidal localization and a few related results.
In Section \ref{sec:comp_prod_mod_cat}, we establish a few key results
on Day's convolution product and composition product from a model-categorical
perspective. In Sections \ref{sec:brantner} and \ref{sec:haugseng},
we show that both Brantner and Haugseng's models are characterized
by a certain universal property involving symmetric monoidal model
categories and their localization. In Section \ref{sec:proof}, we
give the proof of the main theorem.

Appendix \ref{sec:oo2} contains a brief summary of results on $\infty$-bicategories
we need in the main body.

\subsection*{Notation and convention}
\begin{itemize}
\item Throughout the paper, we use the word \textbf{$\infty$-category}
as a synonym for quasicategory in the sense of \cite{Joyal_qcat_Kan}.
We will mainly follow \cite{HTT} and \cite{HA} in various terminology
and notation related to $\infty$-categories.
\item We will generally not notationally distinguish between categories
and their nerves, or bicategories with their Duskin nerves (Example
\ref{exa:Duskin}). 
\item If $\cat C$ is an $\infty$-category, we denote its maximal sub Kan
complex by $\cat C^{\simeq}$ and refer to it as the \textbf{core}
of $\cat C$.
\item We write $\Fin$ for the skeleton of the category of finite sets and
set maps. Explicitly, its objects are the sets $\underline{n}=\{1,\dots,n\}$
for $n\geq0$. Various set-theoretical operations will be replaced
by categorical operations in this category. For example, if $f\from S\to T$
is a map in $\Fin$ and $t\in T$, then we write $f^{-1}\pr t\in\Fin$
for the category-theoretic (not set-theoretic) fiber of $f$ over
$t$.
\item We write $\Fin_{\ast}$ for the skeleton of the category of finite
pointed sets and set maps. Explicitly, its objects are $\inp n=\pr{\{\ast,1,\dots,n\},\ast}$
for $n\geq0$.
\item We write $\FB$ for the maximal subgroupoid of $\Fin$.
\item Following \cite{HA}, we use symbols like $\mathcal{C}^{\t}\to\Fin_{\ast}$
to denote a symmetric monoidal $\infty$-category with underlying
$\infty$-category $\mathcal{C}$. We typically denote the unit object
of $\cat C^{\t}$ by $\mathbf{1}_{\cat C}=\mathbf{1}$. We typically
implicitly identify $\Fin_{\ast}$ with the larger category of all
finite pointed sets and pointed maps by choosing an inverse equivalence.
\item If $\cat C^{\t}$ and $\cat D^{\t}$ are symmetric monoidal $\infty$-categories,
we write $\Fun^{\t}\pr{\cat C,\cat D}$ for the $\infty$-category
of symmetric monoidal functors $\cat C^{\t}\to\cat D^{\t}$ \cite[Definition 2.1.3.7]{HA}.
Similar notation will be used for monoidal $\infty$-categories.
\item We write $\Mon\Cat_{\infty}$ for the $\infty$-category of small
monoidal $\infty$-categories. Formally, it is defined as the homotopy
coherent nerve of the simplicial category whose objects are the small
monoidal $\infty$-categories, and whose mapping simplicial sets are
given by $\Fun^{\t}\pr{-,-}^{\simeq}$. (Equivalently, it is the $\infty$-categorical
localization of the ordinary category of monoidal $\infty$-categories
and monoidal functors at equivalences of monoidal $\infty$-categories.)
We define the $\infty$-category of \textit{large} monoidal $\infty$-categories
$\Mon\hat{\Cat}_{\infty}$ similarly.
\item If $\pr{\cat M,\t,\mathbf{1}}$ is a symmetric monoidal category in
the ordinary sense, we write $\cat M^{\t}\to\Fin_{\ast}$ for the
associated symmetric monoidal $\infty$-category \cite[Construction 2.0.0.1]{HA}.
\end{itemize}

\section{\label{sec:statement}Stating the main result}

Let $\cat C$ be a cocomplete closed symmetric monoidal category.
A \textbf{symmetric sequence} in $\cat C$ is a functor $X\from\FB\to\cat C$.
Such a functor consists of a sequence $\pr{X\pr 0,X\pr 1,\dots}$
of objects in $\cat C$, where each $X\pr n$ carries a $\Sigma_{n}$-action,
hence the name. Given two symmetric sequences $X,Y$ in $\cat C$,
their \textbf{composition product} $X\circ Y$ is defined by the formula\footnote{Classically, the composition product is written in the reverse order.
In other words, what we just defined as $X\circ Y$ is often denoted
by $Y\circ X$. We follow \cite{Haug22} in our convention here.}
\begin{align*}
X\circ Y\pr A & =\colim_{f\in A\to B\in\FB\times_{\Fin}\Fin_{A/}}\pr{\bigotimes_{b\in B}X\pr{f^{-1}\pr b}}\otimes Y\pr B.
\end{align*}
The category $\SSeq\pr{\cat C}=\Fun\pr{\FB,\cat C}$ of symmetric
sequences is a monoidal category with respect to $\circ$ and the
unit symmetric sequence $\pr{\emptyset,\mathbf{1},\emptyset,\emptyset,\dots}$.
Monoid objects in this monoidal category are exactly (monochromatic)
operads enriched in $\cat C$.

In this section, we review the $\infty$-categorical enhancements
of the composition product monoidal structure, due to Brantner and
Chu--Haugseng. We then state the main result of this paper, which
compares the two monoidal structures.

\subsection{\label{subsec:Bra}Brantner's model}

Branter's model of composition product is based on the following observation,
found in Trimble's note \cite{Tri_Notes_Lie} and ``Author's Note''
of \cite{MR2177746} (there attributed to Carboni): The groupoid $\FB$
carries a symmetric monoidal structure, given by disjoint union of
sets. This (and its opposite) is the free symmetric monoidal category
on a singleton $\underline{1}$. It follows that the symmetric monoidal
category $\Fun\pr{\FB,\Set}$ of symmetric sequences with the and
Day's convolution product $\bigstar$ is the free cocomplete closed
symmetric monoidal category generated by the unit symmetric sequence
$\mathfrak{X}=\FB\pr{\underline{1},-}$ \cite[Theorem 5.1]{GK86}.
In analogy with ordinary algebra, let us therefore write $\Set[\mathfrak{X}]=\Fun\pr{\FB,\Set}$
for this symmetric monoidal category. The universal property implies
that the category $\Fun^{\t,L}\pr{\Set[\mathfrak{X}],\Set[\mathfrak{X}]}$
of cocontinuous symmetric monoidal functors $\Set[\mathfrak{X}]\to\Set[\mathfrak{X}]$
is equivalent $\Set[\mathfrak{X}]=\SSeq\pr{\Set}$, via evaluation
at $\mathfrak{X}\in\Set[\mathfrak{X}]$. We then have:
\begin{prop}
\label{prop:carboni}The categorical equivalence
\[
\ev_{\mathfrak{X}}\from\Fun^{\t,L}\pr{\Set[\mathfrak{X}],\Set[\mathfrak{X}]}\xrightarrow{\simeq}\SSeq\pr{\Set}
\]
can be enhanced to a monoidal equivalence, where the left-hand carries
the monoidal structure given by composition of symmetric monoidal
functors, and the right-hand side carries the composition product
monoidal structure.
\end{prop}

\begin{proof}
Since the Yoneda embedding $\FB^{\op}\to\Set[\mathfrak{X}]$ is symmetric
monoidal, we have $\FB\pr{S,-}\cong\mathfrak{X}^{\bigstar S}$ for
every finite set $S\in\FB$. Therefore, the co-Yoneda lemma gives
us an isomorphism
\[
F\cong\int^{S\in\FB}\mathfrak{X}^{\bigstar S}\cdot F\pr S
\]
natural in $F\in\Set[\mathfrak{X}]$. (A more suggestive notation
for the right-hand side will be $\sum_{n\geq0}F\pr{\underline{n}}\times_{\Sigma_{n}}\mathfrak{X}^{\bt n}$.)
Thus, if $\overline{F},\overline{G}\from\Set[\mathfrak{X}]\to\Set[\mathfrak{X}]$
are objects of $\Fun^{\t,L}\pr{\Set[\mathfrak{X}],\Set[\mathfrak{X}]}$
with images $F,G\in\SSeq\pr{\Set}$, then we have
\begin{align*}
\ev_{\mathfrak{X}}\pr{\overline{F}\circ\overline{G}} & =\overline{F}\pr G\\
 & \cong\overline{F}\pr{\int^{S\in\FB}\mathfrak{X}^{\bigstar S}\cdot G\pr S}\\
 & \cong\int^{S\in\FB}F^{\bigstar S}\cdot G\pr S\\
 & \cong F\circ G.
\end{align*}
These isomorphisms and the identity morphism $\ev_{\mathfrak{X}}\pr{\id_{\Set[\mathfrak{X}]}}=\mathfrak{X}$
enhances $\ev_{\mathfrak{X}}$ to a monoidal functor.
\end{proof}

The goal of this subsection is to describe the $\infty$-categorical
generalization of this story, due to Brantner.
\begin{rem}
As is clear from the above discussion, Brantner's model lives in the
world of $\pr{\infty,2}$-categories. In what follows, we will use
$\infty$-bicategories as our preferred model of $\pr{\infty,2}$-categories.
A brief summary of this model can be found in Appendix \ref{sec:oo2}.
We will also make a few references to Lemma \ref{lem:Fun(A.-)}. No
circularity will result from these forward references.
\end{rem}

To state the definition of Brantner's composition product monoidal
structure, we need to introduce a bit of notation.
\begin{defn}
A \textbf{presentably symmetric monoidal $\infty$-category} is a
symmetric monoidal $\infty$-category $\cat C^{\t}$ satisfying the
following pair of conditions:
\begin{itemize}
\item The $\infty$-category $\cat C$ is presentable.
\item The tensor bifunctor $\cat C\times\cat C\to\cat C$ preserves small
colimits in each variable.
\end{itemize}
\end{defn}

\begin{defn}
We define an $\infty$-bicategory $\Pr\SM^{\pr 2}$ to be the one
associated with (via Recollection \ref{recall:Bergner}) the simplicial
category whose:
\begin{itemize}
\item objects are the presentably symmetric monoidal $\infty$-categories;
and
\item whose hom-simplicial set from $\cat C^{\t}$ to $\cat D^{\t}$ is
given by the full subcategory $\Fun^{\t,L}\pr{\cat C,\cat D}\subset\Fun^{\t}\pr{\mathcal{C},\mathcal{D}}$
spanned by the symmetric monoidal functors whose underlying functor
$\cat C\to\cat D$ preserves small colimits.
\end{itemize}
The underlying $\infty$-category of $\Pr\SM^{\pr 2}$ will be denoted
by $\Pr\SM$.
\end{defn}

\begin{notation}
We denote the binary coproduct in $\Pr\SM$ by $\otimes$. More generally,
if $\cat B^{\t}\ot\cat A^{\t}\to\cat C^{\t}$ are maps in $\Pr\SM$,
we denote the pushout by $\pr{\cat B\otimes_{\cat A}\cat C}^{\t}$.
Note that it exists by \cite[Corollary 3.2.3.3]{HA} and \cite[Theorem 5.5.3.18, Corollary 5.5.3.4, and Remark 5.5.3.9]{HTT}.
\end{notation}

\begin{rem}
\label{rem:po_oo.2}Pushouts in $\Pr\SM$ satisfies the following
$\pr{\infty,2}$-universal property: Given maps $\cat B^{\t}\ot\cat A^{\t}\to\cat C^{\t}$
in $\Pr\SM$ and an object $\cat Z^{\t}\in\Pr\SM$, the functor
\[
\theta_{\cat Z}:\Fun^{\t,L}\pr{\cat B\otimes_{\cat A}\cat C,\cat Z}\to\Fun^{\t,L}\pr{\cat B,\cat Z}\times_{\Fun^{\t,L}\pr{\cat A,\cat Z}}\Fun^{\t,L}\pr{\cat C,\cat Z}
\]
is a categorical equivalence. To see this, it suffices to show that
for every $\infty$-category $\mathcal{X}$, the functor $\Fun\pr{\mathcal{X},\theta}^{\simeq}$
is a homotopy equivalence. But we can identify $\Fun\pr{\cat X,\theta}^{\simeq}$
with $\pr{\theta_{\Fun\pr{\cat X,\cat Z}}}^{\simeq}$, which is a
homotopy equivalence by the definition of pushouts. (Here $\Fun\pr{\cat X,\cat Z}^{\t}$
is defined as the fiber product $\Fun\pr{\cat X,\cat Z^{\t}}\times_{\Fun\pr{\cat X,\Fin_{\ast}}}\Fin_{\ast}$.) 
\end{rem}

\begin{notation}
Let $\cat A^{\t}$ be a small symmetric monoidal $\infty$-category.
According to \cite[Corollary 4.8.1.12]{HA}, the $\infty$-category
$\Fun\pr{\cat A,\mathcal{S}}$ admits a presentably symmetric monoidal
structure which is characterized by the property that the Yoneda embedding
$\cat A^{\op}\to\Fun\pr{\cat A,\mathcal{S}}$ can be enhanced to a
symmetric monoidal functor. We let $\Fun\pr{\cat A,\mathcal{S}}^{\bigstar}$
denote corresponding symmetric monoidal $\infty$-category. 

If $\cat C^{\t}$ is a presentably symmetric monoidal $\infty$-category,
we define another presentably symmetric monoidal $\infty$-category
$\Fun\pr{\cat A,\cat C}^{\bigstar}$ and a symmetric monoidal functor
$i_{\cat A,\cat C}$ by the pushout (or coproduct)
\[\begin{tikzcd}
	{\mathcal{S}^\times } & {\operatorname{Fun}(\mathcal{A},\mathcal{S})^{\bigstar }} \\
	{\mathcal{C}^\otimes } & {\operatorname{Fun}(\mathcal{A},\mathcal{C})^{\bigstar}}
	\arrow[from=1-1, to=1-2]
	\arrow[from=1-1, to=2-1]
	\arrow["\ulcorner"{description, pos=1}, draw=none, from=1-1, to=2-2]
	\arrow[from=1-2, to=2-2]
	\arrow["{i_{\mathcal{A},\mathcal{C}}}"', from=2-1, to=2-2]
\end{tikzcd}\](Note that the underlying $\infty$-category of $\Fun\pr{\cat A,\cat C}^{\bigstar}$
is equivalent to $\Fun\pr{\cat A,\cat C}$ by Lemma \ref{lem:Fun(A.-)},
justifying the notation.)
\end{notation}

\begin{notation}
\label{nota:symseq_day}Given a presentably symmetric monoidal $\infty$-category
$\cat C^{\t}$, we write 
\[
\cat C[\mathfrak{X}]^{\bigstar}=\Fun\pr{\FB,\cat C}^{\bigstar},
\]
and let $\mathfrak{X}\in\mathcal{C}[\mathfrak{X}]$ denote the image
of the unit symmetric sequence $\mathfrak{X}=\pr{\emptyset,\underline{1},\emptyset,\emptyset,\dots}\in\Fun\pr{\FB,\mathcal{S}}$.
We also write $i_{\cat C}=i_{\FB^{\op},\cat C}$ for the symmetric
monoidal functor $\cat C^{\t}\to\mathcal{C}[\mathfrak{X}]^{\bigstar}$.
\end{notation}

\begin{rem}
In the situation of Notation \ref{nota:symseq_day}, Lemma \ref{lem:Fun(A.-)}
shows that an object of $\cat C[\mathfrak{X}]$ can informally be
written as a sequence $\pr{X\pr 0,X\pr 1,\dots}$ of objects in $\cat C$,
where each $X\pr n$ carries a $\Sigma_{n}$-action. The functor $i_{\cat C}$
of Notation \ref{nota:symseq_day} is then given by the formula
\[
i_{\cat C}\pr C=\pr{C,\emptyset,\emptyset,\dots}.
\]
\end{rem}

We can now define Brantner's composition product monoidal structure.
\begin{defn}
Let $\mathcal{C}^{\t}$ be a presentably symmetric monoidal $\infty$-category.
We define a monoidal $\infty$-category $\SSeq_{\B}\pr{\mathcal{C}}^{\circ}$
as the endomorphism monoidal $\infty$-category (Definition \ref{def:end})
of the $\infty$-bicategory $\CAlg^{\pr 2}_{\cat C}=\pr{\Pr\SM^{\pr 2}}^{\cat C^{\t}/}$
(Example \ref{exa:slice}) at $\cat C[\mathfrak{X}]^{\bigstar}$.
In symbols, we have
\[
\SSeq_{\B}\pr{\cat C}^{\circ}=\cEnd_{\CAlg^{\pr 2}_{\cat C}}\pr{\cat C[\mathfrak{X}]^{\bigstar}}^{\circ}.
\]
\end{defn}

\begin{rem}
\label{rem:Brantner_functorial}The assignment $\cat C^{\t}\mapsto\SSeq_{\B}\pr{\cat C}^{\circ}$
can be assembled into a functor
\[
\SSeq_{\B}\pr -^{\circ}\from\Pr\SM\to\Mon\hat{\Cat}_{\infty}
\]
as follows: Define an $\infty$-bicategory $\int^{\Pr\SM}\CAlg^{\pr 2}_{\bullet}$
by the pullback
\[\begin{tikzcd}
	{\int^{\mathcal{P}\mathsf{r}\mathcal{SM}}\operatorname{CAlg}_{\bullet}^{(2)}} & {\operatorname{Fun}^{\mathrm{bi}}([1],\mathcal{P}\mathsf{r}\mathcal{SM}^{(2)})} \\
	{\mathcal{P}\mathsf{r}\mathcal{SM}} & {\mathcal{P}\mathsf{r}\mathcal{SM}^{(2)}.}
	\arrow[from=1-1, to=1-2]
	\arrow["\pi"', from=1-1, to=2-1]
	\arrow["\lrcorner"{description, pos=0}, draw=none, from=1-1, to=2-2]
	\arrow["{\operatorname{ev}_0}", from=1-2, to=2-2]
	\arrow[from=2-1, to=2-2]
\end{tikzcd}\]By Example \ref{exa:cc_po} and Remark \ref{rem:po_oo.2}, the functor
$\pi$ is a cocartesian fibration (Definition \ref{def:cart}). We
let $\CAlg^{\pr 2}_{\bullet}\from\Pr\SM\to\BiCat_{\infty}$ denote
the straightening (Theorem \ref{thm:st}) of $\pi$.

Since $\cat S^{\times}\in\Pr\SM$ is initial, there is an (essentially
unique) cocartesian section $\sigma$ of $\pi$ which carries $\cat S^{\times}$
to the inclusion
\[
i_{\cat S}\from\cat S^{\times}\hookrightarrow\cat S[\mathfrak{X}]^{\bigstar}.
\]
Via straightening, the section $\sigma$ lifts the functor $\CAlg^{\pr 2}_{\bullet}$
to a functor
\[
\Pr\SM\to\pr{\BiCat_{\infty}}_{[0]/}.
\]
Composing this with the functor $\cEnd\from\pr{\BiCat_{\infty}}_{[0]/}\to\Mon\Cat_{\infty}$
of Definition \ref{def:end}, we get the desired functor
\[
\SSeq_{\B}\pr -^{\circ}\from\Pr\SM\to\Mon\Cat_{\infty}.
\]
\end{rem}

The notation $\SSeq_{\B}\pr{\mathcal{C}}$ is justified by the following
proposition:
\begin{prop}
\label{prop:SSeq_B_underlying}Let $\cat C^{\t}$ be a presentably
symmetric monoidal $\infty$-category. Then $\cat C[\mathfrak{X}]^{\bigstar}\in\CAlg^{\pr 2}_{\cat C}$
is freely generated by $\mathfrak{X}$ in the following sense: For
every $\cat D^{\t}\in\CAlg^{\pr 2}_{\cat C}$, the evaluation at $\mathfrak{X}$
induces a categorical equivalence\footnote{Here $\CAlg^{\pr 2}_{\cat C}\pr{\cat C[\mathfrak{X}],\cat D}$ is
a shorthand for $\CAlg^{\pr 2}_{\cat C}\pr{\cat C[\mathfrak{X}]^{\bigstar},\cat D^{\t}}$.}
\[
\CAlg^{\pr 2}_{\cat C}\pr{\cat C[\mathfrak{X}],\cat D}\xrightarrow{\simeq}\cat D.
\]
In particular, the evaluation at $\mathfrak{X}$ gives a categorical
equivalence
\[
\SSeq_{\B}\pr{\cat C}\xrightarrow{\simeq}\cat C[\mathfrak{X}]\simeq\Fun\pr{\FB,\mathcal{C}}.
\]
\end{prop}

\begin{proof}
Remark \ref{rem:po_oo.2} and Proposition \ref{prop:mappingcat_arrow}
gives an equivalence 
\[
\CAlg^{\pr 2}_{\cat C}\pr{\cat C[\mathfrak{X}],\cat D}\xrightarrow{\simeq}\Fun^{\t,L}\pr{\cat S[\mathfrak{X}],\cat D},
\]
so we are reduced to the case where $\cat C^{\t}=\cat S^{\times}$.
The universal property of the Day convolution symmetric monoidal structure
(which follows from the discussion in \cite[Corollary 4.8.1.12]{HA}
and an argument similar to Remark \ref{rem:po_oo.2}) gives an equivalence
\[
\Fun^{\t,L}\pr{\cat S[\mathfrak{X}],\cat D}\xrightarrow{\simeq}\Fun^{\t}\pr{\FB,\cat D}.
\]
Since $\FB^{\op,\amalg}\cong\FB^{\amalg}$ can be identified with
the symmetric monoidal envelope (Construction \ref{const:env}) of
the trivial $\infty$-operad $\mathsf{Triv}^{\t}$ \cite[Example 2.1.1.20]{HA},
we further have an equivalence 
\[
\Fun^{\t}\pr{\FB,\cat D}\xrightarrow{\simeq}\Alg_{\mathsf{Triv}}\pr{\cat D}.
\]
By \cite[Example 2.1.3.5]{HA}, the evaluation at the unique object
of $\mathsf{Triv}^{\t}$ gives an equivalence 
\[
\Alg_{\mathsf{Triv}}\pr{\cat D}\xrightarrow{\simeq}\cat D.
\]
The resulting equivalence $\Fun^{\t,L}\pr{\cat S[\mathfrak{X}],\cat D}\xrightarrow{\simeq}\cat D$
is given by the evaluation at $\mathfrak{X}$, and the claim follows.
\end{proof}

\subsection{\label{subsec:Haug}Haugseng's model}

We now turn to Haugseng's model of composition product. In contrast
to Brantner's model, Haugseng's model is characterized by a universal
property of maps into it. To state it, we must introduce a bit of
notation.
\begin{defn}
\cite[Definition 4.1.1]{Haug22}\label{def:Haug22_4.1} We define
a category $\Del_{\mathbb{F}}$ of \textbf{forests} as follows:
\begin{enumerate}
\item Its objects are the finite sequences $S_{0}\to\cdots\to S_{n}$ of
maps of finite sets in $\Fin$, where $n\geq0$. We think of such
a sequence as a forest whose edges are the elements of $\coprod_{i\geq0}S_{i}$
and whose vertices are the elements of $\coprod_{i>1}S_{i}$.
\item A morphism $\pr{S_{0}\to\cdots\to S_{n}}\to\pr{T_{0}\to\cdots\to T_{m}}$
is given by a morphism $\phi\from[n]\to[m]$ of $\Del$ and injective
set maps $\{u_{i}\from S_{i}\to T_{\phi\pr i}\}_{0\leq i\leq n}$
such that, for each $0\leq i\leq j\leq n$, the square 
\[\begin{tikzcd}
	{S_i} & {S_j} \\
	{T_{\phi(i)}} & {T_{\phi(j)}}
	\arrow[from=1-1, to=1-2]
	\arrow[from=1-1, to=2-1]
	\arrow[from=1-2, to=2-2]
	\arrow[from=2-1, to=2-2]
\end{tikzcd}\]is commutative and cartesian. 
\end{enumerate}
\end{defn}

In the theory of operads, the category $\Del^{\op}_{\mathbb{F}}$
roughly plays the role that $\Del^{\op}$ plays for categories. Informally,
the map $\phi$ decomposes $T_{\bullet}$ into a bunch of subtrees
(with boundaries in $T_{\phi\pr i}$ and $T_{\phi\pr{i+1}}$), and
the maps $u_{i}$ carry each of these subtrees into subcorollas of
$S_{\bullet}$ with matching boundaries or discard it entirely.
\begin{defn}
We define the ``vertex functor'' $V\from\Del^{\op}_{\mathbb{F}}\to\Fin_{\ast}$
by the formula
\[
V\pr{S_{0}\to\cdots\to S_{n}}=\pr{\coprod_{i>0}S_{i}}_{\ast},
\]
where $\pr -_{\ast}\from\Fin\to\Fin_{\ast}$ adds a disjoint basepoint.
Given a morphism $\pr{\phi,\{u_{i}\}_{i}}$ as in Definition \ref{def:Haug22_4.1},
the map $\pr{\coprod_{j>0}T_{j}}_{\ast}\to\pr{\coprod_{i>0}S_{i}}_{\ast}$
is defined as follows: Let $0<i\leq n$ and $s\in S_{j}$. Then the
preimage of $s$ is $\pr{\coprod_{\phi\pr{j-1}<k\leq\phi\pr j}T_{k}}_{u_{i}\pr s}$,
where the subscript indicates the preimage over $u_{i}\pr s\in T_{\phi\pr i}$. 
\end{defn}

\begin{defn}
\cite[Notation 4.2.7]{Haug22} Let $\cat O^{\t}\to\Del^{\op}$ be
a non-symmetric $\infty$-operad. A morphism of $\cat O^{\t}\times_{\Del^{\op}}\Del^{\op}_{\mathbb{F}}$
is called \textbf{operadic inert} if its image in $\cat O^{\t}$ is
inert. Given an $\infty$-operad $\cat C^{\t}\to\Fin_{\ast}$, we
write $\Alg^{\opd}_{\cat O^{\t}\times_{\Del^{\op}}\Del^{\op}_{\mathbb{F}}}\pr{\cat C}$
for the full subcategory of $\Fun_{\Fin_{\ast}}\pr{\cat O^{\t}\times_{\Del^{\op}}\Del^{\op}_{\mathbb{F}},\cat C^{\t}}$
spanned by the maps carrying operadic inert maps to inert maps of
$\cat C^{\t}$.
\end{defn}

With these definitions, we can state the universal property of Haugseng's
composition product monoidal structure.
\begin{thm}
\cite[Corollary 4.2.9]{Haug22}\label{thm:haug} Let $\cat C^{\t}$
be a symmetric monoidal $\infty$-category compatible with colimits
indexed by small $\infty$-groupoids. There is a monoidal $\infty$-category
$\SSeq_{\H}\pr{\cat C}^{\circ}$ characterized by the equivalence
\[
\Alg^{\opd}_{\cat O^{\t}\times_{\Del^{\op}}\Del^{\op}_{\mathbb{F}}}\pr{\cat C}\simeq\Alg_{\cat O}\pr{\SSeq_{\H}\pr{\cat C}}
\]
natural in the non-symmetric $\infty$-operad $\cat O^{\t}$. Moreover:
\begin{enumerate}
\item The construction is functorial in the variable $\cat C^{\t}$ and
symmetric monoidal functors preserving colimits indexed by $\infty$-groupoids.
\item There is an equivalence $\SSeq_{\H}\pr{\cat C}\simeq\Fun\pr{\FB,\cat C}$
which is functorial in the sense of (1).
\end{enumerate}
\end{thm}

\subsection{\label{subsec:main_result}Main Result}

We can now state the main result of this paper.
\begin{thm}
\label{thm:main}There is an equivalence of monoidal $\infty$-categories
\[
\SSeq_{\B}\pr{\cat C}^{\circ}\simeq\SSeq_{\H}\pr{\cat C}^{\circ}
\]
natural in $\cat C^{\otimes}\in\Pr\SM$.
\end{thm}

In the next two sections (Sections \ref{sec:review_on_loc} and \ref{sec:comp_prod_mod_cat}),
we establish basic results on monoidal localizations and composition
product and Day's convolution product in the model-categorical setting.
We then take a closer look at Brantner and Haugseng's models in the
ensuing sections (Sections \ref{sec:brantner} and \ref{sec:haugseng}).
The proof of Theorem \ref{thm:main} will be given in Section \ref{sec:proof}.

\section{\label{sec:review_on_loc}Review on Localization}

In this section, we recall a few key results and constructions on
localization of symmetric or non-symmetric monoidal categories. They
will be used in an essential way in the rest of the paper.

First we recall monoidal relative $\infty$-categories and their localization.
\begin{defn}
\label{def:monoidalrelative}\cite{A25a} A \textbf{monoidal relative
$\infty$-category} is a pair $\pr{\cat M^{\t},\cat W}$, where $\cat M^{\t}$
is a monoidal $\infty$-category and $\cat W\subset\cat M$ is a subcategory
containing all equivalences and are stable under tensor products.
Morphisms in $\cat W$ are called \textbf{weak equivalences}. When
$\cat W$ is clear from the context, we often drop it from the notation
and say that $\cat M^{\t}$ is a monoidal relative $\infty$-category.

We let $\Mon\Rel\Cat^{\pr 2}_{\infty}$ denote the $\infty$-bicategory
of monoidal relative $\infty$-category; formally, it is the scaled
nerve of the $\SS^{+}$-enriched category whose:
\begin{itemize}
\item Objects are monoidal relative $\infty$-categories. 
\item Mapping object between a pair of objects $\cat C^{\t},\cat D^{\t}$
is given by the full subcategory $\Fun^{\t,\rel}\pr{\cat C,\cat D}\subset\Fun^{\t}\pr{\cat C,\cat D}$
spanned by the monoidal functors that preserve weak equivalences (with
equivalences marked). 
\end{itemize}
We think of $\Mon\Cat^{\pr 2}_{\infty}$ as a full sub $\infty$-bicategory
of $\Mon\Rel\Cat^{\pr 2}_{\infty}$ via the inclusion $\cat C^{\t}\mapsto\pr{\cat C^{\t},\cat C^{\simeq}}$.
The underlying $\infty$-category of $\Mon\Rel\Cat^{\pr 2}_{\infty}$
is denoted by $\Mon\Rel\Cat_{\infty}$.

We define symmetric monoidal relative $\infty$-categories and the
associated $\infty$-bicategory $\SM\Rel\Cat^{\pr 2}_{\infty}$ similarly.
\end{defn}

\begin{construction}
\label{const:monoidalloc}Let $\pr{\cat M^{\t},\cat W}$ be a monoidal
relative $\infty$-category. The\textbf{ monoidal localization} of
$\pr{\cat M^{\t},\cat W}$ is a symmetric monoidal $\infty$-category
$\cat N^{\t}$ equipped with a symmetric monoidal functor $\eta_{\pr{\cat M^{\t},\cat W}}\from\cat M^{\t}\to\cat N^{\t}$
satisfying the following equivalent conditions:
\begin{enumerate}
\item The underlying functor $\cat M\to\cat N$ is a localization at $\cat W$.
\item For every symmetric monoidal $\infty$-category $\cat P$, the functor
\[
\eta^{*}\from\Fun^{\t}\pr{\cat N,\cat P}\to\Fun^{\t}\pr{\cat M,\cat P}
\]
is fully faithful, and its essential image consists of the functors
$\cat M\to\cat P$ carrying weak equivalences to equivalences.
\end{enumerate}
(The equivalence of these conditions is proved in \cite[Proposition 4.1.7.4]{HA}.)
In this situation, we write $\cat N^{\t}=\cat M[\cat W^{-1}]^{\t}$.

By \cite[Corollary A.2.18]{BB24}, the assignment $\pr{\cat M^{\t},\cat W}\mapsto\cat M[\cat W^{-1}]^{\t}$
can be turned into a functor
\[
L\from\Mon\Rel\Cat^{\pr 2}_{\infty}\to\Mon\Cat^{\pr 2}_{\infty},
\]
and the maps $\eta_{\pr{\cat M^{\t},\cat W}}$ assemble to a natural
transformation $\eta\from\id\To\iota\circ L$. For an explicit construction,
see Lemma \ref{rem:simplicial_realization_of_loc} below. There is
a parallel construction in the symmetric monoidal setting too, which
we leave to the reader. 
\end{construction}

\begin{rem}
\label{rem:loc_unique}By definition, localization of monoidal relative
$\infty$-categories are unique up to equivalence. A stronger uniqueness
is also true: Let $\cat C$ be an $\infty$-category, and let $F\from\cat C\to\Mon\Rel\Cat_{\infty}$
be a functor denoted by $C\mapsto\pr{\cat M^{\t}_{C},\cat W_{C}}$.
Suppose we are given another functor $\cat N^{\t}_{\bullet}\from\cat C\to\Mon\Cat_{\infty}$
and a natural transformation $\alpha\from\cat M^{\t}_{\bullet}\to\cat N^{\t}_{\bullet}$
of functors $\cat C\to\Mon\Cat_{\infty}$. If for each $C\in\cat C$,
the map $\alpha_{C}$ exhibits $\cat N^{\t}_{C}$ as a monoidal localization
of $\cat M^{\t}_{C}$ at $\cat W_{C}$, then there is a natural equivalence
\[
L\pr{\cat M^{\t}_{\bullet},\cat W_{\bullet}}\simeq\cat N^{\t}_{\bullet}
\]
of functors $\cat C\to\Mon\Cat_{\infty}$. This follows from the fact
that the inclusion
\[
\Fun\pr{\cat C,\Mon\Cat_{\infty}}\hookrightarrow\Fun\pr{\cat C,\Mon\Rel\Cat_{\infty}}
\]
has a left adjoint given by postcomposition by $L$ and with unit
induced by $\eta$. (Alternatively, this follows from \cite[Proposition A.11]{Ram23}.)
In other words, the assignment $C\mapsto\cat M^{\t}_{C}[\mathcal{W}^{-1}_{C}]$
can be exteneded to a functor in an essentially unique way. Because
of this, we typically denote the functor $L\pr{\cat M^{\t}_{\bullet},\cat W_{\bullet}}$
by $\cat M^{\t}_{\bullet}[\mathcal{W}^{-1}_{\bullet}]$.
\end{rem}

We are particularly interested in the localization of symmetric monoidal
model categories. 
\begin{defn}
A \textbf{symmetric monoidal model category} is a model category $\mathbf{M}$
equipped with a closed symmetric monoidal structure, subject to the
following pair of conditions:
\begin{itemize}
\item (\textbf{Cofibrant unit}) The monoidal unit is cofibrant.
\item (\textbf{Pushout-product}) For every pair of cofibrations $i\from A\to B$
and $j\from X\to Y$ in $\mathbf{M}$, their pushout-product
\[
i\hat{\otimes}j\from\pr{A\otimes Y}\amalg_{A\otimes X}\pr{B\otimes X}\to B\otimes Y
\]
is a cofibration. If further $i$ is a trivial cofibration, so is
$i\hat{\otimes}j$.
\end{itemize}
We write $\TSMMC$ for the category of tractable symmetric monoidal
model categories and left Quillen symmetric monoidal functors. (Recall
that a model category is \textbf{tractable} if it is locally presentable
as a category and has a generating sets of cofibrations and trivial
cofibrations with cofibrant domains.)

Let $\mathbf{M}$ be a tractable symmetric monoidal model category.
The definition of symmetric monoidal model categories ensures that
the full subcategory $\mathbf{M}_{\cof}\subset\mathbf{M}$ of cofibrant
objects inherits a symmetric monoidal structure from $\mathbf{M}$
and is a relative symmetric monoidal category. The \textbf{underlying
symmetric monoidal $\infty$-category} of $\mathbf{M}$, denoted by
$\mathbf{M}^{\t}_{\infty}$, is the symmetric monoidal localization
of the symmetric monoidal relative category $\mathbf{M}_{\cof}$ at
weak equivalences. 
\end{defn}

Construction \ref{const:monoidalloc} gives us a functor
\[
\TSMMC^{\pr 2}\to\Pr\SM^{\pr 2},\,\mathbf{M}\mapsto\mathbf{M}^{\t}_{\infty}.
\]
The main theorem of \cite{arakawa_pres_smcat_comb_smcat} asserts
the following\footnote{We note that loc. cit. writes $\TSMMC^{\mathbf{1}}$ for the category
$\TSMMC$, because a weaker axiom for the unit object is adopted there. }:
\begin{thm}
\label{thm:delocPrSM}\cite[Theorem 6.2 (2)]{arakawa_pres_smcat_comb_smcat}
The functor
\[
\pr -^{\t}_{\infty}\from\TSMMC\to\Pr\SM
\]
is a localization.
\end{thm}

\section{\label{sec:comp_prod_mod_cat}Composition product and Day convolution
from model-categorical viewpoint}

In this section, we study composition product and Day convolution
from a model-categorical perspective. In Subsection \ref{subsec:Day},
we will show that Day's convolution product models $\infty$-categorical
convolution product in the model-categorical setting. In Subsection
\ref{subsec:comp}, we show that the composition product of cofibrant
symmetric sequences in a model category behaves well homotopically.
\begin{rem}
\label{rem:proj_convention}For the remainder of this paper, we will
adopt the following convention on functor categories of model categories:
Let $\mathbf{M}$ be a cofibrantly generated model category, and let
$\cat I$ be a small category. We will always equip $\Fun\pr{\cat I,\mathbf{M}}$
with the \textit{projective} model structure, whose fibrations and
weak equivalences are defined pointwise \cite[Theorem 11.6.1]{Hirschhorn}. 

Recall that if $I$ and $J$ are generating sets of cofibrations of
$\mathbf{M}$, then the maps $\{\cat I\pr{i,-}\otimes f\mid f\in I\}$
and $\{\cat I\pr{i,-}\otimes g\mid g\in J\}$ generate the cofibrations
and trivial cofibrations of the projective model structure. (Here
$\otimes$ denotes tensor by sets.)
\end{rem}

\subsection{\label{subsec:Day}Day convolution}

Let $\mathbf{M}$ be a combinatorial symmetric monoidal model category.
Given a small symmetric monoidal category $\mathcal{A}$, the category
$\Fun\pr{\cat A,\mathbf{M}}$ carries a symmetric monoidal structure,
given by the ordinary ($1$-categorical) Day convolution product $\bigstar_{1}$.
The goal of this subsection is to prove the following proposition,
which says that the Day convolution makes $\Fun\pr{\cat A,\mathbf{M}}$
into a symmetric monoidal model category whose underlying symmetric
monoidal $\infty$-category is what we would expect:
\begin{prop}
\label{prop:Day_model}Let $\mathbf{M}$ be a combinatorial symmetric
monoidal model category with a cofibrant unit, and let $\cat A$ be
a small symmetric monoidal category. Equip $\Fun\pr{\cat A,\mathbf{M}}$
with the projective model structure and Day's convolution product.
\begin{enumerate}
\item The category $\Fun\pr{\cat A,\mathbf{M}}$ is a symmetric monoidal
model category with cofibrant unit.
\item There is a coproduct cone
\[
\mathbf{M}^{\t}_{\infty}\xrightarrow{}\Fun\pr{\cat A,\mathbf{M}}^{\bigstar_{1}}_{\infty}\xleftarrow{g}\Fun\pr{\cat A,\cat S}^{\bigstar}
\]
in $\Pr\SM$, where the left arrow is induced by the symmetric monoidal
functor 
\[
\mathbf{M}\to\Fun\pr{\cat A,\mathbf{M}},\,M\mapsto\cat A\pr{\mathbf{1},-}\cdot M.
\]
\item For every left Quillen symmetric monoidal functor $\mathbf{M}\to\mathbf{N}$
of combinatorial symmetric monoidal model categories, the square 
\[\begin{tikzcd}
	{\mathbf{M}_\infty^{\otimes }} & {\mathbf{N}_\infty^{\otimes }} \\
	{\operatorname{Fun}(\mathcal{A},\mathbf{M})^{\bigstar_1}_{\infty}} & {\operatorname{Fun}(\mathcal{A},\mathbf{N})^{\bigstar_1}_{\infty}}
	\arrow[from=1-1, to=1-2]
	\arrow[from=1-1, to=2-1]
	\arrow[from=1-2, to=2-2]
	\arrow[from=2-1, to=2-2]
\end{tikzcd}\]is cocartesian in $\Pr\SM$.
\item We have an equivalence of symmetric monoidal $\infty$-categories
\[
\Fun\pr{\cat A,\mathbf{M}}^{\bigstar_{1}}_{\infty}\simeq\Fun\pr{\cat A,\mathbf{M}_{\infty}}^{\bigstar}.
\]
\end{enumerate}
\end{prop}

Before turning to the proof of Proposition \ref{prop:Day_model},
let us prove one of its consequences.
\begin{cor}
\label{cor:ordinary_day}Let $\cat C$ be a symmetric monoidal category
such that $\cat C^{\t}$ is presentably symmetric monoidal, and let
$\cat A$ be a small symmetric monoidal category. There is a coproduct
cone
\[
\cat C^{\t}\xrightarrow{\pr{\{\mathbf{1}\}\hookrightarrow\cat A}_{!}}\Fun\pr{\cat A,\cat C}^{\bigstar_{1}}\ot\Fun\pr{\cat A,\cat S}^{\bigstar}
\]
in $\Pr\SM$. In particular, we have an equivalence of symmetric monoidal
$\infty$-categories
\[
\Fun\pr{\cat A,\cat C}^{\bigstar_{1}}\simeq\Fun\pr{\cat A,\cat C}^{\bigstar}.
\]
\end{cor}

\begin{proof}
We will apply Proposition \ref{prop:Day_model} to the trivial model
structure on $\cat C$, whose weak equivalences are the isomorphisms
 and whose morphisms are all fibrations. The only nontrivial part
is that this model structure is combinatorial. 

To see the trivial model structure is combinatorial, take a small
regular cardinal $\kappa$ such that $\cat C$ is $\kappa$-presentable,
and let $S$ be a set of representatives of isomorphism classes of
$\kappa$-compact objects in $\cat C$. We then set $I=\{A\to B\}_{A,B\in S}$
and $J=\{\id_{A}\}_{A\in S}$. We claim that $I$ and $J$ generate
the classes of cofibrations and trivial cofibrations of $\cat C$.

It is obvious that every morphism of $\cat C$ has the right lifting
property for the maps in $J$, and that every isomorphism has the
right lifting property for the maps in $I$. It will therefore suffice
to show that if a map $f\from X\to Y$ has the right lifting property
for the maps in $I$, then $f$ is an isomorphism. For each $A\in S$,
the right lifting property for the maps $\emptyset\to A$ and $A\amalg A\to A$
implies that the map $f_{\ast}\from\cat C\pr{A,X}\to\cat C\pr{A,Y}$
is bijective. Since every object of $\cat C$ is a colimit of objects
in $S$, it follows that this map is bijective for any object $A\in\cat C$.
Thus $f$ is an isomorphism, as required.
\end{proof}

We now turn to the proof Proposition \ref{prop:Day_model}, which
needs a few preliminaries.

For the next lemma, recall from \cite[Corollary 5.5.3.4 and Theorem 5.5.3.18]{HTT}
that the $\infty$-category $\Pr^{L}$ of presentable $\infty$-categories
and left adjoints has small colimits. 
\begin{lem}
\label{lem:PrL_generation}The $\infty$-category $\Pr^{L}$ is generated
under small colimits by the presheaves on small $\infty$-categories.
\end{lem}

\begin{proof}
Recall that every presentable $\infty$-category has the form $\cat P\pr{\cat A}[S^{-1}]$,
where $\cat A$ is a small $\infty$-category and $S$ is a small
set of morphisms of $\cat P\pr{\cat A}$. This localization can be
written as a pushout (in $\Pr^{L}$) of the span
\[
\cat P\pr{\coprod_{S}[0]}\xleftarrow{}\cat P\pr{\coprod_{f\in S}[1]}\xrightarrow{f}\cat P\pr{\cat A},
\]
where $f$ is the unique colimit preserving functor determined by
the tautological functor $\coprod_{f\in S}[1]\to\cat P\pr{\cat A}$.
The claim follows.
\end{proof}

\begin{lem}
\label{lem:Fun(A.-)}Let $\cat C$ be a presentable $\infty$-category,
and let $\cat A$ be a small $\infty$-category. Consider the functor
$\Fun\pr{\cat A^{\op},\cat S}\times\cat C\to\Fun\pr{\cat A^{\op},\cat C}$
adjoint to the composite
\[
\cat A^{\op}\times\Fun\pr{\cat A^{\op},\cat S}\times\cat C\to\cat S\times\cat C\to\cat S\otimes\cat C\simeq\cat C
\]
The functor $\Phi$ induces an equivalence $\alpha_{\cat C}\from\cat C\otimes\Fun\pr{\cat A^{\op},\cat S}\xrightarrow{\simeq}\Fun\pr{\cat A^{\op},\cat C}$
in $\Pr^{L}$.
\end{lem}

\begin{proof}
We first observe that the functor $\Fun\pr{\cat A^{\op},-}:\Pr^{L}\to\Pr^{L}$
preserves small colimits. Indeed, using the equivalence $\Pr^{L}\simeq\pr{\Pr^{R}}^{\op}$
of \cite[Corollary 5.5.3.4]{HTT}, we only have to show that the functor
$\Fun\pr{\cat A^{\op},-}\from\Pr^{R}\to\Pr^{R}$ preserves small limits.
This is immediate from \cite[Theorem 5.5.3.18]{HTT}.

Now since $\Pr^{L}$ is closed symmetric monoidal \cite[Remark 4.8.1.18]{HA},
the functor $\cat P\pr{\cat A}\otimes-\from\Pr^{L}\to\Pr^{L}$ also
preserves small colimits. Therefore, by Lemma \ref{lem:PrL_generation},
it will suffice to show that $\alpha_{\cat C}$ is an equivalence
when $\cat C=\cat P\pr{\cat B}$ for some small $\infty$-category
$\cat B$. 

Let $\cat D$ be a presentable $\infty$-category. Precomposing the
Yoneda embeddings, we get equivalences
\[
\Fun^{L}\pr{\cat P\pr{\cat A}\otimes\cat P\pr{\cat B},\cat D}\xrightarrow{\simeq}\Fun\pr{\cat A\times\cat B,\cat D}\xleftarrow{\simeq}\Fun\pr{\cat P\pr{\cat A\times\cat B},\cat D}.
\]
This gives us an equivalence $\cat P\pr{\cat A}\otimes\cat P\pr{\cat B}\xrightarrow{\simeq}\cat P\pr{\cat A\times\cat B}$.
Unwinding the definitions, this equivalence is exactly the functor
$\alpha_{\cat P\pr{\cat B}}$, and we are done.
\end{proof}

\begin{lem}
\label{lem:coproduct_calg}Let $\cat C^{\t}$ be a symmetric monoidal
$\infty$-category, and let $X\xrightarrow{f}Z\xleftarrow{g}Y$ be
morphisms of commutative algebra objects of $\cat C^{\t}$. The following
conditions are equivalent:
\begin{enumerate}
\item The maps $f$ and $g$ exhibit $Z$ as a coproduct of $X$ and $Y$
in $\CAlg\pr{\cat C}$.
\item The composite
\[
U\pr X\otimes U\pr Y\xrightarrow{Uf\otimes Ug}U\pr Z\otimes U\pr Z\xrightarrow{\mu}U\pr Z
\]
is an equivalence in $\cat C$, where $U:\CAlg\pr{\cat C}\to\cat C$
denotes the forgetful functor and $\mu$ denotes the multiplication
of $Z$.
\end{enumerate}
\end{lem}

\begin{proof}
Recall that the symmetric monoidal $\infty$-category $\CAlg\pr{\cat C}^{\t}$
is cocartesian \cite[Proposition 3.2.4.3]{HA}. Therefore, by \cite[Proposition 2.4.3.16]{HA},
we can lift $Z\in\CAlg\pr{\cat C}$ in an essentially unique way to
a commutative algebra object $\overline{Z}$ in $\CAlg\pr{\cat C}^{\t}$.
Using \cite[Remark 2.4.3.4]{HA}, the multiplication map
\[
Z\otimes Z\to Z
\]
of $\overline{Z}$ can be identified with the codiagonal map $Z\amalg Z\to Z$.
So the composite
\[
X\otimes Y\xrightarrow{f\otimes g}Z\otimes Z\xrightarrow{\mu}Z
\]
is just the map induced by the universal property of coproducts and
the maps $f$ and $g$. It follows that condition (1) is equivalent
to the condition that $\mu\circ\pr{f\otimes g}$ be an equivalence.
Since $U$ is conservative and symmetric monoidal \cite[Proposition 3.2.4.3]{HA},
this is equivalent to condition (2). The proof is now complete.
\end{proof}

We now arrive at the proof of Proposition \ref{prop:Day_model}.
\begin{proof}
[Proof of Proposition \ref{prop:Day_model}]Part (1) is well-known
(see, e.g., \cite[Theorem 4.1]{BB17}), but it is difficult to find
a reference stating this in its exact form, so we record a proof anyway.
Recall from Remark \ref{rem:proj_convention} that if $I$ and $J$
are generating sets of cofibrations and trivial cofibrations of $\mathbf{M}$,
then the sets $\{\mathcal{A}\pr{a,-}\cdot i\}_{a\in\cat A,\,i\in I}$
and $\{\mathcal{A}\pr{a,-}\cdot j\}_{a\in\cat A,\,j\in J}$ generate
the cofibrations and trivial cofibrations of $\Fun\pr{\cat A,\mathbf{M}}$.
Using the isomorphism in $\Fun\pr{\cat A,\mathbf{M}}$
\[
\pr{\cat A\pr{a,-}\cdot M}\star\pr{\cat A\pr{b,-}\cdot N}\cong\cat A\pr{a\otimes b,-}\cdot\pr{M\otimes N},
\]
natural in $M,N\in\mathbf{M}$, we deduce that the projective model
structure satisfies the pushout-product axiom. Also, the unit object
$\cat A\pr{\mathbf{1}_{\cat A},-}\cdot\mathbf{1}_{\mathbf{M}}$ for
the Day convolution product is cofibrant, because $\mathbf{1}_{\mathbf{M}}$
is cofibrant. This proves (1).

For part (2), suppose first that there is an equivalence of symmetric
monoidal $\infty$-categories $\cat S^{\times}\simeq\mathbf{M}^{\t}_{\infty}$.
The Yoneda embedding
\begin{align*}
\cat A^{\op} & \to\Fun\pr{\cat A,\mathbf{M}},\\
a & \mapsto\cat A\pr{a,-}\cdot\mathbf{1}
\end{align*}
is symmetric monoidal and takes values in the full subcategory of
cofibrant objects. Thus, composing it with the symmetric monoidal
localization $\Fun\pr{\cat A,\mathbf{M}}^{\bigstar_{1}}_{\cof}\to\Fun\pr{\cat A,\mathbf{M}}^{\bigstar_{1}}_{\infty}$,
we obtain a symmetric monoidal functor $\cat A^{\op}\to\Fun\pr{\cat A,\mathbf{M}}^{\bigstar_{1}}_{\infty}$.
The universal property of the Day convolution product on $\Fun\pr{\cat A,\cat S}$
\cite[$\S$ 4.8]{HA} now gives a symmetric monoidal functor $F$ indicated
by the dashed arrow:
\[\begin{tikzcd}
	& {\mathcal{A}^{\mathrm{op}}} \\
	{\operatorname{Fun}(\mathcal{A},\mathcal{S})^{{\bigstar}}} && {\operatorname{Fun}(\mathcal{A},\mathbf{M})^{\bigstar_1}_\infty.}
	\arrow[from=1-2, to=2-1]
	\arrow[from=1-2, to=2-3]
	\arrow["F", dashed, from=2-1, to=2-3]
\end{tikzcd}\]We wish to show that $F$ is an equivalence. Since it is symmetric
monoidal, it suffices to show that $F$ induces an equivalence between
the underlying $\infty$-categories. Using \cite[Theorem 7.9.8]{HCHA},
we can identify $\Fun\pr{\cat A,\mathbf{M}}_{\infty}$ with $\Fun\pr{\cat A,\mathbf{M}_{\infty}}$.
Under this identification, the dashed arrow is given by postcomposing
the equivalence $\cat S\simeq\mathbf{M}_{\infty}$. In particular,
it is an equivalence, as desired.

For the general case, use \cite[Corollary 5.11]{arakawa_pres_smcat_comb_smcat}
to find a symmetric monoidal left Quillen functor $\Set^{\square^{\op}_{\Sigma}}\to\mathbf{M}$,
where $\Set^{\square^{\op}_{\Sigma}}$ denotes the tractable symmetric
monoidal model category of symmetric cubical sets. This gives us symmetric
monoidal left Quillen functors
\[
\mathbf{M}\xrightarrow{\phi}\Fun\pr{\cat A,\mathbf{M}}\xleftarrow{\psi}\Fun\pr{\cat A,\Set^{\square^{\op}_{\Sigma}}},
\]
where $\phi$ is given by $\phi\pr M=\cat A\pr{\mathbf{1},-}\cdot M$.
Localizing at weak equivalences, we obtain symmetric monoidal functors
\[
\mathbf{M}^{\t}_{\infty}\xrightarrow{\phi'}\Fun\pr{\cat A,\mathbf{M}}^{\bigstar_{1}}_{\infty}\xleftarrow{\psi'}\Fun\pr{\cat A,\Set^{\square^{\op}_{\Sigma}}}^{\bigstar_{1}}_{\infty}.
\]
From the argument in the previous paragraph, we know that $\Fun\pr{\cat A,\Set^{\square^{\op}_{\Sigma}}}^{\bigstar_{1}}_{\infty}$
is equivalent to $\Fun\pr{\cat A,\cat S}^{\bigstar}$. Therefore,
it suffices to show that $\phi'$ and $\psi'$ form a coproduct cone
in $\Pr\SM$. 

According to Lemma \ref{lem:coproduct_calg}, we must show that the
composite
\begin{align*}
\theta\from\mathbf{M}_{\infty}\times\Fun\pr{\cat A,\Set^{\square^{\op}_{\Sigma}}}_{\infty} & \xrightarrow{\phi'\times\psi'}\Fun\pr{\cat A,\mathbf{M}}_{\infty}\times\Fun\pr{\cat A,\mathbf{M}}_{\infty}\\
 & \xrightarrow{\otimes_{1\Day}}\Fun\pr{\cat A,\mathbf{M}}_{\infty}
\end{align*}
exhibits $\Fun\pr{\cat A,\mathbf{M}}_{\infty}$ as a tensor product
of $\mathbf{M}_{\infty}$ and $\Fun\pr{\cat A,\Set^{\square^{\op}_{\Sigma}}}_{\infty}$
in $\Pr^{L}$. Using \cite[Theorem 7.9.8]{HCHA}, we can identify
$\Fun\pr{\cat A,\mathbf{M}}_{\infty}$ with $\Fun\pr{\cat A,\mathbf{M}_{\infty}}$,
and $\Fun\pr{\cat A,\Set^{\square^{\op}_{\Sigma}}}_{\infty}$ with
$\Fun\pr{\cat A,\cat S}$. Under these identifications, the map $\theta$
is adjoint to the composite
\begin{align*}
\mathbf{M}_{\infty}\times\Fun\pr{\cat A,\cat S}\times\cat A & \xrightarrow{\id\times\ev}\mathbf{M}_{\infty}\times\cat S\\
 & \to\mathbf{M}_{\infty}\otimes\cat S\\
 & \xrightarrow{\simeq}\mathbf{M}_{\infty}.
\end{align*}
So the claim follows from Lemma \ref{lem:Fun(A.-)}.

Part (3) follows from the argument in part (2) and the pasting law
of pushouts \cite[Lemma 4.4.2.1]{HTT}. Part (4) is a consequence
of (3). The proof is now complete.
\end{proof}

\subsection{\label{subsec:comp}Composition Product}

Let $\mathbf{M}$ be a symmetric monoidal model category. The composition
product monoidal structure on $\SSeq\pr{\mathbf{M}}$ is generally
\textit{not} closed, because composition product may not preserve
small colimits in the first variable. In particular, $\SSeq\pr{\mathbf{M}}$\textit{
}is generally not a monoidal model category. The goal of this subsection
is to show that the composition product still behaves homotopically
for cofibrant symmetric sequences. More precisely, we prove the following
proposition:
\begin{prop}
\label{prop:sseq_model}Let $\mathbf{M}$ be a cofibrantly generated
symmetric monoidal model category that admits generating sets of cofibrations
and trivial cofibrations whose domains are cofibrant. Then:
\begin{enumerate}
\item For every projectively cofibrant symmetric sequence $X$, the functor
\[
X\circ-\from\SSeq\pr{\mathbf{M}}\to\SSeq\pr{\mathbf{M}}
\]
is left Quillen.
\item For every projectively cofibrant symmetric sequence $Y$, the functor
\[
-\circ Y\from\SSeq\pr{\mathbf{M}}\to\SSeq\pr{\mathbf{M}}
\]
preserves weak equivalences of projectively cofibrant objects.
\end{enumerate}
\end{prop}

For the proof of Proposition \ref{prop:sseq_model}, we need some
preliminaries.
\begin{lem}
\label{lem:tensor_cof}Let $\mathbf{M}$ be a symmetric monoidal model
category, and let $\{f_{i}\from X_{i}\to Y_{i}\}_{i=1,2}$ be cofibrations
of $\mathbf{M}$. If $X_{1}$ and $X_{2}$ are cofibrant, the map
\[
f_{1}\otimes f_{2}\from X_{1}\otimes X_{2}\to Y_{1}\otimes Y_{2}
\]
is a cofibration. If further one of $f_{1}$ or $f_{2}$ is a weak
equivalence, so is $f_{1}\otimes f_{2}$.
\end{lem}

\begin{proof}
We will show that $f_{1}\otimes f_{2}$ is a cofibration; the latter
claim can be proved similarly. We can factor $f_{1}\otimes f_{2}$
as 
\[
X_{1}\otimes X_{2}\xrightarrow{f_{1}\otimes\id}Y_{1}\otimes X_{2}\xrightarrow{\id\otimes f_{2}}Y_{1}\otimes Y_{2}.
\]
The map $f_{1}\otimes\id$ is a cofibration by the pushout-product
axiom (applied to the cofibrations $f_{1}$ and $\emptyset\to X_{2}$).
Likewise, $\id\otimes f_{2}$ is a cofibration. Hence $f_{1}\otimes f_{2}$
is the composite of two cofibrations, and we are done.
\end{proof}

To state our next lemma, we need to introduce some notation.
\begin{notation}
For each $A\in\FB$, we let $\Sigma_{A}$ denote the automorphism
group of $A$, and write $B\Sigma_{A}\subset\FB$ for the full subcategory
spanned by $A$. 
\end{notation}

It will also be useful to have the following alternative notation
for $B\Sigma_{A}$:
\begin{notation}
For each $r\geq0$, we let $\FB\pr r\subset\FB$ denote the full subgroupoid
spanned by the (unique) object of cardinality $r$. 
\end{notation}

\begin{lem}
\label{lem:F_r}Let $\mathbf{M}$ be a tractable symmetric monoidal
model category, let $X$ be a projectively cofibrant symmetric sequence
in $\mathbf{M}$, and let $A\in\FB$ be a finite set. For each $r\geq0$,
the functor
\begin{align*}
F_{r}\from\SSeq\pr{\mathbf{M}} & \to\Fun\pr{B\Sigma_{A},\mathbf{M}}\\
Y & \mapsto\colim_{f:A\to B\in\FB\pr r\times_{\Fin}\Fin_{A/}}\pr{\bigotimes_{b\in B}X\pr{f^{-1}\pr b}}\otimes Y\pr B
\end{align*}
is left Quillen.
\end{lem}

\begin{proof}
Since $\mathbf{M}$ is tractable, there are generating sets $I$ and
$J$ of cofibrations and trivial cofibrations with cofibrant domains.
We define sets $I_{\Sigma}$ and $J_{\Sigma}$ of morphisms of $\SSeq\pr{\mathbf{M}}$
by
\begin{align*}
I_{\Sigma} & =\{S_{!}\alpha\mid S\in\FB,\alpha\in I\},\\
J_{\Sigma} & =\{S_{!}\alpha\mid S\in\FB,\alpha\in J\},
\end{align*}
where $S_{!}=\FB\pr{S,-}\otimes-:\mathbf{M}\to\SSeq\pr{\mathbf{M}}$
denotes the left adjoint to the evaluation at $S$. As we saw in Remark
\ref{rem:proj_convention}, the sets $I_{\Sigma}$ and $J_{\Sigma}$
form generating sets of cofibrations and trivial cofibrations of $\SSeq\pr{\mathbf{M}}$.
Since $F_{r}$ preserves small colimits, it suffices to show that
$F_{r}$ carries morphisms in $I_{\Sigma}$ to cofibrations and morphisms
in $J_{\Sigma}$ to trivial cofibrations. In what follows, we will
focus on the case of $I_{\Sigma}$; the proof for $J_{\Sigma}$ can
be treated similarly.

Let $S$ be a finite set, and let $\alpha:P\to P'$ be a morphism
in $I$. We wish to show that the morphism $F_{r}\pr{S_{!}\pr{\alpha}}$
is a cofibration. If the cardinality of $S$ is not equal to $r$,
then $F_{r}\pr{S_{!}\pr{\alpha}}$ is an isomorphism, and we are done.
If $S$ has exactly $r$ elements, then we can identify $F_{r}\pr{S_{!}\pr{\alpha}}$
with the map 
\[
\coprod_{f\in\Fin\pr{A,S}}\pr{\bigotimes_{s\in S}X\pr{f^{-1}\pr s}}\otimes P\to\coprod_{f\in\Fin\pr{A,S}}\pr{\bigotimes_{s\in S}X\pr{f^{-1}\pr s}}\otimes P',
\]
Therefore, it suffices to show that the functor
\begin{align*}
\Phi\from\SSeq\pr{\mathbf{M}}^{S}\times\mathbf{M} & \to\Fun\pr{B\Sigma_{A},\mathbf{M}},\\
\pr{\pr{X_{s}}_{s\in S},N} & \mapsto\coprod_{f\in\Fin\pr{A,S}}\pr{\bigotimes_{s\in S}X_{s}\pr{f^{-1}\pr s}}\otimes N
\end{align*}
carries cofibrations to cofibrations.\textbf{ }As before, we can check
this on the level of generating cofibrations. In other words, it suffices
to show that, for each collection $\pr{T_{s}}_{s\in S}$ of objects
in $\FB$, the composite
\[
\Psi\from\mathbf{M}^{S}\times\mathbf{M}\xrightarrow{\pr{\prod_{s\in S}\pr{T_{s}}_{!}}\times\id_{\mathbf{M}}}\SSeq\pr{\mathbf{M}}^{S}\times\mathbf{M}\xrightarrow{\Phi}\Fun\pr{B\Sigma_{A},\mathbf{M}}
\]
preserves cofibrations. Unwinding the definitions, the functor $\Psi$
is given by
\[
\Psi\pr{M_{1},\dots,M_{r},N}=\FB\pr{\coprod_{s\in S}T_{s},A}\otimes\pr{\pr{\bigotimes_{s\in S}M_{s}}\otimes N}.
\]
Since $\FB\pr{\coprod_{s\in S}T_{s},A}$ is a free $\Sigma_{A}$-set,
the functor $\FB\pr{\coprod_{s\in S}T_{s},A}\otimes-:\mathbf{M}\to\Fun\pr{B\Sigma_{A},\mathbf{M}}$
is left Quillen. Thus, we are reduced to showing that the functor
\[
\bigotimes\from\mathbf{M}^{S}\times\mathbf{M}\to\mathbf{M}
\]
carries cofibrations of cofibrant objects to cofibrations. But this
is the content of Lemma \ref{lem:tensor_cof}, and we are done.
\end{proof}

\begin{lem}
\label{lem:comp_hocolim}Let $\mathbf{M}$ be a tractable symmetric
monoidal model category, and let $X$ and $Y$ be projectively cofibrant
symmetric sequence in $\mathbf{M}$. For each finite set $A$ and
each $r\geq0$, the colimit cone for
\[
\colim_{f:A\to B\in\FB\times_{\Fin}\Fin_{A/}}\pr{\bigotimes_{b\in B}X\pr{f^{-1}\pr b}}\otimes Y\pr B
\]
is a homotopy colimit cone (i.e., its image in $\mathbf{M}[\mathrm{weq}^{-1}]$
is a colimit cone). Moreover, this colimit is cofibrant.
\end{lem}

\begin{proof}
The final claim is a consequence of Lemma \ref{lem:F_r}, because
cofibrant objects are stable under coproducts. For the former, let
$G_{X,Y}\from\FB\times_{\Fin}\Fin_{A/}\to\mathbf{M}$ denote the diagram
\[
\pr{f\from A\to B}\mapsto\pr{\bigotimes_{b\in B}X\pr{f^{-1}\pr b}}\otimes Y\pr B,
\]
and let $p\from\FB\times_{\Fin}\Fin_{A/}\to\FB$ denote the projection.
The left Kan extension $\Lan_{p}G_{X,Y}$ is given by the formula
\[
\Lan_{p}G_{X,Y}\pr B=\coprod_{f\in\Fin\pr{A,B}}\bigotimes_{b\in B}X\pr{f^{-1}\pr b}\otimes Y\pr B.
\]
Since coproducts of cofibrant objects are homotopy coproducts, \cite[\href{https://kerodon.net/tag/02ZM}{Tag 02ZM}]{kerodon}
shows that this is a homotopy left Kan extension. (Note that $X$
and $Y$ are objectwise cofibrant, being projectively cofibrant).
By the transitivity of Kan extensions, we have $\colim_{\FB{}_{A/}}\cong\colim_{\FB}\circ\Lan_{p}$.
Therefore, it suffices to show that the colimit cone for $\colim_{\FB}\Lan_{p}G_{X,Y}$
is a homotopy colimit diagram. We prove this by showing that $\Lan_{p}G_{X,Y}\in\Fun\pr{\FB,\mathbf{M}}$
is projectively cofibrant. In fact, we will prove more strongly that
the functor

\begin{align*}
\SSeq\pr{\mathbf{M}} & \to\Fun\pr{\FB,\mathbf{M}},\\
Y & \mapsto\pr{\Lan_{p}G_{X,Y}}
\end{align*}
is left Quillen. 

Let $B\in\FB$ be an arbitrary object. We must show that the functor
\begin{align*}
\Phi\from\SSeq\pr{\mathbf{M}} & \to\Fun\pr{B\Sigma_{B},\mathbf{M}},\\
Y & \mapsto\coprod_{f\in\Fin\pr{A,B}}\bigotimes_{b\in B}X\pr{f^{-1}\pr b}\otimes Y\pr B
\end{align*}
is left Quillen. By Remark \ref{rem:proj_convention}, it suffices
to show that for each cofibration (resp. trivial cofibration) $\alpha\from M\to M'$
in $\mathbf{M}$ and each $S\in\FB$, the map $\Phi\pr{\FB\pr{S,-}\otimes\alpha}$
is a cofibration (resp. trivial cofibration). We will focus on the
case of cofibrations, because the case of trivial cofibrations can
be dealt with similarly.

If $\abs S\neq\abs B$, then $\Phi\pr{\FB\pr{S,-}\otimes\alpha}$
is an isomorphism, and we are done. If $\abs S=\abs B$, let $\Pi\pr{A,B}$
denote the set of connected components of $B\Sigma_{B}\times_{\Fin}\Fin_{A/}$.
For each $\pi\in\Pi\pr{A,B}$, we will write $\Fin^{\pi}\pr{A,B}\subset\Fin\pr{A,B}$
for the set of maps $A\to B$ lying in $\pi$. We also choose a representative
$f_{\pi}$ of each $\pi\in\Pi\pr{A,B}$ and set $X_{\pi}=\bigotimes_{b\in B}X\pr{f^{-1}_{\pi}\pr b}$.
We then have a $\Sigma_{B}$-equivariant isomorphism
\[
\Phi\pr{\FB\pr{S,-}\otimes M}\cong\coprod_{\pi\in\Pi\pr{A,B}}\pr{\pr{\Fin^{\pi}\pr{A,B}\times\FB\pr{S,B}}\otimes X_{\pi}\otimes M}
\]
natural in $M$. Since the $\Sigma_{B}$-set $\Fin^{\pi}\pr{A,B}\times\FB\pr{S,B}$
is free, this isomorphism tells us that $\Phi\pr{\FB\pr{S,-}\otimes\alpha}$
can be identified with the image of a cofibration of $\mathbf{M}$
under the left Kan extension functor $\mathbf{M}\to\Fun\pr{B\Sigma_{B},\mathbf{M}}$.
Since the latter is a left Quillen functor, we have shown that $\Phi\pr{\FB\pr{S,-}\otimes\alpha}$,
as desired.
\end{proof}

We now arrive at the proof of Proposition \ref{prop:sseq_model}.
\begin{proof}
[Proof of Proposition \ref{prop:sseq_model}]For (1), we must show
that, for each finite set $A$, the composite
\[
F\from\SSeq\pr{\mathbf{M}}\xrightarrow{X\circ-}\SSeq\pr{\mathbf{M}}\xrightarrow{\text{restriction}}\Fun\pr{B\Sigma_{A},\mathbf{M}}
\]
is left Quillen. This follows from Lemma \ref{lem:F_r}, because $F$
is the coproduct of the functors $\{F_{r}\}_{r\geq0}$. Part (2) follows
from \ref{lem:comp_hocolim}, and we are done.
\end{proof}

\section{\label{sec:brantner}Closer look at Brantner's model}

The goal of this section is to characterize Brantner's model by a
universal property. As we saw in Proposition \ref{prop:sseq_model},
for every tractable symmetric monoidal model category $\mathbf{M}$,
the full subcategory $\SSeq\pr{\mathbf{M}}_{\cof}\subset\SSeq\pr{\mathbf{M}}$
of projectively cofibrant objects is a monoidal category itself, and
its tensor product preserves weak equivalences in each variable. In
particular, it admits a monoidal localization. Our main theorem asserts
that Brantner's model is naturally equivalent to this monoidal localization:
\begin{thm}
\label{thm:Brantner_univ}There is a natural equivalence 
\[
\SSeq\pr{\mathbf{M}}_{\cof}[\weq^{-1}]^{\circ}\simeq\SSeq_{\B}\pr{\mathbf{M}_{\infty}}^{\circ}
\]
of functors $\TSMMC\to\Mon\hat{\Cat}_{\infty}$.
\end{thm}

The remainder of this section is devoted to the proof of Theorem \ref{thm:Brantner_univ}. 
\begin{notation}
Let $\mathbf{M}$ be a combinatorial symmetric monoidal model category.
We write $\mathbf{M}[\mathfrak{X}]=\Fun\pr{\FB,\mathbf{M}}$, and
write $\bigstar$ for the (ordinary) Day convolution product in $\mathbf{M}[\mathfrak{X}]$.
We also write $\mathfrak{X}=\pr{\emptyset,\mathbf{1},\emptyset,\emptyset,\dots}\in\mathbf{M}[\mathfrak{X}]$
for the unit symmetric sequence, and $i_{\mathbf{M}}\from\mathbf{M}\to\mathbf{M}[\mathfrak{X}]$
for the functor $M\mapsto\pr{M,\emptyset,\emptyset,\dots}$. (This
set of notation is justified by Corollary \ref{cor:ordinary_day}.)
\end{notation}

\begin{defn}
\label{def:empty}We let $\TSMMC_{\emptyset}\subset\TSMMC$ denote
the full subcategory spanned by the objects with \textit{exactly one}
initial object, denoted by $\emptyset$. 
\end{defn}

\begin{rem}
Here is the motivation for Definition \ref{def:empty}. In the rest
of this section, we will frequently be considering diagrams of the
form 
\[\begin{tikzcd}
	{\mathbf{M}} & {\mathbf{N}} \\
	{\mathbf{M}[\mathfrak{X}]} & {\mathbf{N}[\mathfrak{X}],}
	\arrow["F", from=1-1, to=1-2]
	\arrow["{i_{\mathbf{M}}}"', hook, from=1-1, to=2-1]
	\arrow["{i_{\mathbf{N}}}", hook, from=1-2, to=2-2]
	\arrow["{F_\ast}"', from=2-1, to=2-2]
\end{tikzcd}\]where $F$ is a left Quillen symmetric monoidal functor and $i_{\mathbf{M}},i_{\mathbf{N}}$
are defined in Notation \ref{nota:symseq_day}. The diagram may not
commute on the nose, but it does when $\mathbf{N}$ has only one initial
object.
\end{rem}

\begin{rem}
\label{rem:empty}Theorem \ref{thm:delocPrSM} remains valid if we
replace $\TSMMC$ by $\TSMMC_{\emptyset}$. Indeed, the argument of
\cite[Proposition 1.4]{arakawa_pres_smcat_comb_smcat} shows that
the $\pr{2,1}$-categorical enhancements of $\TSMMC_{\emptyset}$
and $\TSMMC$ are localizations of $\TSMMC$ and $\TSMMC_{\emptyset}$,
and these $\pr{2,1}$-categorical enhancements are equivalent. Since
the maps inverted by these localizations are contained in the class
of symmetric monoidal left Quillen equivalences, this implies that
the inclusion $\TSMMC_{\emptyset}\hookrightarrow\TSMMC$ induces an
equivalence upon localizing at symmetric monoidal left Quillen equivalences.
\end{rem}

\begin{construction}
The category $\TSMMC_{\emptyset}$ admits a natural $2$-categorical
enhancement $\TSMMC^{\pr 2}_{\emptyset}$. Explicitly, given a pair
of objects $\mathbf{M}$ and $\mathbf{N}$, the mapping categories
from $\mathbf{M}$ to $\mathbf{N}$ are given by the full subcategory
$\Fun^{\t,LQ}\pr{\mathbf{M},\mathbf{N}}\subset\Fun^{\t}\pr{\mathbf{M},\mathbf{N}}$
of left Quillen symmetric monoidal functors. 

If $\mathbf{M}$ and $\mathbf{N}$ are equipped with a left Quillen
symmetric monoidal functors $F\from\mathbf{A}\to\mathbf{M}$ and $G\from\mathbf{A}\to\mathbf{N}$,
we define another category $\Fun^{\t,LQ}_{\mathbf{A}/}\pr{\mathbf{M},\mathbf{N}}$
by the (strict) pullback
\[\begin{tikzcd}
	{\operatorname{Fun}^{\otimes, LQ}_{\mathbf{A}/}(\mathbf{M},\mathbf{N})} & {\operatorname{Fun}^{\otimes, LQ}(\mathbf{M},\mathbf{N})} \\
	{\{\ast\}} & {\operatorname{Fun}^{\otimes, LQ}(\mathbf{A},\mathbf{N}).}
	\arrow[from=1-1, to=1-2]
	\arrow[from=1-1, to=2-1]
	\arrow[from=1-2, to=2-2]
	\arrow[from=2-1, to=2-2]
\end{tikzcd}\]We will often consider the category $\Fun^{\t,LQ}_{\mathbf{A}/}\pr{\mathbf{M},\mathbf{N}}$
when $F$ is injective on objects and morphisms; in this case, the
above square is homotopy cartesian.
\end{construction}

\begin{prop}
\label{prop:symseq_cof}For every morphism $i\from\mathbf{M}\to\mathbf{N}$
in $\TSMMC_{\emptyset}$, the evaluation at the unit symmetric sequence
$\mathfrak{X}\in\mathbf{M}[\mathfrak{X}]$ induces a categorical equivalence
\[
\theta\from\Fun^{\t,LQ}_{\mathbf{M}/}\pr{\mathbf{M}[\mathfrak{X}],\mathbf{N}}\xrightarrow{\simeq}\mathbf{N}_{\cof}.
\]
\end{prop}

\begin{proof}
Using Proposition \ref{prop:SSeq_B_underlying} and Corollary \ref{cor:ordinary_day},
we deduce that the evaluation at $\mathfrak{X}$ gives an equivalence
of categories
\[
\theta'\from\Fun^{\t,L}_{\mathbf{M}/}\pr{\mathbf{M}[\mathfrak{X}],\mathbf{N}}\xrightarrow{\simeq}\mathbf{N}.
\]
Therefore, it suffices to show that a functor $F\in\Fun^{\t,L}_{\mathbf{M}/}\pr{\mathbf{M}[\mathfrak{X}],\mathbf{N}}$
is left Quillen if and only if $\theta'\pr F$ is cofibrant. 

Necessity is obvious, because $E$ is cofibrant. For sufficiency,
suppose that $\theta'\pr F$ is cofibrant. We must show that $F$
preserves cofibrations and trivial cofibrations. We will show that
it preserves cofibrations; the argument for trivial cofibrations will
be similar. Recall from Remark \ref{rem:proj_convention} that the
class of cofibrations of $\Fun\pr{\FB,\mathbf{M}}$ is generated by
the maps of the form
\[
f\from\FB\pr{S,-}\cdot M\to\FB\pr{S,-}\cdot M',
\]
where $S\in\FB$ is a finite set and $M\to M'$ is a cofibration of
$\mathbf{M}$. We can rewrite $f$ as
\[
\mathfrak{X}^{\bigstar S}\star i_{\mathbf{M}}\pr M\to\mathfrak{X}^{\bigstar S}\star i_{\mathbf{M}}\pr{M'}.
\]
Since $F$ is symmetric monoidal and is compatible with the restrictions
to $\mathbf{M}$, the morphism $F\pr f$ can be identified with the
map
\[
F\pr{\mathfrak{X}}^{\otimes S}\otimes i\pr M\to F\pr{\mathfrak{X}}^{\otimes S}\otimes i\pr{M'}.
\]
This is a cofibration because $i$ is left Quillen and $\theta'\pr F=F\pr{\mathfrak{X}}$
is cofibrant. Hence $F$ carries generating cofibrations to cofibrations.
It follows that $F$ carries all cofibrations to cofibrations, as
claimed.
\end{proof}

\begin{notation}
We write:
\begin{itemize}
\item $\mathsf{Mon}\hat{\mathsf{Cat}}$ for the category of large monoidal
categories and monoidal functors;
\item $\mathsf{Bi}\hat{\mathsf{Cat}}$ for the category of large bicategories; 
\item $B\from\mathsf{Mon}\hat{\mathsf{Cat}}\hookrightarrow\mathsf{Bi}\hat{\mathsf{Cat}}$
for the inclusion; and
\item $\mathcal{B}\mathsf{i}\widehat{\Cat}_{\infty}$ for the $\infty$-category
of large $\infty$-bicategories.
\end{itemize}
\label{not:X}We write $\int^{\TSMMC_{\emptyset}}B\SSeq\pr -^{\circ}_{\cof}$
for the full sub $\infty$-bicategory of the fiber product
\[
\Fun^{\sc}\pr{[1],\TSMMC^{\pr 2}_{\emptyset}}\times_{\Fun\pr{\{0\},\TSMMC^{\pr 2}_{\emptyset}}}\TSMMC_{\emptyset}
\]
spanned by the objects $\{i_{\mathbf{M}}\from\mathbf{M}\to\mathbf{M}[\mathfrak{X}]\}_{\mathbf{M}\in\TSMMC_{\emptyset}}$.
\end{notation}

Notation \ref{not:X} is justified by the following proposition:
\begin{prop}
\label{prop:globalTrimble}Let $p\from\int^{\TSMMC_{\emptyset}}B\SSeq\pr -^{\circ}_{\cof}\to\TSMMC_{\emptyset}$
be as in Notation \ref{not:X}.
\begin{enumerate}
\item For every map $F\from\mathbf{M}\to\mathbf{N}$ in $\TSMMC_{\emptyset}$,
the square
\[\begin{tikzcd}
	{\mathbf{M}} & {\mathbf{N}} \\
	{\mathbf{M}[\mathfrak{X}]} & {\mathbf{N}[\mathfrak{X}]}
	\arrow["F", from=1-1, to=1-2]
	\arrow["{i_{\mathbf{M}}}"', hook, from=1-1, to=2-1]
	\arrow["{i_{\mathbf{N}}}", hook, from=1-2, to=2-2]
	\arrow["{F_\ast}"', from=2-1, to=2-2]
\end{tikzcd}\]is $p$-cocartesian.
\item The functor $p$ is a cocartesian fibration of $\infty$-bicategories
which straightens to the composite
\[
\Phi\from\TSMMC_{\emptyset}\xrightarrow{\pr{\SSeq\pr -_{\cof},\circ}}\mathsf{Mon}\hat{\mathsf{Cat}}\xrightarrow{B}\mathsf{Bi}\widehat{\mathsf{Cat}}\hookrightarrow\mathcal{B}\mathsf{i}\widehat{\Cat}_{\infty}.
\]
\end{enumerate}
\end{prop}

\begin{proof}
We start with (1). Let $\cat X$ denote the full sub $2$-category
of 
\[
\widetilde{\cat X}=2\Fun\pr{[1],\TSMMC^{\pr 2}_{\emptyset}}\times_{2\Fun\pr{\{0\},\TSMMC^{\pr 2}_{\emptyset}}}\TSMMC_{\emptyset}
\]
spanned by the objects $\{i_{\mathbf{M}}\from\mathbf{M}\to\mathbf{M}[\mathfrak{X}]\}_{\mathbf{M}\in\TSMMC_{\emptyset}}$,
where $2\Fun\pr{-,-}$ denotes the category of $2$-functors and $2$-natural
transformations (i.e., $\mathsf{Cat}$-enriched natural transformations).
By part (2) of Proposition \ref{prop:mappingcat_arrow}, the functor
$\cat X\to\int^{\TSMMC_{\emptyset}}B\SSeq\pr -^{\circ}_{\cof}$ is
an equivalence. It will therefore suffice to show that the square
in the statement is cocartesian for the projection $\widetilde{\cat X}\to\TSMMC_{\emptyset}$.

Given a left Quillen symmetric monoidal functor $G\from\mathbf{P}\to\mathbf{Q}$,
the mapping category of $\widetilde{\cat X}$ is given by
\begin{align*}
\cat X\pr{i_{0,\mathbf{M}},G} & =\coprod_{\mathbf{M}\to\mathbf{P}\in\Fun^{\t,LQ}\pr{\mathbf{M},\mathbf{P}}}\Fun^{\t,LQ}_{\mathbf{M}/}\pr{\mathbf{M}[\mathfrak{X}],\mathbf{Q}}.
\end{align*}
Thus, our goal is to show that, for each left Quillen symmetric monoidal
functor $\mathbf{N}\to\mathbf{P}$, the map
\[
\Fun^{\t,LQ}_{\mathbf{N}/}\pr{\mathbf{N}[\mathfrak{X}],\mathbf{Q}}\to\Fun^{\t,LQ}_{\mathbf{M}/}\pr{\mathbf{M}[\mathfrak{X}],\mathbf{Q}}
\]
is an equivalence. This follows from Proposition \ref{prop:symseq_cof},
which identifies this map with $\id\from\mathbf{Q}_{\cof}\to\mathbf{Q}_{\cof}$.

We next turn to (2). Part (1) implies that $p$ is a cocartesian fibration.
To identify its straightening, let $\int\Phi\to\TSMMC_{\emptyset}$
denote the bicategorical Grothendieck construction (Variant \ref{var:bicategoricalgroth})
of $\Phi$. Explicitly, objects of $\int\Phi$ are the objects of
$\TSMMC_{\emptyset}$, and the mapping categories are defined by
\[
\pr{\int\Phi}\pr{\mathbf{M},\mathbf{N}}=\coprod_{\mathbf{M}\to\mathbf{N}\in\Fun^{\t,LQ}\pr{\mathbf{M},\mathbf{N}}}\SSeq\pr{\mathbf{N}}_{\cof}.
\]
So a $1$-morphism $\mathbf{M}\to\mathbf{N}$ is a pair $\pr{F,X}$,
where $F\from\mathbf{M}\to\mathbf{N}$ is a left Quillen symmetric
monoidal functor, and $X$ is a projectively cofibrant symmetric sequence
in $\mathbf{N}$. Composition of such are defined by
\[
\pr{G,Y}\circ\pr{F,X}=\pr{GF,Y\circ\pr{F_{\ast}X}}.
\]
$2$-morphisms are defined and composed in a similar way. By Variant
\ref{var:bicategoricalgroth}, it suffices to produce a strictly unitary
equivalence $\cat X\xrightarrow{\simeq}\int\Phi$ of bicategories.
Such an equivalence is obtained from Proposition \ref{prop:symseq_cof}
by arguing as in Proposition \ref{prop:carboni}.
\end{proof}

\begin{lem}
\label{rem:simplicial_realization_of_loc}We can construct the symmetric
monoidal localization functor $L\from\SM\Rel\Cat^{\pr 2}_{\infty}\to\SM\Cat^{\pr 2}_{\infty}$
and the natural transformation $\eta=\{\pr{\cat C^{\t},\cat W}\to\cat C[\cat W^{-1}]^{\t}\}_{\pr{\cat C^{\t},\cat W}\in\SM\Rel\Cat^{\pr 2}_{\infty}}$
of Construction \ref{const:monoidalloc} so that they satisfy the
following properties:
\begin{enumerate}
\item $L$ is the scaled nerve of an $\SS^{+}$-enriched functor.
\item $\eta$ is the scaled nerve of an $\SS^{+}$-enriched natural transformation.
\item If $\pr{\cat C^{\t},\cat W_{\cat C}}\to\pr{\cat D^{\t},\cat W_{\cat D}}$
is a morphism of symmetric monoidal relative categories such that
$\cat C^{\t}\to\cat D^{\t}$ is a monomorphism of simplicial sets,
then the map $\cat C[\cat W^{-1}_{\cat C}]^{\t}\to\cat D[\cat W^{-1}_{\cat D}]^{\t}$
is a monomorphism of simplicial sets.
\end{enumerate}
\end{lem}

\begin{proof}
The key ingredient is the enriched version of Quillen's small object
argument \cite[Theorem 13.2.1]{cathtpy}, which we now recall. Let
$\Fin^{\sharp}_{\ast}$ denote the marked simplicial set obtained
from (the nerve of) $\Fin_{\ast}$ by marking all of its morphisms.
According to \cite[Variant 2.1.4.13, Remark B.2.5]{HA}, the category
$\SS^{+}_{/\Fin^{\sharp}_{\ast}}$ admits a combinatorial $\SS^{+}$-enriched
model structure with the following properties:
\begin{itemize}
\item Its cofibrations are the monomorphisms; and 
\item Its fibrant objects are the symmetric monoidal $\infty$-categories
with cocartesian edges marked.
\end{itemize}

Now fix a generating set $J$ of trivial cofibrations of $\SS^{+}_{/\Fin_{\ast}}$,
and choose a regular cardinal $\kappa$ such that the domain and codomain
of the maps in $J$ are all $\kappa$-compact. We define a $\kappa$-sequence
$\id=L_{0}\To L_{1}\To\cdots$ of enriched endofunctors of $\SS^{+}_{/\Fin_{\ast}}$
and enriched natural transformations inductively as follows: For each
limit ordinal $\lambda$, we set $L_{\lambda}=\colim_{\alpha<\lambda}L_{\alpha}$.
Assuming that $L_{\alpha}$ has been defined for some $\alpha<\kappa$,
we define $L_{\alpha+1}$ by the pushout 
\[\begin{tikzcd}
	{\coprod_{(j\colon A_j \to B_j)\in J}[A_j,L_{\alpha}(X)]\times A_j} & {L_{\alpha }(X)} \\
	{\coprod_{(j\colon A_j \to B_j)\in J}[A_j,L_{\alpha}(X)]\times B_j} & {L_{\alpha+1}(X),}
	\arrow[from=1-1, to=1-2]
	\arrow[from=1-1, to=2-1]
	\arrow[from=1-2, to=2-2]
	\arrow[from=2-1, to=2-2]
\end{tikzcd}\]where $[-,-]$ denotes the mapping objects of the $\SS^{+}$-enriched
category $\SS^{+}_{/\Fin_{\ast}}$. We then set $L_{\kappa}=\colim_{\alpha<\kappa}L_{\alpha}$
and write $\eta_{\kappa}\from\id_{\SS^{+}_{/\Fin_{\ast}}}\To L_{\kappa}$
for the resulting enriched natural transformation. By construction,
the components of $\eta_{\kappa}$ are trivial cofibrations of $\SS^{+}_{/\Fin_{\ast}}$,
and $L$ takes values in the full subcategory of fibrant objects.
Moreover, $L$ preserves monomorphisms because the functor $[A,-]$
preserves monomorphisms for all $A\in\SS^{+}_{/\Fin_{\ast}}$.

To relate this construction to the setting at hand, let $\cat X$
and $\cat Y$ denote the $\SS^{+}$-enriched categories defining $\SM\Rel\Cat^{\pr 2}_{\infty}$
and $\SM\Cat^{\pr 2}_{\infty}$. For each $\pr{\cat C^{\t},\cat W_{\cat C}}\in\cat X$,
let $\cat C^{\t,\cat W_{\cat C}}$ denote the marked simplicial set
obtained from $\cat C^{\t}$ by marking the cocartesian edges and
the equivalences in $\cat W_{\cat C}$. The assignment $\pr{\cat C^{\t},\cat W_{\cat C}}\mapsto\cat C^{\t,\cat W_{\cat C}}$
determines an enriched functor
\[
\cat X\to\SS^{+}_{/\Fin_{\ast}}.
\]
In general, given a pair of objects $\pr{\cat C^{\t},\cat W_{\cat C}},\pr{\cat D^{\t},\cat W_{\cat D}}\in\cat X$
the map 
\[
\cat X\pr{\pr{\cat C^{\t},\cat W_{\cat C}},\pr{\cat D^{\t},\cat W_{\cat D}}}\to[\cat C^{\t,\cat W_{\cat C}},\cat D^{\t,\cat W_{\cat D}}]
\]
is \textit{not} an isomorphism, but it is an isomorphism when $\pr{\cat D^{\t},\cat W_{\cat D}}\in\cat Y$
(i.e., $\cat W_{\cat D}$ are the equivalences of $\cat D$). Therefore,
$L_{\kappa}$ restricts to a functor $L\from\cat X\to\cat Y$, and
$\eta$ restricts to an enriched natural transformation $\eta\from\id_{\cat X}\To L$.
The desired functor and natural transformation is given by taking
the scaled nerve of $L$ and $\eta$.
\end{proof}

\begin{proof}
[Proof of Theorem \ref{thm:Brantner_univ}]By Remark \ref{rem:empty},
we may replace $\TSMMC$ by $\TSMMC_{\emptyset}$. We will first construct
a natural transformation $\theta\from\SSeq\pr -^{\circ}\to\SSeq_{B}\pr{\pr -_{\infty}}^{\circ}$
of functors $\TSMMC_{\emptyset}\to\Mon\hat{\Cat}_{\infty}$. Define
a functor $\tau\from\TSMMC_{\emptyset}\to\int^{\TSMMC_{\emptyset}}B\SSeq\pr -^{\circ}$
by
\[
\tau\pr{\mathbf{M}}=\pr{i_{\mathbf{M}}\from\mathbf{M}\hookrightarrow\mathbf{M}[\mathfrak{X}]}.
\]
By Proposition \ref{prop:globalTrimble}, $\tau$ is a cocartesian
section of the projection $\int^{\TSMMC_{\emptyset}}B\SSeq\pr -^{\circ}\to\TSMMC_{\emptyset}$.
Postcompositon with the functor $\pr -^{\t}_{\infty}\from\TSMMC^{\pr 2}_{\emptyset}\to\Pr\SM^{\pr 2}$
determines a functor
\[
F\from\int^{\TSMMC_{\emptyset}}B\SSeq\pr -^{\circ}\to\int^{\Pr\SM}\CAlg_{\bullet}
\]
Part (3) of Proposition \ref{prop:Day_model} shows that $F$ preserves
cocartesian edges. It follows from that the diagram 
\[\begin{tikzcd}
	{\int^{\mathsf{TractSMMC}_{\emptyset}}B\Sigma\mathrm{Seq}(-)^\circ} & {\int^{\mathcal{P}\mathsf{r}\mathcal{SM}}\operatorname{CAlg}^{(2)}_{\bullet}} \\
	{\mathsf{TractSMMC}_{\emptyset}} & {\mathcal{P}\mathsf{r}\mathcal{SM}}
	\arrow["F", from=1-1, to=1-2]
	\arrow["\tau", from=2-1, to=1-1]
	\arrow["{(-)_{\infty}^\otimes}"', from=2-1, to=2-2]
	\arrow["\sigma"', from=2-2, to=1-2]
\end{tikzcd}\]commutes up to natural equivalence, where $\sigma$ is the section
defined in Remark \ref{rem:Brantner_functorial}. By straightening,
this gives rise to a natural transformation
\[
\theta\from\SSeq\pr -^{\circ}\to\SSeq_{B}\pr{\pr -_{\infty}}^{\circ}.
\]

To complete the proof, we must show that for each $\mathbf{M}\in\TSMMC_{\emptyset}$,
the map
\[
\theta_{\mathbf{M}}\from\SSeq\pr{\mathbf{M}}\to\SSeq_{B}\pr{\mathbf{M}_{\infty}}
\]
is a localization at weak equivalences. (See Remark \ref{rem:loc_unique}.)
For this, we will construct the functor $\pr -^{\t}_{\infty}$ by
using Lemma \ref{rem:simplicial_realization_of_loc}; this will ensure
that $\mathbf{M}^{\t}_{\infty}\to\Fun\pr{\FB,\mathbf{M}}^{\t}_{\infty}$
is a monomorphism of simplicial sets for all $\mathbf{M}\in\TSMMC_{\emptyset}$.
By part (2) of Proposition \ref{prop:mappingcat_arrow}, we can identify
$\theta_{\mathbf{M}}$ with the dashed arrow in the diagram 
\[\begin{tikzcd}[column sep = small, scale cd = .8]
	& {\operatorname{Fun}^{\otimes,L}_{\mathbf{M}_\infty/}(\mathbf{M}[\mathfrak{X}]_\infty,\mathbf{M}[\mathfrak{X}]_\infty)} &&& {\{(i_{\mathbf{M}})_\infty^{\otimes}\}} \\
	{\operatorname{Fun}^{\otimes,LQ}_{\mathbf{M}/}(\mathbf{M}[\mathfrak{X}],\mathbf{M}[\mathfrak{X}])} && {\{i_{\mathbf{M}}\}} \\
	& {\operatorname{Fun}^{\otimes,L}(\mathbf{M}[\mathfrak{X}]_\infty,\mathbf{M}[\mathfrak{X}]_\infty)} &&& {\operatorname{Fun}^{\otimes,L}(\mathbf{M}_\infty,\mathbf{M}[\mathfrak{X}]_\infty)} \\
	{\operatorname{Fun}^{\otimes,LQ}(\mathbf{M}[\mathfrak{X}],\mathbf{M}[\mathfrak{X}])} && {\operatorname{Fun}^{\otimes,LQ}(\mathbf{M},\mathbf{M}[\mathfrak{X}]),}
	\arrow[from=1-2, to=1-5]
	\arrow[from=1-2, to=3-2]
	\arrow[from=1-5, to=3-5]
	\arrow["{\theta_{\mathbf{M}}}", dashed, from=2-1, to=1-2]
	\arrow[from=2-1, to=2-3]
	\arrow[from=2-1, to=4-1]
	\arrow[from=2-3, to=1-5]
	\arrow[from=2-3, to=4-3]
	\arrow[from=3-2, to=3-5]
	\arrow[from=4-1, to=3-2]
	\arrow[from=4-1, to=4-3]
	\arrow[from=4-3, to=3-5]
\end{tikzcd}\]where the front and back faces are defined by strict pullback (but
which happen to be homotopy pullback), and the slanted arrows of the
bottom face is induced by the functor $\SS^{+}$-enriched functor
$\pr -^{\t}_{\infty}$. We now consider the following diagram:
\[\begin{tikzcd}
	{\operatorname{Fun}^{\otimes,LQ}_{\mathbf{M}/}(\mathbf{M}[\mathfrak{X}],\mathbf{M}[\mathfrak{X}])} & {\operatorname{Fun}^{\otimes,L}_{\mathbf{M}_\infty/}(\mathbf{M}[\mathfrak{X}]_\infty,\mathbf{M}[\mathfrak{X}]_\infty)} \\
	{\operatorname{Fun}^{\otimes,LQ}(\mathbf{M}[\mathfrak{X}],\mathbf{M}[\mathfrak{X}])} & {\operatorname{Fun}^{\otimes,L}(\mathbf{M}[\mathfrak{X}]_\infty,\mathbf{M}[\mathfrak{X}]_\infty)} \\
	{\operatorname{Fun}^{\otimes,\mathrm{rel}}(\mathbf{M}[\mathfrak{X}]_{\mathrm{cof}},\mathbf{M}[\mathfrak{X}]_{\mathrm{cof}})} & {\operatorname{Fun}^{\otimes}(\mathbf{M}[\mathfrak{X}]_\infty,\mathbf{M}[\mathfrak{X}]_\infty)} \\
	{\operatorname{Fun}^{\otimes,\mathrm{rel}}(\mathsf{FB},\mathbf{M}[\mathfrak{X}]_{\mathrm{cof}})} & {\operatorname{Fun}^{\otimes}(L(\mathsf{FB}),\mathbf{M}[\mathfrak{X}]_{\infty})} \\
	& {\operatorname{Fun}^{\otimes}(\mathsf{FB},\mathbf{M}[\mathfrak{X}]_{\infty})} \\
	{\operatorname{Fun}(\mathsf{FB},\mathbf{M})_{\mathrm{cof}}} & {\operatorname{Fun}(\mathsf{FB},\mathbf{M})_{\infty}.}
	\arrow["{\theta_{\mathbf{M}}}", dashed, from=1-1, to=1-2]
	\arrow[from=1-1, to=2-1]
	\arrow[from=1-2, to=2-2]
	\arrow["{(-)_\infty^{\otimes}}", from=2-1, to=2-2]
	\arrow[from=2-1, to=3-1]
	\arrow[from=2-2, to=3-2]
	\arrow["L", from=3-1, to=3-2]
	\arrow[from=3-1, to=4-1]
	\arrow[from=3-2, to=4-2]
	\arrow["L", from=4-1, to=4-2]
	\arrow["{(\eta_{\operatorname{Fun}(\mathsf{FB},\mathbf{M})_{\mathrm{cof}}})_\ast}"', from=4-1, to=5-2]
	\arrow["\epsilon"', from=4-1, to=6-1]
	\arrow["{\eta_{\mathsf{FB}}^*}", from=4-2, to=5-2]
	\arrow["\epsilon", from=5-2, to=6-2]
	\arrow["{\eta_{\operatorname{Fun}(\mathsf{FB},\mathbf{M})_{\mathrm{cof}}}}"', from=6-1, to=6-2]
\end{tikzcd}\]Here $\FB$ is regarded as a symmetric monoidal relative $\infty$-category
whose weak equivalences are exactly the isomorphisms, the symmetric
monoidal functor $\FB\to\Fun\pr{\FB,\mathbf{M}}$ carries the singleton
$\underline{1}$ to the unit symmetric sequence $\pr{\emptyset,\mathbf{1},\emptyset,\emptyset,\dots}$,
and $\epsilon$ is the evaluation at the singleton. The triangle commutes
by the enriched naturality of $\eta$, and the remaining squares commute
trivially. The vertical composites are equivalences by Propositions
\ref{prop:SSeq_B_underlying} and \ref{prop:symseq_cof}. Therefore,
we are reduced to showing that the bottom horizontal arrow is a localization,
which it is by construction. The proof is now complete.
\end{proof}

\section{\label{sec:haugseng}Closer look at Haugseng's model}

In this section, we prove the following analog of Theorem \ref{thm:Brantner_univ}
for Haugseng's composition product monoidal structure:
\begin{thm}
\label{thm:Haugseng_univ}There is a natural equivalence 
\[
\SSeq\pr{\mathbf{M}}^{\circ}_{\cof}[\weq^{-1}]\simeq\SSeq_{\H}\pr{\mathbf{M}_{\infty}}^{\circ}
\]
of functors $\TSMMC\to\Mon\hat{\Cat}_{\infty}$.
\end{thm}

The remainder of this section is devoted to the proof of Theorem \ref{thm:Haugseng_univ}. 

Our proof will rely on an alternative construction of Haugseng's model,
which can take arbitrary symmetric monoidal $\infty$-categories as
its input and has $\infty$-operads as its output. (In the $1$-categorical
setting, Ching developed a closely related idea in \cite{Ching12}.)
To describe the construction, we need to introduce a bit of notation.
\begin{notation}
\hfill
\begin{itemize}
\item We write $\Sigma\subset\Fun\pr{[1],\Del^{\op}}$ for the full subcategory
spanned by the inert maps. The evaluation at $0\in[1]$ determines
a cartesian fibration $\ev_{0}\from\Sigma\to\Del^{\op}$, whose fiber
over $[n]\in\Del^{\op}$ can be identified with the poset of subintervals
of $[n]$ ordered by reverse inclusion. 
\item We then define a category $\Sigma\Del^{\op}_{\mathbb{F}}$ and functors
$\pi_{\Del^{\op}},\pi_{\Fin_{\ast}}$ by the commutative diagram
\[\begin{tikzcd}
	{\Sigma \mathbf{\Delta}_{\mathbb{F}}^{\mathrm{op}}} & \Sigma & {\mathbf{\Delta}^{\mathrm{op}}} \\
	{\mathbf{\Delta}^{\mathrm{op}}_{\mathbb{F}}} & {\mathbf{\Delta}^{\mathrm{op}}} \\
	{\mathsf{Fin}_\ast}
	\arrow[from=1-1, to=1-2]
	\arrow["{\pi_{\mathbf{\Delta}^{\mathrm{op}}}}", curve={height=-18pt}, from=1-1, to=1-3]
	\arrow[from=1-1, to=2-1]
	\arrow["\lrcorner"{description, pos=0}, draw=none, from=1-1, to=2-2]
	\arrow["{\pi_{\mathsf{Fin}_\ast}}"', curve={height=24pt}, from=1-1, to=3-1]
	\arrow["{\mathrm{ev}_0}", from=1-2, to=1-3]
	\arrow["{\mathrm{ev}_1}", from=1-2, to=2-2]
	\arrow[from=2-1, to=2-2]
	\arrow["V", from=2-1, to=3-1]
\end{tikzcd}\]whose top square is a pullback. We will think of $\Sigma\Del^{\op}_{\mathbb{F}}$
as lying over $\Fin_{\ast}$ and $\Del^{\op}$ by these maps.
\end{itemize}

We typically denote an object $\iota\from[n]\to[m]$ of $\Sigma$
by $[n]\to[a,b]$, where $a=\iota\pr 0$ and $b=\iota\pr m$. With
this notation, a typical object of $\Sigma\Del^{\op}_{\mathbb{F}}$
can be written as a pair
\[
\pr{[n]\to[a,b],S_{a}\to\cdots\to S_{b}}
\]
where $S_{a}\to\cdots\to S_{b}$ is a sequence of morphsims in $\Fin$.
\end{notation}

\begin{rem}
\label{rem:flat_cfib}The functor $\pi_{\Del^{\op}}$ is a flat categorical
fibration, since it is the composite of a cocartesian fibration and
a cartesian fibration. 
\end{rem}

\begin{notation}
Given a simplicial set $X\in\SS_{/\Del^{\op}}$ and a map $f\from K\to\Del^{\op}$
of simplicial sets, we write $X_{f}=X_{K}=K\times_{\Del^{\op}}X$.
\end{notation}

\begin{defn}
\label{def:SSeq_H_2}Let $p\from\cat O^{\t}\to\Fin_{\ast}$ be an
$\infty$-operad. We define a functor $\widetilde{\SSeq_{\H}}\pr{\cat O}^{\circ}\to\Del^{\op}$
to be the image of $\cat O^{\t}\in\SS_{/\Fin_{\ast}}$ under the composite
\[
\SS_{/\Fin_{\ast}}\xrightarrow{\pr{\pi_{\Fin_{\ast}}}^{*}}\SS_{/\Sigma\Del^{\op}_{\mathbb{F}}}\xrightarrow{\pr{\pi_{\Del^{\op}}}_{\ast}}\SS_{/\Del^{\op}}.
\]
In other words, $\widetilde{\SSeq_{\H}}\pr{\cat O}^{\circ}$ is characterized
by the universal property
\[
\Fun_{\Del^{\op}}\pr{K,\widetilde{\SSeq_{\H}}\pr{\cat O}^{\circ}}\cong\Fun_{\Fin_{\ast}}\pr{\pr{\Sigma\Del^{\op}_{\mathbb{F}}}_{K},\cat O^{\otimes}}.
\]
Note that the projection $\widetilde{\SSeq_{\H}}\pr{\cat O}^{\circ}\to\Del^{\op}$
is a categorical fibration by Remark \ref{rem:flat_cfib}. We write
$\SSeq_{\H}\pr{\cat O}^{\circ}\subset\widetilde{\SSeq_{\H}}\pr{\cat O}^{\circ}$
for the full subcategory spanned by the functors $\{[n]\}\times_{\Del^{\op}}\Sigma\Del^{\op}_{\mathbb{F}}\to\cat O^{\t}$
carrying every morphism to an inert map.
\end{defn}

\begin{rem}
\label{rem:SSeq_H_und}In the situation of Definition \ref{def:SSeq_H_2},
the functor 
\[
\FB\to\{[1]\}\times_{\Del^{\op}}\Sigma\Del^{\op}_{\mathbb{F}},\,S\mapsto\pr{S\to\underline{1}}
\]
induces a categorical equivalence
\[
\SSeq_{\H}\pr{\cat O}\xrightarrow{\simeq}\Fun\pr{\FB,\cat O}.
\]
Indeed, we can identify $\SSeq\pr{\cat O}$ with the full subcategory
of $\Fun_{\Fin_{\ast}}\pr{\{[1]\}\times_{\Del^{\op}}\Sigma\Del^{\op}_{\mathbb{F}},\cat O^{\t}}$
spanned by the functors that are $p$-right Kan extended from $\FB$.
\end{rem}

The next two propositions assert that $\SSeq_{\H}\pr{\cat O}^{\circ}$
is an $\infty$-operad characterized by a certain universal property,
and that the notation $\SSeq_{\H}\pr -^{\circ}$ does not conflict
with Theorem \ref{subsec:Haug}.
\begin{prop}
\label{prop:SSeq_opd}Let $p\from\cat O^{\t}\to\Fin_{\ast}$ be an
$\infty$-operad. The projection $q\from\SSeq_{\H}\pr{\cat O}^{\circ}\to\Del^{\op}$
is a non-symmetric $\infty$-operad. Moreover, if $\alpha\from X\to Y$
is a morphism of $\SSeq_{\H}\pr{\cat O}^{\circ}$ lying over an inert
map $\underline{\alpha}\from[n]\to[m]$, the following conditions
are equivalent:

\begin{enumerate}[label=(\alph*)]

\item The morphism $\alpha$ is $q$-cocartesian.

\item The functor $\pr{\Sigma\Del^{\op}_{\mathbb{F}}}_{\underline{\alpha}}\to\cat O^{\t}$
adjoint to $\alpha$ carries every morphism lying over an inert map
of $\Fin_{\ast}$ to an inert map of $\cat O^{\t}$.

\end{enumerate}
\end{prop}

\begin{prop}
\label{prop:SSeq_univ}Let $\cat O^{\t}$ be an $\infty$-operad,
and let $\cat P^{\t}$ be a non-symmetric $\infty$-operad. The restriction
along the diagonal map $\Del^{\op}\to\Sigma$ induces a categorical
equivalence
\[
\Alg_{\cat P}\pr{\SSeq_{\H}\pr{\cat O}}\xrightarrow{\simeq}\Alg^{\opd}_{\cat P^{\t}\times_{\Del^{\op}}\Del^{\op}_{\mathbb{F}}}\pr{\cat O}.
\]
\end{prop}

\begin{proof}
[Proof of Proposition \ref{prop:SSeq_opd}]It is tempting to use \cite[Theorem B.4.2]{HA},
but condition (5) of loc. cit. is not satisfied. So we must take a
bit of a detour. We will say that a morphism $\alpha$ of $\SSeq\pr{\cat O}^{\circ}$
that lies over an inert map is \textit{special} if it satisfies condition
(b) in the statement. We must prove the following:
\begin{enumerate}
\item For each inert map $\underline{\alpha}:[n]\to[m]$ of $\Del^{\op}$
and object $X\in\SSeq\pr{\cat O}^{\circ}_{[n]}$, there is a special
morphism $\alpha:X\to Y$ lying over $\underline{\alpha}$.
\item Every special morphism of $\SSeq\pr{\cat O}^{\circ}$ is $q$-cocartesian.
\item For each $n\geq0$ and each collection of objects $X_{1},\dots,X_{n}\in\SSeq\pr{\cat O}$,
there is an object $X\in\SSeq\pr{\cat O}$ that admits special maps
$\{X\to X_{i}\}_{1\leq i\leq n}$ lying over the inert maps $\{[n]\to[i-1,i]\cong[1]\}_{1\leq i\leq n}$.
\item Any collection of maps $\{X\to X_{i}\}_{1\leq i\leq n}$ as in (3)
determines a $q$-limit diagram in $\SSeq\pr{\cat O}$.
\end{enumerate}

We start with (1). There is a retraction $\rho\from\pr{\Sigma\Del^{\op}_{\mathbb{F}}}_{\underline{\alpha}}\to\pr{\Sigma\Del^{\op}_{\mathbb{F}}}_{[n]}$
carrying each object $\pr{\iota\from[m]\to[i,j],\,S_{i}\to\cdots\to S_{j}}\in\pr{\Sigma\Del^{\op}_{\mathbb{F}}}_{[m]}$
to $\pr{\iota\underline{\alpha}\from[n]\to[i,j],\,S_{i}\to\cdots\to S_{j}}\in\pr{\Sigma\Del^{\op}_{\mathbb{F}}}_{[n]}$.
Precomposing this retraction to the map $\pr{\Sigma\Del^{\op}_{\mathbb{F}}}_{[n]}\to\cat O^{\t}$
adjoint to $X$, we get the desired special morphism.

Next, we prove (2). Suppose we are given a special morphism $\alpha$
of $\SSeq\pr{\cat O}^{\circ}$ described as in (b). We wish to show
that $\alpha$ is $p$-cocartesian. Using \cite[Proposition B.4.9]{HA}
to the functors $\cat O^{\t}\times_{\Fin_{\ast}}\Sigma\Del^{\op}_{\mathbb{F}}\to\Sigma\Del^{\op}_{\mathbb{F}}\xrightarrow{\ev_{0}}\Del^{\op}$
and the full subcategory $\emptyset\subset\Sigma\Del^{\op}_{\mathbb{F}}$,
we are reduced to showing that the functor
\[
F\from\pr{\Sigma\Del^{\op}_{\mathbb{F}}}_{\alpha}\to\cat O^{\t}
\]
adjoint to $\alpha$ is $p$-left Kan extended from $\pr{\Sigma\Del^{\op}_{\mathbb{F}}}_{0}$.
Since the projection $\pi\from\pr{\Sigma\Del^{\op}_{\mathbb{F}}}_{\alpha}\to[1]$
is a cartesian fibration, this is equivalent to the condition that
$F$ carries $\pi$-cartesian morphisms to inert maps. But $\pi$-cartesian
morphisms lie over isomorphisms of $\Fin_{\ast}$, so the claim follows
from the specialty of $F$.

We next turn to (3). Consider the discrete set $\underline{n}=\{1,\dots,n\}$
and the poset $\underline{n}^{\lcone}$ obtained from $\underline{n}$
by adjoining a minimal element $-\infty$. We let $\underline{n}^{\lcone}\to\Del^{\op}$
denote the functor carrying the map $-\infty\to i$ to the inert map
$[n]\to[i-1,i]\cong[1]$. We can identify the objects $X_{1},\dots,X_{n}$
with a functor $F\from\pr{\Sigma\Del^{\op}_{\mathbb{F}}}_{\underline{n}}\to\cat O^{\t}$,
and we must find a filler as indicated below
\[\begin{tikzcd}
	{(\Sigma \mathbf{\Delta}^{\mathrm{op}}_{\mathbb{F}})_{\underline{n}}} & {\mathcal{O}^\otimes } \\
	{(\Sigma \mathbf{\Delta}^{\mathrm{op}}_{\mathbb{F}})_{\underline{n}^\triangleleft}} & {\mathsf{Fin}_\ast,}
	\arrow["F", from=1-1, to=1-2]
	\arrow[hook, from=1-1, to=2-1]
	\arrow["p", from=1-2, to=2-2]
	\arrow[dashed, from=2-1, to=1-2]
	\arrow[from=2-1, to=2-2]
\end{tikzcd}\]which determines special maps of $\SSeq\pr{\cat O}^{\circ}$. We will
construct the filler as a $p$-right Kan extension of $F$, and then
show that this corresponds to $n$ special maps of $\SSeq\pr{\cat O}^{\circ}$.

Consider an object $\xi=\pr{[n]\to[i,j],S_{i}\to\cdots\to S_{j}}\in\pr{\Sigma\Del^{\op}_{\mathbb{F}}}_{[n]}$.
We will assume that $i<j$, because the argument for the case where
$i=j$ is similar. The inert maps $\{[i,j]\to[k-1,k]\cap[i,j]\}_{i\leq k\leq j+1}$
and the identity maps of the $S_{k}$'s determine an initial map
\[
\{i,\dots,j+1\}\hookrightarrow\pr{\pr{\Sigma\Del^{\op}_{\mathbb{F}}}_{\underline{n}}}_{\xi/}.
\]
Therefore, to prove the existence of a $p$-right Kan extension, it
will suffice to show that there is a $p$-limit diagram $\{i,\dots,j+1\}^{\lcone}\to\cat O^{\t}$
lying over the composite $\{i,\dots,j+1\}^{\lcone}\to\pr{\Sigma\Del^{\op}_{\mathbb{F}}}_{\underline{n}^{\lcone}}\to\Fin_{\ast}$.
This follows from the definition of $\infty$-operads. Moreover, this
argument shows that the $p$-right Kan extension $\overline{F}:\u n^{\lcone}\times_{\Del^{\op}}\Sigma\Del^{\op}_{\mathbb{F}}\to\cat O^{\t}$
determines special maps of $\SSeq\pr{\cat O}^{\circ}$. 

The proof of part (4) is similar to that of part (3), using \cite[Proposition B.4.9]{HA}
again. The proof is now complete.
\end{proof}

To facilitate the proof of Proposition \ref{prop:SSeq_univ}, we will
use \textit{marked simplicial sets}, which is a useful gadget to keep
track of a designated class of morphisms (such as inert maps). 
\begin{recollection}
A \textbf{marked simplicial set} is a pair $\pr{S,M}$, where $S$
is a simplicial set and $M$ is a set of edges of $S$ containing
all degenerate edges. If $\pr{X,E}$ and $\pr{X',E'}$ are marked
simplicial sets equipped with maps $p\from X\to S$ and $p'\from X'\to S$
of simplicial sets, we write $\Fun_{S}\pr{\pr{X,E},\pr{X',E'}}$ for
the full simplicial subset of $\Fun_{S}\pr{X,X'}=\Fun\pr{X,X'}\times_{\Fun\pr{X',S}}\{p'\}$
spanned by the maps $X\to X'$ carrying $E$ into $E'$. 
\end{recollection}

If $\cat O^{\t}$ is a symmetric or non-symmetric $\infty$-operad,
we write $\cat O^{\t,\natural}$ for the marked simplicial set obtained
from $\cat O^{\t}$ by marking the inert edges. If $X$ is a simplicial
set, we write $X^{\sharp}$ for the marked simplicial set obtained
from $X$ by marking \textit{all} edges.
\begin{proof}
[Proof of Proposition \ref{prop:SSeq_univ}]Write $\Del^{\op,\natural}_{\mathbb{F}}$
for the marked simplicial set obtained from $\Del^{\op}$ by marking
the edges lying over inert maps of $\Del^{\op}$. By Proposition \ref{prop:SSeq_opd},
the marked simplicial set $\SSeq\pr{\cat O}^{\circ,\natural}\in\SS^{+}_{/\Del^{\op,\natural}}$
is characterized by the isomorphism of simplicial sets 
\[
\Fun_{\Del^{\op}}\pr{\overline{K},\SSeq_{\H}\pr{\cat O}^{\natural}}\cong\Fun_{\Fin_{\ast}}\pr{\overline{K}\times_{\pr{\Del^{\op,\natural}}^{\{0\}^{\sharp}}}\pr{\Del^{\op,\natural}}^{[1]^{\sharp}}\times_{\pr{\Del^{\op,\natural}}^{\{1\}^{\sharp}}}\Del^{\op,\natural}_{\mathbb{F}},\mathcal{O}^{\t,\natural}}
\]
natural in $\overline{K}\in\SS^{+}_{/\Del^{\op,\natural}}$. In light
of this, it suffices to show that the map
\[
\theta\from\cat P^{\t,\natural}\times_{\Del^{\op,\natural}}\Del^{\op,\natural}_{\mathbb{F}}\to\cat P^{\t,\natural}\times_{\pr{\Del^{\op,\natural}}^{\{0\}^{\sharp}}}\pr{\Del^{\op,\natural}}^{[1]^{\sharp}}\times_{\pr{\Del^{\op,\natural}}^{\{1\}^{\sharp}}}\Del^{\op,\natural}_{\mathbb{F}}
\]
is a weak equivalence in the model category of $\infty$-preoperads
\cite[Proposition 2.1.4.6]{HA}. 

We can factor $\theta$ as
\begin{align*}
\cat P^{\t,\natural}\times_{\Del^{\op,\natural}}\Del^{\op,\natural}_{\mathbb{F}} & \xrightarrow{\theta'}\pr{\cat P^{\t,\natural}}^{[1]^{\sharp}}\times_{\pr{\Del^{\op,\natural}}^{\{1\}^{\sharp}}}\Del^{\op,\natural}_{\mathbb{F}}\\
 & \xrightarrow{\theta''}\cat P^{\t,\natural}\times_{\pr{\Del^{\op,\natural}}^{\{0\}^{\sharp}}}\pr{\Del^{\op,\natural}}^{[1]^{\sharp}}\times_{\pr{\Del^{\op,\natural}}^{\{1\}^{\sharp}}}\Del^{\op,\natural}_{\mathbb{F}}.
\end{align*}
The map $\theta'$ is a weak equivalence because it has a homotopy
inverse, given by the evaluation at $1\in[1]$. The second map is
a pullback of the map $\pr{\cat P^{\t,\natural}}^{[1]^{\sharp}}\to\cat P^{\t,\natural}\times_{\pr{\Del^{\op,\natural}}^{\{0\}^{\sharp}}}\pr{\Del^{\op,\natural}}^{[1]^{\sharp}}$,
which is a trivial fibration by \cite[Proposition B.1.9]{HA}. Hence
$\theta$ is a weak equivalence, as claimed.
\end{proof}

We now focus on the $\infty$-operad $\SSeq_{\H}\pr{\cat C}^{\circ}$
in the case where $\cat C^{\t}$ is a symmetric monoidal $\infty$-category.
We will show that it is a monoidal $\infty$-category when $\cat C^{\t}$
is compatible with colimits indexed by countable groupoids (Corollary
\ref{cor:SSeq_Haug}), and that it is compatible with the classical
composition product (Proposition \ref{prop:Haug_vs_classical}).
\begin{construction}
\label{const:env}Every $\infty$-operad $\cat O^{\t}$ can be freely
made into a symmetric monoidal $\infty$-category $\Env\pr{\cat O}^{\t}$,
called the \textbf{symmetric monoidal envelop} \cite[Construction 2.2.4.1, Proposition 2.2.4.9]{HA}.
Concretely, $\Env\pr{\cat O}^{\t}$ is given by the fiber product
\[\begin{tikzcd}
	{\operatorname{Env}(\mathcal{O})^\otimes } & {\mathcal{O}^\otimes} \\
	{\operatorname{Fun}^{\mathrm{act}}([1],\mathsf{Fin}_\ast)} & {\mathsf{Fin}_\ast}
	\arrow["{p_{\mathcal{O}}}", from=1-1, to=1-2]
	\arrow[from=1-1, to=2-1]
	\arrow["\lrcorner"{description, pos=0}, draw=none, from=1-1, to=2-2]
	\arrow[from=1-2, to=2-2]
	\arrow["{\mathrm{ev}_0}"', from=2-1, to=2-2]
\end{tikzcd}\]with structure map $\Env\pr{\cat O}^{\t}\to\Fin_{\ast}$ given by
the evaluation at $1\in[1]$, where $\Fun^{\act}\pr{[1],\Fin_{\ast}}\subset\Fun\pr{[1],\Fin_{\ast}}$
denotes the full subcategory spanned by the active maps. The diagonal
$\Fin_{\ast}\to\Fun^{\act}\pr{[1],\Fin_{\ast}}$ determines a map
of $\infty$-operads $\eta\from\cat O^{\t}\to\Env\pr{\cat O}^{\t}$,
and for every symmetric monoidal $\infty$-category $\cat C^{\t}$,
pulling back along $\eta$ induces a categorical equivalence
\[
\Fun^{\t}\pr{\Env\pr{\cat O},\cat C}\xrightarrow{\simeq}\Alg_{\cat O}\pr{\cat C}.
\]

Note that the underlying $\infty$-category $\Env\pr{\cat O}$ can
be identified with the subcategory $\cat O^{\t}_{\act}\subset\cat O^{\t}$
of active maps. We write $\oplus\from\cat O^{\t}_{\act}\times\cat O^{\t}_{\act}\to\cat O^{\t}_{\act}$
for the tensor bifunctor of $\Env\pr{\cat O}$.

Applying the above construction to $\cat O^{\t}=\cat C^{\t}$, we
obtain an essentially unique symmetric monoidal functor $\Env\pr{\cat C}^{\t}\to\cat C^{\t}$
that extends the identity of $\cat C^{\t}$ up to equivalence. (One
way to construct it is to choose a cocartesian natural transformation
rendering the diagram 
\[\begin{tikzcd}
	{\operatorname{Env}(\mathcal{C})^\otimes \times \{0\}} && {\mathcal{C}^\otimes} \\
	{\operatorname{Env}(\mathcal{C})^\otimes \times [1]} & {\operatorname{Fun}^{\mathrm{act}}([1],\mathsf{Fin}_\ast)\times [1]} & {\mathsf{Fin}_\ast}
	\arrow["{p_{\mathcal{C}}}", from=1-1, to=1-3]
	\arrow[from=1-1, to=2-1]
	\arrow[from=1-3, to=2-3]
	\arrow[dashed, from=2-1, to=1-3]
	\arrow[from=2-1, to=2-2]
	\arrow["{\mathrm{ev}}"', from=2-2, to=2-3]
\end{tikzcd}\]commutative, and then restricting it along the inclusion $\{1\}\subset[1]$.)
We denote its underlying functor by $\bigotimes\from\cat C^{\t}_{\act}\to\cat C$. 
\end{construction}

\begin{notation}
We let $\Del^{\op}_{\act}\subset\Del^{\op}$ denote the subcategory
of active morphisms. For each $n\ge1$, we will denote the unique
active map $[n]\to[1]$ by $\mu_{n}$.
\end{notation}

\begin{notation}
Let $\cat F\subset\Del^{\op}_{\mathbb{F}}$ denote the subcategory
spanned by the morphisms $\pr{\phi,\{u_{i}\}_{i}}\from\pr{[n],S_{0}\to\cdots\to S_{n}}\to\pr{[m],T_{0}\to\cdots\to T_{m}}$
such that the maps $u_{i}$ are all \textit{bijective}. In other words,
the map $\cat F\to\Del^{\op}$ is the cocartesian fibration corresponding
to the functor 
\[
\Del^{\op}\to\mathsf{Cat},\,[n]\mapsto\pr{\Fun\pr{[n],\Fin}^{\op}}^{\simeq}.
\]
We write $\cat G\subset\cat F$ for the full subcategory spanned by
the objects $\pr{[n],S_{0}\to\cdots\to S_{n}}$ such that $S_{n}=\underline{1}$
is a singleton. 
\end{notation}

\begin{construction}
\label{const:associated_nat}Let $\cat C^{\t}$ be a symmetric monoidal
$\infty$-category. Suppose we are given a morphism $f$ of $\SSeq_{\H}\pr{\cat C}^{\circ}$
lying over the active map $\mu_{n}\from[n]\to[1]$, which we can identify
with a functor $\pr{\Sigma\Del^{\op}_{\mathbb{F}}}_{\mu_{n}}\to\cat C^{\t}$.
We define a natural transformation $\alpha_{f}$ of functors $\cat F_{[n]}\to\cat C$
as follows: Consider the composite
\[
[1]\times\cat F_{[n]}\xrightarrow{h}\cat F_{\mu_{n}}\xrightarrow{g}\pr{\Sigma\Del^{\op}_{\mathbb{F}}}_{\mu_{n}}\xrightarrow{f}\cat C^{\t},
\]
where $h$ is a cocartesian natural transformation, and where $g$
is induced by the diagonal map $\Del^{\op}\to\Fun\pr{[1],\Del^{\op}}$.
The above composite takes values in $\cat C^{\t}_{\act}$, so it can
further be composed with the functor $\bigotimes\from\cat C^{\t}_{\act}\to\cat C$
of Construction \ref{const:env}. We define $\alpha_{f}=\bigotimes\circ f\circ g\circ h$.
\end{construction}

\begin{rem}
\label{rem:informal}In the situation of Construction \ref{const:associated_nat},
write $f\from X\to Y$ and $X=X_{1}\oplus\cdots\oplus X_{n}$, where
$X_{i}\in\SSeq_{\H}\pr{\cat C}$. By Remark \ref{rem:SSeq_H_und},
we can identify $X_{i}$ and $Y$ with symmetric sequences $\overline{X}_{i}$
and $\overline{Y}$ in $\cat C$. The component of the natural transformation
$\alpha_{f}$ at an object $\pr{S_{0}\to\cdots\to S_{n}}\in\cat F_{[n]}$
is given by
\[
\overline{X}_{1}\pr{S_{0}\to S_{1}}\otimes\cdots\otimes\overline{X}_{n}\pr{S_{n-1}\to S_{n}}\to\overline{Y}\pr{S_{0}\to S_{n}},
\]
where $\overline{Y}\pr{A\to B}$ is a shorthand for $\bigotimes_{b\in B}\overline{Y}\pr{A_{b}}$
and we used similar abbreviations for $\overline{X}_{i}\pr{S_{i-1}\to S_{i}}$.
\end{rem}

\begin{prop}
\label{prop:SSeq_H_loc}Let $p\from\cat C^{\t}\to\Fin_{\ast}$ be
a symmetric monoidal $\infty$-category, and let $q\from\SSeq\pr{\cat C}^{\circ}\to\Del^{\op}$
be the associated non-symmetric $\infty$-operad. 
\begin{enumerate}
\item Let $f$ be an active map of $\SSeq\pr{\cat C}^{\circ}$ lying over
$\mu_{n}\from[n]\to[1]$, which we identify with a functor $\pr{\Sigma\Del^{\op}_{\mathbb{F}}}_{\mu_{n}}\to\cat C$.
Suppose that the following condition is satisfied:
\begin{itemize}
\item [($\ast_{\mathcal{G}}$)]The natural transformation $\alpha_{f}\from\cat G_{[n]}\times[1]\to\cat C$
exhibits $f\vert\cat G_{[1]}$ as a left Kan extension of $\alpha_{f}\vert\cat G_{[n]}\times\{0\}$
along the functor $\pr{\mu_{n}}_{!}\from\cat G_{[n]}\to\cat G_{[1]}$.
\end{itemize}
Then $f$ is locally $q$-cocartesian. 
\item Let $X\in\SSeq\pr{\cat C}^{\circ}_{[n]}$. If the composite
\[
\theta_{X}\from\cat G_{[n]}\to\pr{\Sigma\Delta^{\op}_{\mathbb{F}}}_{[n]}\to\cat C^{\t}_{\act}\xrightarrow{\bigotimes}\cat C
\]
admits a left Kan extension along $\pr{\mu_{n}}_{!}\theta_{X}\from\cat G_{[n]}\to\cat G_{[1]}$,
then there is a locally $q$-cocartesian morphism $X\to\pr{\mu_{n}}_{!}X$
lying over $\mu_{n}$.
\end{enumerate}
\end{prop}

The proof of Proposition \ref{prop:SSeq_H_loc} relies on the following
lemma.
\begin{lem}
\label{lem:SSeq_Hom}Let $p\from\cat O^{\t}\to\Fin_{\ast}$ be an
$\infty$-operad. For each $n\geq1$, the active map $[n]\to[1]$
determines a homotopy cartesian square of $\infty$-categories 
\[\begin{tikzcd}
	{\operatorname{Fun}_{\mathbf{\Delta}^{\mathrm{op}}}([1],\Sigma\mathrm{Seq}(\mathcal{O})^\circ)} & {\Sigma\mathrm{Seq}(\mathcal{O})^\circ_{[n]}} \\
	{\operatorname{Fun}'_{\mathsf{Fin}_\ast}((\mathbf{\Delta}^{\mathrm{op}}_{\mathbb{F}})_{\mu_n},\mathcal{O}^\otimes )} & {\operatorname{Fun}'_{\mathsf{Fin}_\ast}((\mathbf{\Delta}^{\mathrm{op}}_{\mathbb{F}})_{[n]},\mathcal{O}^\otimes ).}
	\arrow[from=1-1, to=1-2]
	\arrow[from=1-1, to=2-1]
	\arrow[from=1-2, to=2-2]
	\arrow[from=2-1, to=2-2]
\end{tikzcd}\]Here $\Fun'_{\Fin_{\ast}}\pr{-,-}$ denotes the full subcategory of
$\Fun_{\Fin_{\ast}}\pr{-,-}$ spanned by the functors carrying each
morphism in $\pr{\Del^{\op}_{\mathbb{F}}}_{[n]}$ and $\pr{\Del^{\op}_{\mathbb{F}}}_{[1]}$
to inert maps of $\cat O^{\t}$.
\end{lem}

\begin{proof}
Set $\cat X=\pr{\Sigma\Del^{\op}_{\mathbb{F}}}_{\mu_{n}}$ and $\cat X'=\pr{\Del^{\op}_{\mathbb{F}}}_{\mu_{n}}$.
We will identify $\cat X'$ with the full subcategory of $\cat X$
via the diagonal embedding $\Del^{\op}\to\Sigma$. For each $i\in[1]$,
let $M_{i}$ and $M'_{i}$ denote the set of morphisms of $\cat X$
and $\cat X'$ lying over $i\in[1]$. Our goal is to show that the
square 
\[\begin{tikzcd}
	{\operatorname{Fun}_{\mathsf{Fin}_\ast}((\mathcal{X},M_0\cup M_1),\mathcal{O}^{\otimes ,\natural} )} & {\operatorname{Fun}_{\mathsf{Fin}_\ast}((\mathcal{X}_0,M_0),\mathcal{O}^{\otimes ,\natural} )} \\
	{\operatorname{Fun}_{\mathsf{Fin}_\ast}((\mathcal{X}',M'_0\cup M'_1),\mathcal{O}^{\otimes ,\natural} )} & {\operatorname{Fun}_{\mathsf{Fin}_\ast}((\mathcal{X}'_0,M'_0),\mathcal{O}^{\otimes ,\natural} )}
	\arrow[from=1-1, to=1-2]
	\arrow[from=1-1, to=2-1]
	\arrow[from=1-2, to=2-2]
	\arrow[from=2-1, to=2-2]
\end{tikzcd}\]is homotopy cartesian.

Notice that if $X\to Y$ is a map of $\cat X$ such that $X\in\cat X_{0}$
and $Y\in\cat X'_{1}$, then $X$ necessarily lies in $\cat X'_{0}$.
It follows that $\cat X_{0}\cup\cat X'$ is a full subcategory of
$\cat X$. Moreover, for each $X\in\cat X$ lying outside of $\cat X_{0}\cup\cat X'$,
the category $\pr{\cat X_{0}\cup\cat X'}\times_{\cat X}\cat X_{X/}$
is empty. Since every functor $F\from\cat X\to\cat O^{\t}$ carries
objects outside of $\cat X_{0}\cup\cat X'$ to a $p$-terminal object
(as they lies over $\inp 0\in\Fin_{\ast}$), this means that the functor
\[
\Fun_{\Fin_{\ast}}\pr{\cat X,\cat O^{\t}}\to\Fun_{\Fin_{\ast}}\pr{\cat X_{0}\cup\cat X',\cat O^{\t}}
\]
is an equivalence. In particular, the square 
\[\begin{tikzcd}
	{\operatorname{Fun}_{\mathsf{Fin}_\ast}((\mathcal{X},M_0\cup M'_1),\mathcal{O}^{\otimes ,\natural} )} & {\operatorname{Fun}_{\mathsf{Fin}_\ast}((\mathcal{X}_0,M_0),\mathcal{O}^{\otimes ,\natural} )} \\
	{\operatorname{Fun}_{\mathsf{Fin}_\ast}((\mathcal{X}',M'_0\cup M'_1),\mathcal{O}^{\otimes ,\natural} )} & {\operatorname{Fun}_{\mathsf{Fin}_\ast}((\mathcal{X}'_0,M'_0),\mathcal{O}^{\otimes ,\natural} )}
	\arrow[from=1-1, to=1-2]
	\arrow[from=1-1, to=2-1]
	\arrow[from=1-2, to=2-2]
	\arrow[from=2-1, to=2-2]
\end{tikzcd}\]is homotopy cartesian. The claim now follows from the observation
that every functor $\cat X\to\cat O^{\t,\natural}$ carries every
morphism in $M_{1}\setminus M'_{1}$ to a $p$-cocartesian morphism
(because its codomain lies over $\inp 0\in\Fin_{\ast}$).
\end{proof}

\begin{proof}
[Proof of Proposition \ref{prop:SSeq_H_loc}]We start with (1). Consider
the commutative diagram
\[\begin{tikzcd}
	{\operatorname{Fun}_{\mathbf{\Delta}^{\mathrm{op}}}([1],\Sigma \mathrm{Seq}(\mathcal{C})^\circ)} & {\Sigma \mathrm{Seq}(\mathcal{C})^\circ_{[n]}} \\
	{\operatorname{Fun}'_{\mathsf{Fin}_\ast}((\mathbf{\Delta}^{\mathrm{op}}_{\mathbb{F}})_{\mu_n},\mathcal{C}^\otimes )} & {\operatorname{Fun}'_{\mathsf{Fin}_\ast}((\mathbf{\Delta}^{\mathrm{op}}_{\mathbb{F}})_{[n]},\mathcal{C}^\otimes )} \\
	{\operatorname{Fun}_{\mathsf{Fin}_\ast}(\mathcal{G}_{\mu_n},\mathcal{C}^\otimes )} & {\operatorname{Fun}_{\mathsf{Fin}_\ast}( \mathcal{G}_{[n]},\mathcal{C}^\otimes ).}
	\arrow[from=1-1, to=1-2]
	\arrow[from=1-1, to=2-1]
	\arrow[from=1-2, to=2-2]
	\arrow[from=2-1, to=2-2]
	\arrow["\simeq"', from=2-1, to=3-1]
	\arrow["\simeq", from=2-2, to=3-2]
	\arrow[from=3-1, to=3-2]
\end{tikzcd}\]The top square is homotopy cartesian by Lemma \ref{lem:SSeq_Hom}.
The right bottom vertical arrows is an equivalence because functors
in $\Fun'_{\Fin_{\ast}}\pr{\pr{\Del^{\op}_{\mathbb{F}}}_{[n]},\cat O^{\t}}$
are exactly those that are $p$-right Kan extended from $\cat G_{[n]}$.
Similarly, the left bottom vertical arrow is an equivalence. It follows
that the outer square is homotopy cartesian. In particular, the map
\begin{align*}
 & \Fun_{\Del^{\op}}\pr{[1],\SSeq\pr{\cat C}^{\circ}}\times_{\SSeq\pr{\cat C}^{\circ}_{[n]}}\{f\vert\SSeq\pr{\cat C}^{\circ}_{[n]}\}\\
\to & \Fun_{\Fin_{\ast}}\pr{\cat G_{\mu_{n}},\cat C^{\t}}\times_{\Fun_{\Fin_{\ast}}\pr{\cat G_{[n]},\cat C^{\t}}}\{f\vert\cat G_{[n]}\}
\end{align*}
is a categorical equivalence. Therefore, it suffices to show that
$f\vert\cat G_{\mu_{n}}$ is a $p$-left Kan extension of $f\vert\cat G_{[n]}$.
According to \cite[Propositions 4.3.1.9, 4.3.1.10, and 4.3.1.15]{HTT},
this is equivalent to the condition that the composite
\[
\cat G_{\mu_{n}}\xrightarrow{f\vert\cat G_{\mu_{n}}}\cat C^{\t}_{\act}\xrightarrow{\bigotimes}\cat C
\]
be left Kan extended from $\cat G_{[n]}$. But this is true by ($\ast_{\cat G}$),
and we are done.

For (2), note that the above argument shows that the image of $X$
in $\Fun_{\Fin_{\ast}}\pr{\cat G_{[n]},\cat C^{\t}}$ admits a $p$-left
Kan extension. Since the outer rectangle of the diagram is homotopy
cartesian, the $p$-left Kan extension gives rise to a morphism $f\from X\to Y$
in $\SSeq\pr{\cat C}^{\circ}$ satisfying condition ($\ast_{\cat G}$).
The map $f$ is locally $q$-cocartesian by (1), and this proves (2).
\end{proof}

\begin{cor}
\label{cor:SSeq_Haug}Let $p\from\cat C^{\t}\to\Fin_{\ast}$ be a
symmetric monoidal $\infty$-category compatible with colimits indexed
by countable groupoids. Then $q\from\SSeq\pr{\cat C}^{\circ}\to\Del^{\op}$
is a monoidal $\infty$-category. Moreover, an active map $f$ of
$\SSeq\pr{\cat C}^{\circ}$ lying over $\mu_{n}\from[n]\to[1]$ is
$q$-cocartesian if and only if it satisfies the following condition:
\begin{itemize}
\item [($\ast_{\mathcal{F}}$)]The natural transformation $\alpha_{f}\from\cat F_{[n]}\times[1]\to\cat C$
exhibits $f\vert\cat F_{[1]}$ as a left Kan extension of $\alpha_{f}\vert\cat F_{[n]}\times\{0\}$
along the functor $\pr{\mu_{n}}_{!}\from\cat F_{[n]}\to\cat F_{[1]}$.
\end{itemize}
\end{cor}

\begin{proof}
We first show that an active map $f\from X\to Y$ of $\SSeq\pr{\cat C}^{\circ}$
lying over $\mu_{n}\from[n]\to[1]$ is locally $q$-cocartesian if
and only if it satisfies condition ($\ast_{\cat F}$). By Proposition
\ref{prop:SSeq_H_loc}, it suffices to show that condition ($\ast_{\cat F}$)
is equivalent to condition ($\ast_{\cat G}$) there. Clearly ($\ast_{\cat F}$)
implies ($\ast_{\cat G}$), so we only have to prove the converse. 

Suppose that condition ($\ast_{\cat G}$) is satisfied. We must show
that, for each $S\to T\in\cat F_{[1]}$, the map 
\begin{equation}
\colim_{S_{0}\to\cdots\to S_{n}\in\cat F_{[n]/S\to T}}\bigotimes\circ X\pr{\id_{[n]},S_{0}\to\cdots\to S_{n}}\to\bigotimes\circ Y\pr{\id_{[1]},S\to T}\label{eq:informal}
\end{equation}
is an equivalence. More formally, we must show that the top horizontal
composite in the diagram below is a colimit diagram:
\[\begin{tikzcd}[scale cd = .7]
	{(\mathcal{F}_{[n]/S\to T})^{\triangleright}} & {\mathcal{F}_{\mu_n}} & {(\mathbf{\Delta}_{\mathbb{F}}^{\mathrm{op}})_{\mu_n}} & {\mathcal{C}^{\otimes}_{\mathrm{act}}} & {\mathcal{C}.} \\
	{(\prod_{t\in T}\mathcal{F}_{[n]/S\to \{t\}})^{\triangleright}} & {\prod_{t\in T}((\mathcal{F}_{[n]/S\to \{t\}})^{\triangleright})} & {\prod_{t\in T}(\mathbf{\Delta}_{\mathbb{F}}^{\mathrm{op}})_{\mu_n}} & {\prod_{t\in T}\mathcal{C}^{\otimes}_{\mathrm{act}}} & {\prod_{t\in T}\mathcal{C}}
	\arrow[from=1-1, to=1-2]
	\arrow["\simeq"', from=1-1, to=2-1]
	\arrow[from=1-2, to=1-3]
	\arrow[from=1-3, to=1-4]
	\arrow["\bigotimes", from=1-4, to=1-5]
	\arrow[from=2-1, to=2-2]
	\arrow[from=2-2, to=2-3]
	\arrow[from=2-3, to=2-4]
	\arrow["{\bigoplus_{t\in T}}", from=2-4, to=1-4]
	\arrow["{\prod_{t\in T}\bigotimes}"', from=2-4, to=2-5]
	\arrow["{\bigotimes_{t\in T}}"', from=2-5, to=1-5]
\end{tikzcd}\]The left hand rectangle commutes by \cite[Corollary 4.5]{MEFO}, and
the right hand rectangle commutes because $\bigotimes\from\cat C^{\t}_{\act}\to\cat C$
is symmetric monoidal. In other words, the map (\ref{eq:informal})
can be identified with 
\[
\colim_{S_{0}\to\cdots\to S_{n}\in\cat F_{[n]/S\to T}}\bigotimes_{t\in T}\pr{\bigotimes\circ X\pr{\id_{[n]},S_{0,t}\to\cdots\to S_{n,t}}\xrightarrow{\simeq}\bigotimes\circ Y\pr{\id_{[1]},S\to\{t\}}}.
\]
Since $\cat C^{\t}$ is compatible with colimits indexed by countable
groupoids, we are therefore reduced to showing to that
each of the maps $\{\pr{\cat F_{[n]/S\to\{t\}}}^{\rcone}\to\cat C\}_{t\in T}$
is a colimit diagram. This follows from ($\ast_{\cat G}$). 

Next, we show that $\SSeq\pr{\cat C}^{\circ}$ is a monoidal $\infty$-category.
According to \cite[Lemma 2.1.25]{Haug22}, the map $q$ is a locally
cocartesian fibration, and we only have to show that the composite
of two locally $q$-cocartesian maps lying over active maps $\phi\from[n]\to[2]$
and $\mu_{2}\from[2]\to[1]$ is again locally cocartesian. So take
such maps $u\from X\to Y$ and $v\from Y\to Z$, and let $w\from X\to Z$
denote a composite of $v$ and $u$: 
\[\begin{tikzcd}
	& Y &&& {[2]} \\
	X && Z & {[n]} && {[1]}
	\arrow["v", from=1-2, to=2-3]
	\arrow["{\mu_2}", from=1-5, to=2-6]
	\arrow["u", from=2-1, to=1-2]
	\arrow["w"', from=2-1, to=2-3]
	\arrow["\phi", from=2-4, to=1-5]
	\arrow["{\mu_n}"', from=2-4, to=2-6]
\end{tikzcd}\]We must show that $w$ is also locally $q$-cocartesian. Using ($\ast_{\cat F}$)
and the transitivity of Kan extensions, we are reduced to showing
that the composite
\[
\cat F_{\phi}\to\cat C^{\t}_{\act}\xrightarrow{\bigotimes}\cat C
\]
is left Kan extended from $\cat F_{[n]}$. So take an arbitrary object
$\pr{T_{0}\to T_{1}\to T_{2}}\in\cat F_{[2]}$. We must show that
the composite
\[
\pr{\cat F_{[n]/T_{0}\to T_{1}\to T_{2}}}^{\rcone}\to\cat F_{\phi}\to\cat C^{\t}_{\act}\xrightarrow{\bigotimes}\cat C
\]
is a colimit diagram. For this, find a diagram in $\SSeq\pr{\cat C}^{\circ}$
lying over the diagram on the right:
\[\begin{tikzcd}
	{X_{1}} & X & {X_2} & {[0,\phi(1)]} & {[n]} & {[\phi(1),n]} \\
	{Y_1} & Y & {Y_2} & {[0,1]} & {[2]} & {[1,2]}
	\arrow["{u_1}"', from=1-1, to=2-1]
	\arrow[tail, from=1-2, to=1-1]
	\arrow[tail, from=1-2, to=1-3]
	\arrow["u"', from=1-2, to=2-2]
	\arrow["{u_2}", from=1-3, to=2-3]
	\arrow["{\phi_1}"', from=1-4, to=2-4]
	\arrow[tail, from=1-5, to=1-4]
	\arrow[tail, from=1-5, to=1-6]
	\arrow["\phi", from=1-5, to=2-5]
	\arrow["{\phi_2}", from=1-6, to=2-6]
	\arrow[tail, from=2-2, to=2-1]
	\arrow[tail, from=2-2, to=2-3]
	\arrow[tail, from=2-5, to=2-4]
	\arrow[tail, from=2-5, to=2-6]
\end{tikzcd}\]Here the arrows with tails ``$\rightarrowtail$'' are inert, and
the remaining arrows are active. The maps $u_{1}$ and $u_{2}$ gives
rise to maps $\cat F_{\mu_{\phi\pr 1}}\to\cat C^{\t}_{\act}$ and
$\cat F_{\mu_{n-\phi\pr 1}}\to\cat C^{\t}_{\act}$, which fits into
the following diagram:
\[\begin{tikzcd}[column sep = small, scale cd = .6]
	{(\mathcal{F}_{[n]/T_0\to T_1 \to T_2})^\triangleright} & {\mathcal{F}_{\phi}} && {\mathcal{C}^{\otimes}_{\mathrm{act}}} & {\mathcal{C}} \\
	{(\mathcal{F}_{[0,\phi(1)]/T_0\to T_1}\times \mathcal{F}_{[\phi(1),n]/T_1\to T_2})^\triangleright} & {(\mathcal{F}_{[0,\phi(1)]/T_0\to T_1})^\triangleright\times (\mathcal{F}_{[\phi(1),n]/T_1\to T_2})^\triangleright} & {\mathcal{F}_{\mu_{\phi(1)}}\times \mathcal{F}_{n-\phi(1)}} & {\mathcal{C}^\otimes _{\mathrm{act}}\times \mathcal{C}^\otimes _{\mathrm{act}}} & {\mathcal{C}\times \mathcal{C}}
	\arrow[from=1-1, to=1-2]
	\arrow["\simeq"', from=1-1, to=2-1]
	\arrow[from=1-2, to=1-4]
	\arrow["\bigotimes", from=1-4, to=1-5]
	\arrow[from=2-1, to=2-2]
	\arrow[from=2-2, to=2-3]
	\arrow[from=2-3, to=2-4]
	\arrow["\bigoplus", from=2-4, to=1-4]
	\arrow["{\bigotimes\times \bigotimes}"', from=2-4, to=2-5]
	\arrow["\otimes"', from=2-5, to=1-5]
\end{tikzcd}\]As before, the diagram commutes up to equivalence. Now $u_{1}$ and
$u_{2}$ are locally $q$-cocartesian (see the argument of \cite[Lemma 2.1.25]{Haug22}),
so by ($\ast_{\cat F}$), the maps 
\[
\pr{\cat F_{[0,\phi\pr 1]/T_{0}\to T_{1}}}^{\rcone}\to\cat C,\,\pr{\cat F_{[\phi\pr 1,n]/T_{1}\to T_{2}}}^{\rcone}\to\cat C
\]
are colimit diagrams. Therefore, the claim follows from the compatibility
of $\cat C^{\t}$ with colimits indexed by countable groupoids. The
proof is now complete.
\end{proof}

\begin{prop}
\label{prop:Haug_vs_classical}Let $\cat C$ be a symmetric monoidal
category compatible with countable groupoids. There is an equivalence
of monoidal $\infty$-categories
\[
\SSeq\pr{\cat C}^{\circ}\xrightarrow{\simeq}\SSeq_{\H}\pr{\cat C}^{\circ}
\]
which is natural with respect to monoidal functors preserving colimits
indexed by countable groupoids.
\end{prop}

\begin{proof}
We will construct a map $\Phi\from\SSeq\pr{\cat C}^{\circ}\to\SSeq_{\H}\pr{\cat C}^{\circ}$
as follows: By construction, $\SSeq_{\H}\pr{\cat C}^{\circ}$ is isomorphic
to the nerve of a category, so it suffices to specify $\Phi$ on objects
and morphisms. An object of $\SSeq\pr{\cat C}^{\circ}$ is a finite
sequence $\pr{X_{1},\dots,X_{n}}$ of symmetric sequences in $\cat C$.
Its image $\Phi\pr{X_{1},\dots,X_{n}}\in\SSeq_{\H}\pr{\cat C}^{\circ}$
is given by the functor
\begin{align*}
\pr{\Sigma\Del^{\op}_{\mathbb{F}}}_{[n]} & \to\cat C^{\t}\\
\pr{[n]\to[i,j],S_{i}\to\cdots\to S_{j}} & \mapsto\pr{X_{p}\pr{S_{p-1,s}}}_{\pr{p,s}\in\coprod_{i<p\leq j}S_{p}}.
\end{align*}
Next, given a morphism $f\from\pr{X_{1},\dots,X_{n}}\to\pr{Y_{1},\dots,Y_{m}}$
in $\SSeq\pr{\cat C}^{\circ}$ lying over a map $\underline{f}\from[n]\to[m]$
in $\Del^{\op}$, the functor
\begin{align*}
\Phi\pr f\from\pr{\Sigma\Del^{\op}_{\mathbb{F}}}_{\underline{f}} & \to\cat C^{\t}
\end{align*}
is defined as follows: The functor $\Phi\pr f$ extends $\Phi\pr{X_{1},\dots,X_{n}}$
and $\Phi\pr{Y_{1},\dots,Y_{m}}$. If $\pr{[n]\to[i,j],S_{i}\to\cdots\to S_{j}}\to\pr{[m]\to[k,l],T_{k}\to\cdots\to T_{l}}$
is a morphism of $\pr{\Sigma\Del^{\op}_{\mathbb{F}}}_{\overline{f}}$
lying over the map $0\to1$ in $[1]$, then its image $\pr{X_{p}\pr{S_{p-1,s}}}_{\pr{p,s}\in\coprod_{i<p\leq j}S_{p}}\to\pr{Y_{q}\pr{T_{q-1,t}}}_{\pr{q,t}\in\coprod_{k<q\leq l}T_{q}}$
is determined by the maps
\[
\bigotimes_{\phi\pr{q-1}<p\leq\phi\pr q}\bigotimes_{s\in S_{p,t}}X_{p}\pr{S_{p,s}}\to Y_{q}\pr{S_{\phi\pr{q-1},t}}\cong Y_{q}\pr{T_{q-1,t}},
\]
which in turn is determined by $f$. It is straightforward to check
that $\Phi$ is a categorical equivalence and has the stated naturality.
The claim follows.
\end{proof}

We now arrive at the proof of Theorem \ref{thm:Haugseng_univ}.
\begin{proof}
[Proof of Theorem \ref{thm:Haugseng_univ}]Proposition \ref{prop:Haug_vs_classical}
gives us a natural equivalence 
\[
\SSeq\pr{\mathbf{M}}^{\circ}\xrightarrow{\simeq}\SSeq_{\H}\pr{\mathbf{M}}^{\circ}.
\]
Write $\SSeq_{\H}\pr{\mathbf{M}}^{\circ}_{\cof}\subset\SSeq_{\H}\pr{\mathbf{M}}^{\circ}$
for the essential image of $\SSeq\pr{\mathbf{M}}^{\circ}_{\cof}$,
which is a monoidal $\infty$-category. Since every projectively cofibrant
symmetric sequence is injectively cofibrant (i.e., its values are
cofibrant), we may consider the composite natural transformation
\[
\theta_{\mathbf{M}}\from\SSeq_{\H}\pr{\mathbf{M}}^{\circ}_{\cof}\hookrightarrow\SSeq_{\H}\pr{\mathbf{M}_{\cof}}^{\circ}\to\SSeq_{\H}\pr{\mathbf{M}_{\infty}}^{\circ}.
\]
\textit{A priori}, the map $\theta_{\mathbf{M}}$ is merely lax monoidal,
i.e., a map of non-symmetric $\infty$-operads. (In fact, $\SSeq_{\H}\pr{\mathbf{M}_{\cof}}^{\circ}$
is generally not a monoidal $\infty$-category.) However, by Lemma
\ref{lem:comp_hocolim} and Proposition \ref{prop:SSeq_H_loc}, it
is in fact \textit{monoidal}. 

To complete the proof, it suffices to show that $\theta_{\mathbf{M}}$
is a monoidal localization (Remark \ref{rem:loc_unique}). By Remark
\ref{rem:SSeq_H_und}, the underlying functor of $\theta_{\mathbf{M}}$
can be identified with the composite
\[
\Fun\pr{\FB,\mathbf{M}}_{\cof}\xrightarrow{i}\Fun\pr{\FB,\mathbf{M}_{\cof}}\xrightarrow{j}\Fun\pr{\FB,\mathbf{M}_{\infty}}.
\]
The map $i$ induces an equivalence upon localizing at weak equivalences,
being the restriction of the Quillen equivalence between projective
and injective model structures to the full subcategories of cofibrant
objects. The map $j$ is a localization by \cite[Theorem 7.9.8]{HCHA}.
Hence $\theta_{\mathbf{M}}$ is a monoidal localization, and we are
done.
\end{proof}

\section{\label{sec:proof}Proof of the main result}

We can finally give a proof of Theorem \ref{thm:main}. 
\begin{proof}
[Proof of Theorem \ref{thm:main}]By Theorem \ref{thm:delocPrSM},
it suffices to show that the functors
\[
\SSeq_{\B}\pr{\pr -_{\infty}}^{\circ},\SSeq_{\H}\pr{\pr -_{\infty}}^{\circ}\from\TSMMC\to\Mon\hat{\Cat}_{\infty}
\]
are naturally equivalent. This follows from Theorems \ref{thm:Brantner_univ}
and \ref{thm:Haugseng_univ}.
\end{proof}

\appendix

\section{\label{sec:oo2}Results on $\protect\pr{\infty,2}$-categories}

In this section, we summarize basic definitions and terminology related
to $\pr{\infty,2}$-categories that we will need in the paper.

\subsection{Definitions}

We will use $\infty$-bicategories as our preferred model of $\pr{\infty,2}$-categories.
In this subsection, we recall the definitions and constructions related
to this model.
\begin{recollection}
\cite[$\S3$]{GC} A \textbf{scaled simplicial set} is a pair $\pr{X,T_{X}}$,
where $X$ is a simplicial set and $T_{X}$ is a set of $2$-simplices
of $X$ containing all the degenerate ones, whose elements are called
\textbf{thin triangles.} There is a model structure on the category
$\SS^{\sc}$ of scaled simplicial sets, called the bicategorical model
structure \cite[Theorem 4.2.7]{GC}. The bifibrant objects of this
model structure are called\textbf{ $\infty$-bicategories}. A morphism
of scaled simplicial sets that are $\infty$-bicategories are called
\textbf{functors} of $\infty$-bicategories. 
\end{recollection}

\begin{rem}
\label{rem:weakbicat}By \cite[Theorem 5.1]{GHL22}, the fibrant objects
of the bicategorical model structure are nothing but weak $\infty$-bicategories
in the sense of \cite[Definition 4.1.1]{GC}. It follows immediately
that:
\begin{enumerate}
\item If $\cat C$ is an $\infty$-category, then $\pr{\cat C,\cat C_{2}}$
is an $\infty$-bicategory. 
\item If $\cat C$ is an $\infty$-bicategory, then the simplicial subset
$\Und\pr{\cat C}$ of the underlying simplicial set of $\cat C$ consisting
of the simplices whose $2$-simplices are all thin is an $\infty$-category.
We refer to $\Und\pr{\cat C}$ as the \textbf{underlying $\infty$-category}
of $\cat C$, and call the objects and morphisms of $\Und\pr{\cat C}$
as \textbf{objects} and \textbf{morphisms} of $\cat C$. A morphism
of $\cat C$ is called an \textbf{equivalence} if it is an equivalence
in $\Und\pr{\cat C}$.
\item If $\pr{X,T_{X}}$ is an $\infty$-bicategory, then so is $\pr{X,T_{X}}^{\op}=\pr{X^{\op},T_{X}}$
\cite[Corollary 5.5]{GHL22}. 
\end{enumerate}
\end{rem}

\begin{rem}
\cite[Remark 1.31]{GHL22}\label{rem:bicat_cartesian} The $\infty$-bicategorical
model structure is cartesian. In other words, if $A\to B$ is a cofibration
and $X\to Y$ is a fibration of the bicategorical model structure,
then the induced map
\[
\Fun^{\sc}\pr{B,X}\to\Fun^{\sc}\pr{A,X}\times_{\Fun^{\sc}\pr{A,Y}}\Fun^{\sc}\pr{B,Y}
\]
is again a fibration. Here $\Fun^{\sc}\pr{-,-}$ denotes the internal
hom of $\SS^{\sc}$.
\end{rem}

\begin{example}
\label{exa:slice}Let $\cat C$ be an $\infty$-bicategory. For each
object $X\in\cat C$, we define a scaled simplicial set $\cat C^{X/}$
by the fiber product 
\[
\cat C^{X/}=\Fun^{\sc}\pr{[1],\cat C}\times_{\Fun^{\sc}\pr{\{0\},\cat C}}\{X\}.
\]
This is an $\infty$-bicategory by Remark \ref{rem:bicat_cartesian}.
\end{example}

\begin{rem}
For an $\infty$-category $\cat C$ and an object $X\in\cat C$, there
is another $\infty$-category $\cat C_{X/}$ equivalent to $\cat C^{X/}$.
The former usually goes under the name ``slice,'' while the latter
goes by the ``fat slice'' \cite[Definition 2.5.21]{Landoo-cat}.
As they are equivalent, there is no essential need to distinguish
between them, but we will adhere to this notational convention.
\end{rem}

\begin{rem}
\label{rem:bicatfib}A functor $p\from\cat C\to\cat D$ of $\infty$-bicategories
is a bicategorical fibration (i.e., fibration in the bicategorical
model structure) if and only if it has the right lifting property
for the following class of maps:
\begin{enumerate}
\item The class of scaled anodyne extensions \cite[Definition 3.1.3]{GC}.
Recall that this is generated by the following maps:
\begin{enumerate}
\item For $0<i<n$, the inclusion 
\[
\pr{\Lambda^{n}_{i},\pr{\Lambda^{n}_{i}}_{2}\cap T}\to\pr{\Delta^{n},T},
\]
where $T$ denotes the set of degenerate $2$-simplices of $\Delta^{n}$
and the simplex $\Delta^{\{i-2,i,i+1\}}$.
\item The map $\pr{\Delta^{4},T}\to\pr{\Delta^{4},T\cup\Delta^{\{0,3,4\}}\cup\Delta^{\{0,1,4\}}}$,
where $T$ is the union of the degenerate $2$-simplices and the simplices
$\Delta^{\{0,2,4\}}$, $\Delta^{\{1,2,3\}}$, $\Delta^{\{0,1,3\}}$,
$\Delta^{\{1,3,4\}}$, $\Delta^{\{0,1,2\}}$.\footnote{The original reference by Lurie says $T\cup\Delta^{\{0,3,4\}}\cup\Delta^{\{1,3,4\}}$
instead of $T\cup\Delta^{\{0,3,4\}}\cup\Delta^{\{0,1,4\}}$, but this
is a typo, as corrected in \cite[Definition 1.17]{GHL22}.}
\item For $n>2$, the inclusion
\[
\pr{\Lambda^{n}_{0}\amalg_{\Delta^{\{0,1\}}}\Delta^{0},T}\to\pr{\Delta^{n}\amalg_{\Delta^{\{0,1\}}}\Delta^{0},T},
\]
where $T$ is the set of degenerate $2$-simplices of $\Delta^{n}\amalg_{\Delta^{\{0,1\}}}\Delta^{0}$
and the image of the $2$-simplex $\Delta^{\{0,1,n\}}$.
\end{enumerate}
\item The inclusion $\pr{\{\varepsilon\},\{\varepsilon\}}\to\pr{J,J_{2}}$
for $\varepsilon\in\{0,1\}$.
\end{enumerate}
This follows from Remark \ref{rem:bicat_cartesian} and the characterization
of the bicategorical model structure given in \cite[Corollary 6.4.]{GHL22}.

We remark that condition (1) is automatic if $\cat D$ is the scaled
nerve of an ordinary category. This is because such a scaled simplicial
set has a unique filler for inner horns, and at most one filler for
outer horns in dimensions higher than $2$.
\end{rem}

We will often identify $\infty$-bicategories with their underlying
simplicial sets. This is justified by the following proposition:
\begin{prop}
Let $\overline{X}=\pr{X,T_{X}}$ be a scaled simplicial set. If $\overline{X}$
is an $\infty$-bicategory, then the set $T_{X}$ of thin triangles
is completely determined by the underlying simplicial set $X$.
\end{prop}

\begin{proof}
Consider the set $T'_{X}$ of $2$-simplices $\sigma\in X_{2}$ with
the following property:
\begin{itemize}
\item [($\ast$)] Let $n\in\{3,4\}$, and let $\phi\from\Lambda^{n}_{1}\to X$
be a map of simplicial sets. Suppose that $\phi\vert\Delta^{\{n-2,n-1,n\}}$
factors through $\Delta^{1}$ via the map carrying $n-2$ to $0$
and $n-1$ and $n$ to $1$. Then $\phi$ extends to a map $\Delta^{n}\to X$.
\end{itemize}
We will show that $T'_{X}=T_{X}$. Since $T'_{X}$ depends only on
$X$, this will prove the claim.

By Remark \ref{rem:weakbicat} and the definition of weak $\infty$-bicategories,
we have $T_{X}\subset T'_{X}$. To prove the reverse inclusion, suppose
we are given a $2$-simplex $\sigma\in T'_{X}$ depicted as 
\[\begin{tikzcd}
	& y \\
	x && z.
	\arrow["g", from=1-2, to=2-3]
	\arrow["f", from=2-1, to=1-2]
	\arrow["h"', from=2-1, to=2-3]
\end{tikzcd}\]We must show that $\sigma$ is thin. For this, we use the fact that
$X$ is an $\infty$-bicategory to find a thin $2$-simplex $\tau$
depicted as
\[\begin{tikzcd}
	& y \\
	x && z.
	\arrow["g", from=1-2, to=2-3]
	\arrow["f", from=2-1, to=1-2]
	\arrow["gf"', from=2-1, to=2-3]
\end{tikzcd}\]We then use condition ($\ast$) for $n=3$ (applied to $\Delta^{\{0,1,2,3\}}$
and $\Delta^{\{0,1,3,4\}}$) to construct a map 
\[
\psi\from\Lambda^{4}_{1}=\Delta^{\{0,1,2,3\}}\cup\Delta^{\{0,1,2,4\}}\cup\Delta^{\{0,1,3,4\}}\cup\Delta^{\{1,2,3,4\}}\to X
\]
whose restrictions to the $2$-dimensional faces are given by
\begin{align*}
\psi\vert\Delta^{\{0,1,2\}} & =\sigma,\,\psi\vert\Delta^{\{0,1,3\}}=\tau,\,\psi\vert\Delta^{\{1,2,3\}}=1_{g},\\
\psi\vert\Delta^{\{0,1,3\}} & =\tau,\,\psi\vert\Delta^{\{0,1,4\}}=\sigma,\,\psi\vert\Delta^{\{1,3,4\}}=1_{g},
\end{align*}
and such that $\psi\vert\Delta^{\{0,1,2,4\}}$ is a degeneration of
$\sigma$ and $\psi\vert\Delta^{\{1,2,3,4\}}$ is a degeneration of
$g$. Below is a picture of $\psi$:
\[\begin{tikzcd}
	&& x \\
	\\
	y && z \\
	& z && z.
	\arrow["f"', from=1-3, to=3-1]
	\arrow["h"{description}, from=1-3, to=3-3]
	\arrow["gf"{description}, from=1-3, to=4-2]
	\arrow["h"{description}, from=1-3, to=4-4]
	\arrow["g", from=3-1, to=3-3]
	\arrow["g"', from=3-1, to=4-2]
	\arrow[equal, from=3-3, to=4-4]
	\arrow[equal, from=4-2, to=3-3]
	\arrow[equal, from=4-2, to=4-4]
\end{tikzcd}\] 

Using ($\ast$), we can extend $\psi$ to a map $\overline{\psi}\from\Delta^{4}\to X$.
Since $\pr{X^{\op},T_{X}}$ is an $\infty$-bicategory and the restriction
of $\psi$ to $\Delta^{\{0,2,4\}},\Delta^{\{1,2,3\}},\Delta^{\{1,3,4\}},\Delta^{\{0,1,3\}},\Delta^{\{2,3,4\}}$
are thin ($\psi\vert\Delta^{\{0,1,3\}}=\tau$ is thin, and the rest
are all degenerate), the definition of $\infty$-bicategory shows
that $\overline{\psi}\vert\Delta^{\{0,1,4\}}=\sigma$ is thin, as
required.
\end{proof}

\begin{rem}
In \cite[01W9]{kerodon}, Lurie introduces defines an $\pr{\infty,2}$-category
to be a simplicial set satisfying some conditions. The underlying
simplicial set of an $\infty$-bicategory is an $\pr{\infty,2}$-bicategory
in this sense. The converse is expected to be true, but a proof seems
to be missing in the literature \cite{MO461820}.
\end{rem}

Many $\infty$-bicategories in this paper arise from the following
``nerve'' constructions:
\begin{example}
\label{exa:Duskin}Let $\cat C$ be a bicategory. The \textbf{Duskin
nerve} \cite[009U]{kerodon} of $\cat C$ is an $\infty$-bicategory.
The Duskin nerve determines a fully faithful functor from the category
of bicategories and strictly unitary lax functors into the category
of simplicial sets \cite[00AU]{kerodon}. Because of this, we often
identify bicategories with their Duskin nerve.
\end{example}

\begin{recollection}
\cite[Theorem 4.2.7]{GC}\label{recall:Bergner} Let $\SS^{+}$ denote
the category of marked simplicial sets (i.e., pairs $\pr{X,E_{X}}$,
where $X$ is a simplicial set and $E_{X}$ is a set of edges of $X$
containing all the degenerate ones. The category $\SS^{+}$ admits
a model structure whose bifibrant objects are the $\infty$-categories
(quasicategories) with equivalences marked. The category $\mathsf{Cat}_{\mathsf{sSet}^{+}}$
of $\SS^{+}$-enriched categories admits an induced model structure,
and there is a Quillen equivalence
\[
\mathfrak{C}^{\sc}\from\mathsf{sSet}^{\mathrm{sc}}\adj\mathsf{Cat}_{\mathsf{sSet}^{+}}\from N^{\sc}.
\]
The right adjoint of this adjunction is called the \textbf{scaled
nerve} functor.
\end{recollection}

\begin{defn}
We define the $\infty$-bicategory $\Cat^{\pr 2}_{\infty}$ of $\infty$-categories
to be the scaled nerve of the $\SS^{+}$-enriched category of $\infty$-categories,
with mapping objects given by $\pr{\Fun\pr{\cat C,\cat D},\{\text{equivalences}\}}$.
Likewise, we define the $\infty$-bicategory $\BiCat^{\pr 2}_{\infty}$
of $\infty$-categories as the scaled nerve of the $\SS^{+}$-enriched
category of $\infty$-bicategories, with mapping objects given by
$\pr{\Und\pr{\Fun^{\sc}\pr{\cat C,\cat D}},\{\text{equivalences}\}}$.
Their underlying $\infty$-categories are denoted by $\Cat_{\infty}$
and $\BiCat_{\infty}$.
\end{defn}

\subsection{Mapping $\infty$-categories}

To each $\infty$-bicategory $\cat C$ and each pair of objects $X,Y\in\cat C$,
we can associate an $\infty$-category $\cat C\pr{X,Y}$, called the
\textbf{mapping $\infty$-category}. In this subsection, we define
them and give a formula of the mapping $\infty$-categories of the
arrow $\infty$-bicategory of $\infty$-bicategories.
\begin{defn}
\label{def:map_cat}Let $\cat C$ be an $\infty$-bicategory. The
mapping category functor $\cat C\pr{-,-}\from\cat C^{\op}\times\cat C\to\Cat^{\pr 2}_{\infty}$
is defined as the composite
\[
\cat C^{\op}\times\cat C\to N^{\sc}\pr{\cat C_{+}}^{\op}\times N^{\sc}\pr{\cat C_{+}}\cong N^{\sc}\pr{\cat C^{\op}_{+}\times\cat C_{+}}\xrightarrow{\cat C_{+}\pr{-,-}}\Cat^{\pr 2}_{\infty},
\]
where the functor $\cat C_{+}$ is a fibrant $\SS^{+}$-enriched category
equipped with an equivalence $\cat C\to N^{\sc}\pr{\cat C_{+}}$ of
$\infty$-bicategories, and $\cat C_{+}\pr{-,-}$ denotes the hom-functor
of $\cat C_{+}$ (with markings dropped).

If $f\from\cat C\to\cat D$ is a functor of $\infty$-bicategories,
then we define a natural transformation $\alpha\from\cat C\pr{-,-}\to\cat D\pr{f-,f-}$
as follows: Choose $\cat C_{+}$ and $\cat D_{+}$ so that there is
a $\SS^{+}$-enriched functor $f_{+}\from\cat C_{+}\to\cat D_{+}$
rendering the diagram
\[\begin{tikzcd}
	{\mathcal{C}} & {N^{\mathrm{sc}}(\mathcal{C}_+)} \\
	{\mathcal{D}} & {N^{\mathrm{sc}}(\mathcal{D}_+)}
	\arrow[from=1-1, to=1-2]
	\arrow[from=1-1, to=2-1]
	\arrow[from=1-2, to=2-2]
	\arrow[from=2-1, to=2-2]
\end{tikzcd}\]commutative. Then $\alpha$ is induced by $f_{+}$. 
\end{defn}

\begin{defn}
Let $f\from\cat C\to\cat D$ be a functor of $\infty$-bicategories.
We say that $f$ is:
\begin{itemize}
\item \textbf{fully faithful} if for each pair of objects $X,Y\in\cat C$,
the map
\[
\cat C\pr{X,Y}\to\cat D\pr{fX,fY}
\]
is an equivalence of $\infty$-categories. 
\item \textbf{essentially surjective} if it is essentially surjective on
the level of underlying $\infty$-categories.
\end{itemize}
\end{defn}

\begin{rem}
\cite[Remark 4.2.1]{GC} Let $\cat C$ be an $\infty$-bicategory
and let $\cat C'\subset\cat C$ be a full sub $\infty$-bicategory
(i.e., a full simplicial subset). Then the inclusion $\cat C'\hookrightarrow\cat C$
is fully faithful.
\end{rem}

\begin{rem}
\cite[(Argument of) Lemma 4.2.4]{GC} A functor of $\infty$-bicategories
is a weak equivalence in $\SS^{+}$ if and only if it fully faithful
and essentially surjective. We call such a map a \textbf{bicategorical
equivalence} or an \textbf{equivalence of $\infty$-bicategories}.
\end{rem}

In the main body of the paper, we will often need to identify the
mapping spaces of arrow $\infty$-bicategories. The following proposition
will be useful for this. We refer the readers to \cite[Theorem 4.1]{AS23}
and \cite[Example A.2.6]{BB24} for different flavors of similar results.
\begin{prop}
\label{prop:mappingcat_arrow}\hfill 
\begin{enumerate}
\item Let $\cat C$ be an $\infty$-bicategory, and let $f\from X_{0}\to X_{1}$
and $g\from Y_{0}\to Y_{1}$ be morphisms in $\cat C$. The mapping
category of $\Ar\pr{\cat C}=\Fun^{\sc}\pr{[1],\cat C}$ fits into
the pullback square 
\[\begin{tikzcd}
	{\operatorname{Ar}(\mathcal{C})(f,g)} & {\mathcal{C}(X_1,Y_1)} \\
	{\mathcal{C}(X_0,Y_0)} & {\mathcal{C}(X_0,Y_1)}
	\arrow["{\operatorname{ev}_1}", from=1-1, to=1-2]
	\arrow["{\operatorname{ev}_0}"', from=1-1, to=2-1]
	\arrow["{a^*}", from=1-2, to=2-2]
	\arrow["{b_\ast}"', from=2-1, to=2-2]
\end{tikzcd}\]which is natural in $f,g\in\Ar\pr{\cat C}$.
\item Let $\mathsf{C}$ be a fibrant $\SS^{+}$-enriched category. Call
a morphism $f\from X_{0}\to X_{1}$ a \textbf{quasi-fibration} if
for each object $C\in\mathsf{C}$, the map
\[
f_{\ast}\from\mathsf{C}\pr{C,X_{0}}\to\mathsf{C}\pr{C,X_{1}}
\]
is a fibration in $\SS^{+}$. The full subcategory $\mathsf{C}^{[1]}_{\mathrm{qfib}}\subset\mathsf{C}^{[1]}$
spanned by quasi-fibrations is a fibrant $\SS^{+}$-enriched category,
and the functor
\[
\phi\from N^{\sc}\pr{\mathsf{C}^{[1]}_{\mathrm{qfib}}}\to\Fun^{\sc}\pr{[1],N^{\sc}\pr{\mathsf{C}}}
\]
is fully faithful.
\end{enumerate}
\end{prop}

\begin{proof}
For (1), we may assume that $\cat C=N^{\sc}\pr{\mathsf{C}}$ for some
fibrant $\SS^{+}$-category $\mathsf{C}$. Let $\mathbf{A}=\pr{\SS^{+}}^{\mathsf{C}^{\op}}$
denote the $\SS^{+}$-enriched category of $\SS^{+}$-enriched functors
$\mathsf{C}^{\op}\to\SS^{+}$, equipped with the projective model
structure. By the enriched Yoneda embedding, the functor $\mathsf{C}\to\mathbf{A}$
is fully faithful (i.e., induces isomorphisms between hom objects).
Moreover, since $\mathsf{C}$ is fibrant, it takes values in the full
subcategory $\mathbf{A}^{\circ}\subset\mathbf{A}$ of bifibrant objects.
Therefore, it will suffice to prove the assertion for $N^{\sc}\pr{\mathbf{A}^{\circ}}$.
In this case, \cite[Proposition 3.89]{ASgroth2} gives an equivalence
\[
N^{\sc}\pr{\pr{\mathbf{A}^{\circ}}^{[1]}}\xrightarrow{\simeq}\Fun^{\sc}\pr{[1],N^{\sc}\pr{\mathbf{A}^{\circ}}}.
\]
The mapping space of $\pr{\mathbf{A}^{\circ}}^{[1]}$ from $f\from A_{0}\to A_{1}$
to $g\from B_{0}\to B_{1}$ is given by the pullback 
\[\begin{tikzcd}
	{\mathbf{A}^{[1]}(f,g)} & {\mathbf{A}(A_1,B_1)} \\
	{\mathbf{A}(A_0,B_0)} & {\mathbf{A}(A_0,B_1),}
	\arrow["{\operatorname{ev}_1}", from=1-1, to=1-2]
	\arrow["{\operatorname{ev}_0}"', from=1-1, to=2-1]
	\arrow["{f^*}", from=1-2, to=2-2]
	\arrow["{g_\ast}"', from=2-1, to=2-2]
\end{tikzcd}\]which is a homotopy pullback because $f^{*}$ is a fibration. The
claim readily follows.

The proof of (2) is similar to that of (1) but is more elaborate.
The fibrancy of $\mathsf{C}^{[1]}_{\mathrm{qfib}}$ is clear from
the definition of quasi-fibrations. To show that $\phi$ is fully
faithful, let $\mathbf{A}=\Fun^{+}\pr{\mathsf{C}^{\op},\SS^{+}}$
be as above, and identify $\mathsf{C}$ with a full sub $\SS^{+}$-category
via the enriched Yoneda embedding. We will also write $\mathbf{A}^{\circ}\subset\mathbf{A}$
for the full $\SS^{+}$-category spanned by the fibrant-cofibrant
objects.

Let us say that an object $A\in\mathbf{A}^{\circ}$ is \textbf{good}
if for each quasi-fibration $X_{0}\to X_{1}$ in $\mathsf{C}$, the
induced map $\mathbf{A}\pr{A,X_{0}}\to\mathbf{A}\pr{A,X_{1}}$ is
a fibration. We consider the following categories:
\begin{itemize}
\item The full sub $\SS^{+}$-category $\mathbf{A}_{\mathrm{good}}\subset\mathbf{A}^{\circ}$
spanned by the good objects.
\item The full sub $\SS^{+}$-category $\pr{\mathbf{A}_{\mathrm{good}}}^{[1]}_{\mathrm{qfib}}\subset\pr{\mathbf{A}_{\mathrm{good}}}^{[1]}$
spanned by the quasi-fibrations in $\mathbf{A}_{\mathrm{good}}$. 
\end{itemize}
Since every object in $\mathsf{C}$ is good, and since every quasi-fibration
in $\mathsf{C}$ is a quasi-fibration in $\mathbf{A}_{\mathrm{good}}$,
we have the following commutative diagram:
\[\begin{tikzcd}
	{N^{\mathrm{sc}}(\mathsf{C}^{[1]}_{\mathrm{qfib}})} & {\operatorname{Fun}^{\mathrm{sc}}([1],N^{\mathrm{sc}}(\mathsf{C}))} \\
	{N^{\mathrm{sc}}((\mathbf{A}_{\mathrm{good}})^{[1]}_{\mathrm{qfib}})} & {\operatorname{Fun}^{\mathrm{sc}}([1],N^{\mathrm{sc}}(\mathbf{A}^\circ))}
	\arrow["\phi", from=1-1, to=1-2]
	\arrow[hook, from=1-1, to=2-1]
	\arrow[hook, from=1-2, to=2-2]
	\arrow["\psi", from=2-1, to=2-2]
\end{tikzcd}\]Our goal is to show that $\phi$ is fully faithful. Since the vertical
arrows are fully faithful, it will suffice to show that $\psi$ is
fully faithful.

Let $f\from A_{0}\to A_{1}$ and $g\from B_{0}\to B_{1}$ be arbitrary
objects in $\pr{\mathbf{A}_{\mathrm{good}}}^{[1]}_{\mathrm{qfib}}$.
We wish to show that the map
\[
N^{\sc}\pr{\pr{\mathbf{A}_{\mathrm{good}}}^{[1]}_{\mathrm{qfib}}}\pr{f,g}\to\Fun^{\sc}\pr{[1],N^{\sc}\pr{\mathbf{A}^{\circ}}}\pr{f,g}
\]
is an equivalence. For this, find cofibrant replacements $\alpha\from f'\xrightarrow{\simeq}f$
and $\beta\from g'\xrightarrow{\simeq}g$ in $\mathbf{A}^{[1]}$,
where $\mathbf{A}^{[1]}$ is equipped with the projective model structure.
(This means $f'$ and $g'$ are cofibrations of cofibrant objects.)
Consider the following $\SS^{+}$-category $\mathbf{X}$:
\begin{itemize}
\item The collection of objects of $\mathbf{X}$ is $\ob\pr{\mathbf{A}_{\mathrm{good}}}^{[1]}_{\mathrm{qfib}}\amalg\{f'\}\amalg\{g'\}$.
\item Mapping categories are given by 
\[
\mathbf{X}\pr{a,b}=\begin{cases}
\emptyset & \text{if }a\in\ob\pr{\mathbf{A}_{\mathrm{good}}}^{[1]}_{\mathrm{qfib}}\text{ and }b\in\{f'\}\amalg\{g'\},\\
\mathbf{A}^{[1]}\pr{a,b} & \text{otherwise.}
\end{cases}
\]
\end{itemize}
The maps
\[
\mathbf{X}\pr{f,g}\to\mathbf{X}\pr{f',g}\ot\mathbf{X}\pr{f',g'}
\]
are equivalences in $\SS^{+}$. Also, the images of $\alpha$ and
$\beta$ in $\Fun^{\sc}\pr{[1],N^{\sc}\pr{\mathbf{A}^{\circ}}}$ are
equivalences, too. Therefore, it suffices to show that the map
\[
N^{\sc}\pr{\mathbf{X}}\pr{f',g'}\to\Fun^{\sc}\pr{[1],N^{\sc}\pr{\mathbf{X}^{\circ}}}\pr{f',g'}
\]
is an equivalence. The full sub $\SS^{+}$-category of $\mathbf{X}$
spanned by $f'$ and $g'$ maps fully faithfully into $\pr{\mathbf{A}^{\circ}}^{[1]}$,
so we are reduced to showing that the map
\[
N^{\sc}\pr{\pr{\mathbf{A}^{\circ}}^{[1]}}\pr{f',g'}\to\Fun^{\sc}\pr{[1],N^{\sc}\pr{\mathbf{A}^{\circ}}}\pr{f',g'}
\]
is an equivalence. This follows from \cite[Proposition 3.89]{ASgroth2}.
\end{proof}

\subsection{Cartesian fibrations}

The straightening--unstraightening equivalence is a fundamental construction
in $\infty$-category theory. Briefly, it gives an alternative presentation
of $\Cat_{\infty}$-valued functors in terms of (co)cartesian fibrations,
which are often easier to handle than $\Cat_{\infty}$-valued functors
themselves. In this subsection, we revisit this equivalence in the
$\infty$-bicategorical setting.

\begin{defn}
\label{def:cart}Let $p\from\cat E\to\cat B$ be a functor of $\infty$-bicategories.
A morphism $f\from X\to Y$ of $\cat E$ is said to be \textbf{$p$-cartesian}
if for each object $E\in\cat E$, the square 
\[\begin{tikzcd}
	{\mathcal{E}(E,X)} & {\mathcal{E}(E,Y)} \\
	{\mathcal{B}(p(E),p(X))} & {\mathcal{B}(p(E),p(Y))}
	\arrow["{f_\ast}", from=1-1, to=1-2]
	\arrow[from=1-1, to=2-1]
	\arrow[from=1-2, to=2-2]
	\arrow["{p(f)_\ast}"', from=2-1, to=2-2]
\end{tikzcd}\]is cartesian in $\Cat_{\infty}$. Dually, we say that $f$ is $p$-cocartesian
if it is $p^{\op}$-cocartesian.

Assume now that $\cat B$ is an \textit{$\infty$-category}. We say
that $p$ is a \textbf{cartesian fibration} if it satisfies the following
pair of conditions:
\begin{enumerate}
\item $p$ is a fibration in the bicategorical model structure; and
\item For each object $X\in\cat E$ and each morphism $f\from p\pr X\to B$
in $\cat B$, there is a $p$-cartesian morphism $X\to Y$ lying over
$f$.
\end{enumerate}
\textbf{Cocartesian fibrations} are defined dually.
\end{defn}

\begin{rem}
In order to maximize efficiency, our treatment of cartesian fibrations
of $\infty$-bicategories is deliberately\textit{ }incomplete. For
example, we can define cartesian fibrations over an arbitrary $\infty$-bicategory,
but we have decided not to include the definitions because there are
four flavors (inner and outer (co)cartesian fibrations) instead of
two. 
\end{rem}

\begin{rem}
In the setting of scaled simplicial sets, Definition \ref{def:cart}
is not a standard definition of variations of cartesian fibrations.
Proving the equivalence between our definition and a more established
definition is somewhat technical and is deferred to Subsection \ref{subsec:more_cart}.
\end{rem}

\begin{example}
\label{exa:cc_po}Let $\cat C$ be an $\infty$-bicategory, and let
\[\begin{tikzcd}
	{X_0} & {X'_0} \\
	{X_1} & {X'_1}
	\arrow[from=1-1, to=1-2]
	\arrow["f"', from=1-1, to=2-1]
	\arrow["{f'}", from=1-2, to=2-2]
	\arrow[from=2-1, to=2-2]
\end{tikzcd}\]be a diagram in $\Und\pr{\cat C}$. Suppose that the square is cocartesian
in $\cat C$ (i.e., for each $C\in\cat C$, the functor $\cat C\pr{-,C}$
carries the square into a pullback square of $\infty$-categories).
Then the morphism $f\to f'$ in $\Ar\pr{\cat C}=\Fun^{\bi}\pr{[1],\cat C}$
is cocartesian for the projection $\Ar\pr{\cat C}\to\Fun^{\bi}\pr{\{0\},\cat C}\cong\cat C$.
This follows from Proposition \ref{prop:mappingcat_arrow}.
\end{example}

Just like cartesian fibrations of $\infty$-categories, there is a
straightening--unstraightening equivalence for $\infty$-bicategorical
cartesian fibrations. To state this, we introduce the following notation:
\begin{notation}
Let $\cat B$ be an $\infty$-category. We let $\mathsf{Cart}\pr{\cat B}$
denote the following category:
\begin{itemize}
\item Objects are cartesian fibrations $\cat E\to\cat B$ of $\infty$-bicategories.
\item Morphisms are commutative diagrams
\[\begin{tikzcd}
	{\mathcal{E}} && {\mathcal{E}'} \\
	& {\mathcal{B}}
	\arrow["f", from=1-1, to=1-3]
	\arrow[from=1-1, to=2-2]
	\arrow[from=1-3, to=2-2]
\end{tikzcd}\]of scaled simplicial sets, where $f$ preserves cartesian edges.
\end{itemize}
We let $\Cart\pr{\cat B}$ denote the $\infty$-categorical localization
of $\mathsf{Cart}\pr{\cat B}$ at the maps whose underlying maps are
bicategorical equivalences. 
\end{notation}

We then have the following $\infty$-bicategorical straightening--unstraightening
equivalence:
\begin{thm}
\cite[Theorem 3.85]{ASgroth2}\label{thm:st} For every $\infty$-category
$\cat B$, there is an equivalence of $\infty$-categories
\[
\Fun\pr{\cat B^{\op},\BiCat_{\infty}}\simeq\Cart\pr{\cat B}.
\]
\end{thm}

\begin{rem}
As in \cite[Appendix A]{GHN17}, the equivalence of Theorem \ref{thm:st}
is natural in $\cat B\in\Cat_{\infty}$.
\end{rem}

\begin{rem}
\cite[Theorem 4.21]{ASgroth2}\label{rem:strict_st} The equivalence
of Theorem \ref{thm:st} is a refinement of the classical $2$-categorical
Grothendieck construction. More precisely, let $2\mathsf{Cat}$ denote
the category of $2$-categories (i.e., $\mathsf{Cat}$-enriched categories)
and $2$-functors. Given a category $\mathbf{B}$ and a functor $F\from\mathbf{B}^{\op}\to2\mathsf{Cat}$,
we can form its Grothendieck construction $\pi\from\int F\to\mathbf{B}$.
Explicitly:
\begin{itemize}
\item Objects of $\int F$ are pairs $\pr{B,X}$.
\item The mapping categories are given by
\[
\pr{\int F}\pr{\pr{B_{0},X_{0}},\pr{B_{1},X_{1}}}=\coprod_{f\in\mathbf{B}\pr{B_{0},B_{1}}}\pr{FB_{1}}\pr{Ff\pr{X_{0}},X_{1}}.
\]
\end{itemize}
The Duskin nerve of $\pi$ is a cartesian fibration, and it corresponds
to the composite
\[
\mathbf{B}^{\op}\to2\mathsf{Cat}\xrightarrow{N^{D}}\BiCat_{\infty}.
\]
\end{rem}

\begin{variant}
\label{var:bicategoricalgroth}Let $\mathbf{B}$ be a category and
$F\from\mathbf{B}^{\op}\to\mathsf{BiCat}$ be a functor, where $\mathsf{BiCat}$
denotes the category of bicategories and pseudofunctors between them.
We can form the Grothendieck construction $\int F$ as in Remark \ref{rem:strict_st},
which is a bicategory. The Duskin nerve of the projection $\pi\from\int F\to\mathbf{B}$
is a cartesian fibration, and it corresponds to the composite
\[
\mathbf{B}^{\op}\to\mathsf{BiCat}\xrightarrow{N^{D}}\BiCat_{\infty}.
\]
To see this, note that the strictification functor $\mathrm{st}\from\mathsf{BiCat}\to2\mathsf{Cat}$
of \cite[$\S$ 4.10]{GPS95} gives an equivalence of bicategories
\[
\int F\xrightarrow{\simeq}\int\mathrm{st}\circ F.
\]
The claim then follows from this equivalence and Remark \ref{rem:strict_st}
Remark \ref{rem:bicatfib}.
\end{variant}

\subsection{\label{subsec:more_cart}More on cartesian fibrations}

In this subsection, we will prove that cartesian fibrations in the
sense of Definition \ref{def:cart} are nothing but $\mathbf{O2C}$-fibrations
(or outer $2$-cartesian fibrations) in the sense of \cite[Definition 4.22]{AS22_I}
(Corollary \ref{prop:O2C-1}).
\begin{convention}
For disambiguation, we make the following convention throughout this
subsection: We will refer to cartesian morphisms and cartesian fibrations
in the sense of Definition \ref{def:cart} as \textbf{fake cartesian
morphisms} and fake cartesian fibrations.
\end{convention}

With this convention, the main result of this subsection can be stated
as follows:
\begin{prop}
\label{prop:fakeO2C}Let $p\from\cat E\to\cat B$ be a map of $\infty$-bicategories,
where $\cat B$ is an $\infty$-category. Then $p$ is a fake cartesian
fibration if and only if it is an $\mathbf{O2C}$-fibration.
\end{prop}

The rest of this subsection is devoted to the proof of Proposition
\ref{prop:fakeO2C}. 

We start by recalling the definition of $\mathbf{O2C}$-fibrations.
\begin{defn}
\cite[Definitions 2.3 and 4.7]{AS22_I} Let $X$ be a scaled simplicial
set, and let $\sigma$ be a $2$-simplex of $X$, depicted as
\[\begin{tikzcd}
	& Y \\
	X && Z.
	\arrow["g", from=1-2, to=2-3]
	\arrow["f", from=2-1, to=1-2]
	\arrow[""{name=0, anchor=center, inner sep=0}, "h"', from=2-1, to=2-3]
	\arrow["\sigma", between={0}{0.8}, Rightarrow, from=1-2, to=0]
\end{tikzcd}\]
\begin{itemize}
\item We say that $\sigma$ is \textbf{left degenerate} if $\sigma\vert\Delta^{\{0,1\}}$
is degenerate. 
\item A \textbf{left degeneration} of $\sigma$ is a $2$-simplex $\tau$
which admits a $3$-simplex $\rho\from\Delta^{3}\to X$ depicted as
\[\begin{tikzcd}
	& X && Y &&& X && Y \\
	\\
	X &&&& Z & X &&&& {Z,}
	\arrow["f", from=1-2, to=1-4]
	\arrow[""{name=0, anchor=center, inner sep=0}, from=1-2, to=3-5]
	\arrow["g", from=1-4, to=3-5]
	\arrow["f", from=1-7, to=1-9]
	\arrow["g", from=1-9, to=3-10]
	\arrow["{1_X}", from=3-1, to=1-2]
	\arrow[""{name=1, anchor=center, inner sep=0}, "h"', from=3-1, to=3-5]
	\arrow["{1_X}", from=3-6, to=1-7]
	\arrow[""{name=2, anchor=center, inner sep=0}, "f"', from=3-6, to=1-9]
	\arrow[""{name=3, anchor=center, inner sep=0}, "h"', from=3-6, to=3-10]
	\arrow["\tau"', between={0}{0.8}, Rightarrow, from=1-2, to=1]
	\arrow["\alpha", between={0}{0.8}, Rightarrow, from=1-4, to=0]
	\arrow["{1_f}", between={0}{0.8}, Rightarrow, from=1-7, to=2]
	\arrow["\sigma", between={0}{0.8}, Rightarrow, from=1-9, to=3]
\end{tikzcd}\]where $\alpha$ is thin.
\end{itemize}
\end{defn}

\begin{defn}
\cite[Definition 2.14]{AS22_I} A map of scaled simplicial sets is
called a \textbf{weak $\mathbf{S}$-fibration} if it has the right
lifting property for the scaled anodyne maps.
\end{defn}

\begin{defn}
\cite[Definitions 4.14]{AS22_I} Let $p\from\mathcal{E}\to\mathcal{B}$
be a weak $\mathbf{S}$-fibration, where $\mathcal{B}$ is an $\infty$-bicategory.
We say that an edge $e\from\Delta^{1}\to\mathcal{E}$ is \textbf{$p$-cartesian}
(resp. \textbf{strongly $p$-cartesian}) if it satisfies the following
conditions:
\begin{itemize}
\item Let $n\geq2$, and consider a lifting problem 
\[\begin{tikzcd}
	{\Delta^{\{n-1,n\}}} & {\Lambda^n_n} & {\mathcal{E}} \\
	& {\Delta^n} & {\mathcal{B}}
	\arrow[from=1-1, to=1-2]
	\arrow["e", curve={height=-18pt}, from=1-1, to=1-3]
	\arrow["f", from=1-2, to=1-3]
	\arrow[from=1-2, to=2-2]
	\arrow["p", from=1-3, to=2-3]
	\arrow["{\widehat{f}}"{description}, dashed, from=2-2, to=1-3]
	\arrow[from=2-2, to=2-3]
\end{tikzcd}\]Suppose further that $f\vert\Delta^{\{0,n-1,n\}}$ is thin (resp.
$p$-cocartesian) when $n\geq3$. Then there is a dashed filler $\hat f$
such that $\hat f\vert\Delta^{\{0,n-1,n\}}$ is thin (resp. $p$-cocartesian).
\end{itemize}
\end{defn}

\begin{rem}
\label{rem:fakecart}Let $p\from\cat E\to\cat B$ be a bicateogrical
fibration of $\infty$-bicategories. Then a morphism of $\cat E$
is $p$-cartesian if and only if it is fake $p$-cartesian. This follows
from \cite[Lemma 2.4.2 and Proposition 4.2.7]{GHL24_cart}. 
\end{rem}

\begin{defn}
Let $\cat C$ be an $\infty$-bicategory, and let $X,Y\in\cat C$
be objects of $\cat C$. We let $\Hom^{R}_{\cat C}\pr{X,Y}$ denote
the simplicial set whose $n$-simplices are the maps $\phi\from\Delta^{n}\star\Delta^{0}\to\cat C$
such that $\phi\vert\Delta^{n}$ and $\phi\vert\Delta^{0}$ are the
constant maps at $X$ and $Y$. (This is the underlying simplicial
set of $\Hom^{\rcone}_{\cat C}\pr{X,Y}$ defined in \cite[$\S$ 2.3]{GHL22}.)
\end{defn}

\begin{rem}
\cite[Corollary 2.26]{GHL22}\label{rem:Hom_innfib} Let $p\from\mathcal{C}\to\mathcal{D}$
be a weak $\mathbf{S}$-fibration of $\infty$-bicategories. For every
pair of objects $X,Y\in\cat C$, the induced map
\[
\Hom^{R}_{\cat C}\pr{X,Y}\to\Hom^{R}_{\cat D}\pr{p\pr X,p\pr Y}
\]
is an inner fibration of $\infty$-categories. In particular, $\Hom^{R}_{\cat C}\pr{X,Y}$
is an $\infty$-category. Moreover, its equivalences are precisely
the thin triangles.
\end{rem}

\begin{defn}
\cite[Definition 4.1, Proposition 4.4]{AS22_I} Let $p\from\mathcal{E}\to\mathcal{B}$
be a weak $\mathbf{S}$-fibration, where $\mathcal{B}$ is an $\infty$-bicategory.
We say that a left degenerate $2$-simplex $\sigma$ of $X$ is \textbf{$p$-cocartesian}
if it is cocartesian for the map
\[
\Hom^{R}_{\cat C}\pr{X,Y}\to\Hom^{R}_{\cat D}\pr{p\pr X,p\pr Y}.
\]
A $2$-simplex of $X$ is said to be \textbf{$p$-cocartesian} if
its left degenerations are $p$-cocartesian. We denote the collection
of cocartesian triangles by $C_{X}$. 
\end{defn}

\begin{defn}
\cite[Definitions 4.24]{AS22_I}\label{def:O2C} Let $p\from\cat E\to\mathcal{B}$
be a weak $\mathbf{S}$-fibration, where $\cat B$ is an $\infty$-bicategory.
We say that $p$ is an\textbf{ outer $2$-cartesian fibration}, or
an \textbf{$\mathbf{O2C}$-fibration}, if it satisfies the following
conditions:

\begin{enumerate}[label=(O2C-\arabic*)]

\item The map $p$ is \textbf{locally fibered} in the following sense:
For each pair of objects $X,Y\in\cat E$, the functor
\[
\Hom^{R}_{\cat E}\pr{X,Y}\to\Hom^{R}_{\cat B}\pr{p\pr X,p\pr Y}
\]
is a cocartesian fibration of $\infty$-categories.

\item The map $p$ is \textbf{functorially fibered} in the following
sense: Let $0<i<3$, and let $\rho\from\Delta^{3}\to\cat E$ be a
simplex such that $\rho\vert\Delta^{\{i-1,i,i+1\}}$ is thin. If $\rho\vert\Lambda^{3}_{i}$
carries all triangles to $p$-cocartesian triangles, then $\rho\vert\Delta^{[3]\setminus\{i\}}$
is $p$-cocartesian. (Somewhat informally, this says cocartesian $2$-cells
are stable under pasting.)

\item Every degenerate edge of $\cat E$ is strongly $p$-cartesian.

\item The map $p$ has \textbf{enough cartesian morphisms} in the
following sense: Every morphism admits a $p$-cartesian lift with
a given target.

\end{enumerate}
\end{defn}

Having recalled the definition of $\mathbf{O2C}$-fibrations, we get
down to the proof of Proposition \ref{prop:fakeO2C}. We need a few
preliminary results.
\begin{lem}
\label{lem:bicat_fib}Let $p\from\mathcal{C}\to\mathcal{D}$ be a
weak $\mathbf{S}$-fibration of $\infty$-bicategories. The following
conditions are equivalent.
\begin{enumerate}
\item The map $p$ is a bicategorical fibration.
\item The map $p$ induces a categorical fibration of underlying $\infty$-categories.
Equivalently (by \cite[Corollary 2.4.6.5]{HTT} and \cite[Remark 3.1.5]{GC}),
for every object $C\in\cat C$ and every morphism $f\from p\pr C\xrightarrow{\simeq}D'$
in $\mathcal{D}$, there is an equivalence $\widetilde{f}\from C\xrightarrow{\simeq}C'$
of $\cat C$ lying over $f$.
\end{enumerate}
\end{lem}

\begin{proof}
The implication (2)$\implies$(1) follows from Remark \ref{rem:bicatfib},
because the inclusion $\{\varepsilon\}\subset J$ is a trivial cofibration
in the Joyal model structure for $\varepsilon\in\{0,1\}$.

For the converse, suppose condition (1) is satisfied. We must show
that condition (2) is satisfied. By Remark \ref{rem:bicatfib}, it
will suffice to prove the following: If $f\from\Delta^{1}\to\cat D$
is an equivalence in an $\infty$-category, then $f$ extends to a
map $J\to\cat D$. To see this, we note that our assumption ensures
that $f$ factors through the maximal sub Kan complex $\cat D^{\simeq}$.
So we may assume that $\cat D$ is a Kan complex. In this case, the
claim is trivial because $\Delta^{1}\hookrightarrow J$ is an anodyne
extension of simplicial sets (as it is a monomorphism and both $\Delta^{1}$
and $J$ are weakly contractible). Hence (1)$\implies$(2), as claimed.
\end{proof}

\begin{prop}
\label{prop:O2C}Let $\mathcal{B}$ be an $\infty$-bicategory, and
let $p\from\mathcal{E}\to\mathcal{B}$ be a weak $\mathbf{S}$-fibration.
If $p$ is an $\mathbf{O2C}$-fibration, then it is a bicategorical
fibration.
\end{prop}

\begin{proof}
By Lemma \ref{lem:bicat_fib}, it will suffice to show that $\Und\pr p$
is a categorical fibration. This is clear, because $\Und\pr p$ is
a cartesian fibration of $\infty$-categories.
\end{proof}

\begin{prop}
\label{prop:O2C-1}Let $\cat B$ be an $\infty$-category (with all
triangles scaled). A weak $\mathbf{S}$-fibration $p\from\mathcal{E}\to\mathcal{B}$
is an $\mathbf{O2C}$-fibration if and only if it satisfies the following
pair of conditions:
\begin{enumerate}
\item The map $p$ is a bicategorical fibration.
\item Every morphism of $S$ admits a $p$-cartesian lift with a given target.
\end{enumerate}
Moreover, a triangle of $\mathcal{E}$ is $p$-cocartesian if and
only if it is thin. 
\end{prop}

\begin{proof}
The last assertion is immediate from the definitions and the ``three
out of four'' property of thin triangles \cite[Remark 3.1.4]{GC}.
The ``only if'' part of the proposition follows from Proposition
\ref{prop:O2C}. For the ``if'' part, suppose that $p$ satisfies
conditions (1) and (2). We must show that $p$ satisfies conditions
(O2C-1) through (O2C-4) of Definition \ref{def:O2C}:

\begin{enumerate}[label=(O2C-\arabic*)]

\item \textit{For each pair of objects $X,Y\in\cat E$, the functor
\[
p_{X,Y}\from\Hom^{R}_{\cat E}\pr{X,Y}\to\Hom^{R}_{\cat B}\pr{p\pr X,p\pr Y}
\]
is a cocartesian fibration of $\infty$-categories.} We know that
$p_{X,Y}$ is an inner fibration (Remark \ref{rem:Hom_innfib}). Moreover,
$\Hom^{R}_{\cat B}\pr{p\pr X,p\pr Y}$ is a Kan complex because $\cat B$
is an $\infty$-category. Thus, our task is to show that every morphism
of $\Hom^{R}_{\cat B}\pr{p\pr X,p\pr Y}$ lifts to an equivalence
of $\Hom^{R}_{\cat E}\pr{X,Y}$ with a given source. This follows
from the definition of weak $\mathbf{S}$-fibrations.

\item \textit{The map $p$ is functorially fibered.} This follows
from the characterization of $p$-cocartesian triangles as thin triangles
and the ``three out of four'' property of thin triangles.

\item \textit{Every degenerate edge of $\cat E$ is strongly $p$-cartesian.}
Note that there is no distinction between strongly $p$-cartesian
edges and $p$-cartesian edges, because $p$-cocartesian triangles
are thin. Hence, it suffices to show that every degenerate edge of
$\cat E$ is $p$-cartesian. This is immediate from Remark \ref{rem:fakecart}.

\item \textit{The map $p$ has enough cartesian morphisms.} This
is part of the assumption.

\end{enumerate}

Thus we have shown that $p$ satisfies conditions (O2C-1) through
(O2C-4) of Definition \ref{def:O2C}, as required.
\end{proof}

We finally arrive at the proof of Proposition \ref{prop:fakeO2C}.
\begin{proof}
[Proof of Proposition \ref{prop:fakeO2C}]This follows from Remark
\ref{rem:fakecart} and Proposition \ref{prop:O2C-1}.
\end{proof}

\subsection{Endomorphism $\infty$-category}

A monoidal category can be regarded as a bicategory with a single
object. This perspective gives us a fully faithful left adjoint
\[
B\from\mathsf{MonCat}\to\mathsf{BiCat}_{[0]/}
\]
with right adjoint given by $\pr{\cat C,X}\mapsto\pr{\cat C\pr{X,X},\circ}$.
In this subsection, we record an $\infty$-categorical version of
this observation. 
\begin{prop}
\label{prop:deloop}There is an adjunction of $\infty$-categories
\[
\Mon\Cat_{\infty}\adj\pr{\BiCat_{\infty}}_{[0]/}
\]
whose left adjoint is fully faithful. The essential image consists
of essentially surjective functors $[0]\to\cat C$ of $\infty$-bicategories. 
\end{prop}

\begin{proof}
Let $\Alg\pr{\SS^{+}}$ denote the category of monoid objects in $\SS^{+}$.
By \cite[Proposition 4.1.8.3]{HA}, $\Alg\pr{\SS^{+}}$ has a model
structure whose weak equivalences and fibrations are detected by the
forgetful functor $\Alg\pr{\SS^{+}}\to\SS^{+}$. The inclusion $\Alg\pr{\SS^{+}}\hookrightarrow\pr{\mathsf{Cat}_{\SS^{+}}}_{[0]/}$
admits a right adjoint, which carries a pointed $\SS^{+}$-enriched
category $\pr{\cat C,X}$ to the mononid object $\cat C\pr{X,X}\in\Alg\pr{\SS^{+}}$.
The resulting adjunction

\[
\Alg\pr{\SS^{+}}\adj\pr{\mathsf{Cat}_{\SS^{+}}}_{[0]/}
\]
is a Quillen adjunction. Moreover, the underlying $\infty$-categories
of $\Alg\pr{\SS^{+}}$ and $\pr{\mathsf{Cat}_{\SS^{+}}}_{[0]/}$ can
be identified with $\Mon\Cat_{\infty}$ and $\pr{\BiCat_{\infty}}_{[0]/}$
by \cite[Theorem 4.1.8.4]{HA} and \cite[Corollary 7.6.13]{HCHA},
respectively. We thus get an induced adjunction
\[
\Mon\Cat_{\infty}\adj\pr{\BiCat_{\infty}}_{[0]/}.
\]
By inspection, the unit of this adjunction is an equivalence, so the
left adjoint is fully faithful. The claim follows.
\end{proof}

\begin{rem}
A version of Proposition \ref{prop:deloop} appears in \cite[Theorem 6.3.2]{GH15}.
\end{rem}

\begin{defn}
\label{def:end}Let $\cat C$ be an $\infty$-bicategory, and let
$X\in\cat C$ be its object. We write $\cEnd_{\cat C}\pr X^{\circ}$
for the image of $\pr{\cat C,X}\in\pr{\BiCat_{\infty}}_{[0]/}$ under
the functor $\pr{\BiCat_{\infty}}_{[0]/}\to\Mon\Cat_{\infty}$, and
call it the \textbf{endomorphism monoidal $\infty$-category} of $\cat C$
at $X$.
\end{defn}

\subsection*{Acknowledgment}

This work was supported by JSPS KAKENHI Grant Number 24KJ1443.

\providecommand{\bysame}{\leavevmode\hbox to3em{\hrulefill}\thinspace}
\providecommand{\MR}{\relax\ifhmode\unskip\space\fi MR }
\providecommand{\MRhref}[2]{%
  \href{http://www.ams.org/mathscinet-getitem?mr=#1}{#2}
}
\providecommand{\href}[2]{#2}

\end{document}